\documentclass[11pt]{article}
\usepackage{amsmath,amssymb,amscd,color}
\usepackage[notcite,notref]{showkeys}
\newcommand{\ep}{{ \epsilon  }}

\newcommand{\bq}{\begin{equation}}
\newcommand{\eq}{\end{equation}}
\newcommand{\pa}{\partial}
\newcommand{\R}{{ \mathbb{R}  }}

\newcommand{\bbr}{{ \mathbb{R}  }}

\newcommand{\calC}{{ \mathcal C  }}

\newcommand{\si}{\sigma}

\newcommand{\bke}[1]{\left( #1 \right)}

\newcommand{\bket}[1]{\left\{ #1 \right\}}
\newcommand{\norm}[1]{\left\Vert #1 \right\Vert}
\newcommand{\abs}[1]{\left| #1 \right|}

\newcommand{\Om}{{ \Omega  }}
\newcommand {\la}{\lambda}
\newcommand{\na}{\nabla}
\newcommand {\ga}{\gamma}
\newcommand {\al}{\alpha}
\newcommand {\be}{\beta}
\newcommand {\La}{\Lambda}

\newcommand{\qed}{\hfill\fbox{}\par\vspace{.2cm}}
\newcommand{\Rn}{{\mathbb R}^{n-1}}
\newcommand{\De}{\Delta}
\newcommand{\Ga}{\Gamma}
\newcommand{\te}{\theta}
\newcommand{\ri}{\rightarrow}
\newcommand{\de}{\delta}

\begin{document}
\bibliographystyle{plain}
%\english

% \newtheorem{thm}{Theorem}[section]
% \newtheorem{cor}[thm]{Corollary}
% \newtheorem{lem}[thm]{Lemma}
% \newtheorem{prop}[thm]{Proposition}
% \theoremstyle{definition}
% \newtheorem{defn}[thm]{Definition}
% \theoremstyle{remark}
% \newtheorem{rem}[thm]{Remark}

\newtheorem{defn}{Definition}
\newtheorem{lemma}[defn]{Lemma}
\newtheorem{proposition}{Proposition}
\newtheorem{theorem}[defn]{Theorem}
\newtheorem{cor}{Corollary}
\newtheorem{remark}{Remark}
\numberwithin{equation}{section}

\def\Xint#1{\mathchoice
   {\XXint\displaystyle\textstyle{#1}}%
   {\XXint\textstyle\scriptstyle{#1}}%
   {\XXint\scriptstyle\scriptscriptstyle{#1}}%
   {\XXint\scriptscriptstyle\scriptscriptstyle{#1}}%
   \!\int}
\def\XXint#1#2#3{{\setbox0=\hbox{$#1{#2#3}{\int}$}
     \vcenter{\hbox{$#2#3$}}\kern-.5\wd0}}
\def\ddashint{\Xint=}
\def\dashint{\Xint-}
\def\aint{\Xint\diagup}

\newenvironment{proof}{{\bf Proof.}}{\hfill\fbox{}\par\vspace{.2cm}}
\newenvironment{pfthm1}{{\par\noindent\bf
            Proof of Theorem \ref{Theorem1} }}{\hfill\fbox{}\par\vspace{.2cm}}
\newenvironment{pfprop1}{{\par\noindent\bf
            Proof of Proposition  \ref{time-decay} }}{\hfill\fbox{}\par\vspace{.2cm}}
\newenvironment{pfthm2}{{\par\noindent\bf
            Proof of Theorem  \ref{counter-exmaple} }}{\hfill\fbox{}\par\vspace{.2cm}}

\newenvironment{pfthm4}{{\par\noindent\bf
Sketch of proof of Theorem \ref{Theorem6}.
}}{\hfill\fbox{}\par\vspace{.2cm}}
\newenvironment{pfthm5}{{\par\noindent\bf
Proof of Theorem 5. }}{\hfill\fbox{}\par\vspace{.2cm}}
\newenvironment{pflemsregular}{{\par\noindent\bf
            Proof of Lemma \ref{sregular}. }}{\hfill\fbox{}\par\vspace{.2cm}}

\title{Solvability for Stokes system in H\"older spaces
in bounded domains and its applications}
\author{Tongkeun Chang and Kyungkeun Kang}

\date{}

\maketitle
\begin{abstract}
We consider Stokes system in bounded domains and we present
conditions of given data, in particular, boundary data, which ensure
H\"older continuity of solutions. For H\"older continuous solutions
for the Stokes system the normal component of boundary data requires
a bit more regular than boundary data of H\"older continuous
solutions for the heat equation. We also construct an example, which
shows that H\"older continuity is no longer valid, unless the
proposed condition of boundary data is fulfilled. As an application,
we consider a certain general types of nonlinear systems coupled to
fluid equations and local well-posedness is established in H\"older
spaces.
\\
\newline{\bf 2000 AMS Subject Classification}:  primary 35K61, secondary 76D07.
\newline {\bf Keywords}: Stokes system,  Initial-boundary value problem,  H\"older continuous.
\end{abstract}

%{2010 AMS Subject Classification}\,:\, 35K55, 75D05, 92B05

%{Keywords}\,:\, asymptotic behavior, Keller-Segel, Navier-Stokes
%equations

\section{Introduction}
 \setcounter{equation}{0}

In this paper, we study the initial and boundary value problem of
non-stationary Stokes system in bounded domains with $\calC^2$
boundary in $\R^n$, $n\ge 2$. To be more precise, we consider
\begin{equation}\label{CK-Aug6-10}
\partial_t u -\Delta u +\nabla p=\nabla \cdot {\mathcal F}+ f,\qquad
{\rm{div}}\, u=0 \qquad  \mbox{ in }\,\,Q_{T} :=\Omega\times [0,\,
T]
\end{equation}
with initial condition and boundary condition
\begin{equation}\label{CK-Aug6-20}
u(x,0)=u_0(x),\qquad u(x,t)=\phi(x, t) \,\,\mbox{ on
}\,\,\partial\Omega\times [0,T],
\end{equation}
where vector field $f$ and tensor ${\mathcal F}=(F_{ij})$ are given external
forces. Here we assume that the compatibility conditions hold
\begin{equation}\label{CK-Aug29-10}
{\rm div }\,u_0=0,\qquad u_0(x)=\phi(x,0)\,\,\mbox{ on
}\,\partial\Omega,\qquad \int_{\partial\Om}\phi(x,t)\cdot{\bf
n}=0\,\,\mbox{ for all }\,t\in [0,T],
\end{equation}
where ${\bf n}$ is the outward unit normal vector on $\partial\Om$.

Our main objective of this paper is to establish well-posedness of
the Stokes system \eqref{CK-Aug6-10}-\eqref{CK-Aug6-20} in the
H\"older spaces, ${\mathcal C}^{\al,\frac12 \al} (\overline Q_T)$.

 We recall
some known results related to our concerns. In case $u_0 \in {C}^{s}
({\mathbb R}^3_+), f\in {C}^{s-2,\frac{s}2 -1}({\mathbb R}^3_+\times
(0,T)), \, {\mathcal F} =0$ and $ \phi \in C^{s,\frac{s}2 }({\mathbb R}^2
\times (0,T))$  for $2<s<3$, Solonnikov  showed in \cite{sol2} that
a unique solution of the Stokes system \eqref{CK-Aug6-10}-\eqref{CK-Aug6-20} exists so that
\begin{align*}
  \| u\|_{{\dot C}^{s,\frac{s}2 }({\mathbb R}^3_+\times (0,T))}
 &\leq c\Big( \|u_0\|_{\dot{C}^{s}({\mathbb R}^3_+)}+\|\phi\|_{\dot {C}^{s,\frac{s}2}({\mathbb R}^2 \times (0,T))}
  \\
  &\qquad+ \| R'(D_t \phi_{3}) \|_{L^\infty({\mathbb R}^2 ; \dot {C}^{\frac{s}2} (0,T)) }
  +\|f\|_{{\dot C}^{s-2,\frac{s}2-1}({\mathbb R}^3_+\times (0,T))}
  \Big),
\end{align*}
where $R'$ is $\Rn$-dimensional Riesz transform. Under a weaker
assumption on $\phi$ than that of \cite{sol2}, the first author and
Jin \cite{CJ} proved that in case that $f =0$, the following
estimate holds: for $0 < \al <1$
\begin{align*}
 \| u\|_{{C}^{\al,\frac{\al}2 }({\mathbb R}^n_+\times (0,T))}
 &\leq c\Big(\|u_0\|_{{C}^{\alpha}({\mathbb R}^n_+)}+  \|R'{u_0}_n  \|_{ {C}^{\alpha}({\mathbb R}^n_+)}
  +\|\phi\|_{{C}^{\alpha,\frac{\al}2  }(\Rn\times (0,T))}\nonumber\\
  &\quad +\|R'\phi_n\|_{{C}^{\alpha,\frac{\al}2  }(\Rn\times (0,T))}
  +\max\{T^\frac12, T^{\frac{1}{2}+\frac{\al}{2}}\}
  \| {\mathcal F}\|_{C^{\al,\frac{\al}{2}}({\mathbb R}^n_+\times (0,T)) }
  \Big).%\label{CK-Oct30-300}
 \end{align*}
When $\Om$ is a bounded domain with $C^{2 + \al}$, $\al>0$,
Solonnikov\cite{So3,So2} showed that if $f=0$, ${\mathcal F} =0$ and $\phi \in
C(\pa \Om \times (0,T))$ with $\phi \cdot {\bf n} =0$, then the
solution $u$ of \eqref{CK-Aug6-10}-\eqref{CK-Aug6-20} is continuous in $\Om \times (0,T)$
such that
\begin{align}\label{CK-Oct30-310}
\|u\|_{L^\infty(\Om \times (0,T))}  \leq c  \| \phi\|_{L^\infty(\pa
\Om \times (0,T))}.
\end{align}
The estimate \eqref{CK-Oct30-310} was improved by the first author
and Choe as following inequality (see \cite{CC})
\begin{align*}
 \|u\|_{L^\infty(\Om \times (0,T))}  \leq  c \big( \| \phi\|_{L^\infty(\pa \Om \times (0,T))}
 +\max_{ t\in (0,T)}||\phi \cdot {\bf n}(\cdot, t)||_{Dini,\partial\Omega} \big),
\end{align*}
where for an $r_0>0$
$$
|| f||_{Dini,\pa \Omega} =\sup_{P\in \pa \Omega} \int_{0}^{r_0}
\omega(f)(r,P) \frac{dr}{r},\qquad \omega(f)(r,P)=\sup_{Q\in
B_r(P)\cap \Omega} |f(Q)-f(P)|.
$$
%with $\omega(f)(r,P)=\sup_{Q\in B_r(P)\cap \Omega} |f(Q)-f(P)|$ for an $r_0>0$.
There are various literatures for the solvability of the  Stokes
system  \eqref{CK-Aug6-10}-\eqref{CK-Aug6-20}  with homogeneous boundary data, that is, with
$\phi=0$ (see e.g. \cite{shimizu,maremonti,maremonti3,sol1,sol3},
and references therein).
In particular, % Y.Giga, S.Matsui, and Y.Shimizu
 the following estimate is derived in  \cite{sol3}:
%that
\begin{align*}
%\label{known2}
\| u\|_{L^\infty({\mathbb R}^n_+ \times  (0,T))} \leq
c\Big( \|u_0\|_{L^\infty({\mathbb
R}^{n}_+)}+T^{\frac{1}{2}}\|{\mathcal F}\|_{L^\infty({\mathbb R}^n_+
\times (0,T))} \Big),
\end{align*}
where $u_0 \in C(\R^n_+)$ and $f =0, \, {\mathcal
F}=(F_{ij})_{i,j=1}^n\in C(\R^n_+\times (0,T))$ with $ \mbox{div }
u_0=0$,  $\ u_0 |_{x_n=0}=0$,  $F_{nj}|_{x_n=0}=0$ for
$j=1,2,\cdots, n$ (see also \cite{shimizu}).

We compare the system \eqref{CK-Aug6-10}-\eqref{CK-Aug29-10} to
similar situation of the heat equation
\begin{equation}\label{CK-Aug6-30}
\partial_t v -\Delta v =\nabla \cdot {\mathcal F}+ f\qquad  \mbox{ in }\,\,Q_{T} :=\Omega\times [0,\,
T]
\end{equation}
with initial condition and boundary conditions
\begin{equation}\label{CK-Aug6-40}
v(x,0)=v_0(x) \,\,\mbox{ in
}\,\,\Om,  \qquad v(x,t)=\phi(x, t) \,\,\mbox{ on
}\,\,\partial\Omega\times [0,T].
\end{equation}
If we assume that
\begin{equation}\label{CK-Aug28-10}
v_0\in \calC^{\al+1}(\overline{\Omega} ),\quad \phi\in\calC^{\al,
\frac12 \al}( \partial\Om \times [0, T]),\quad f\in L^\infty(\Om
\times (0, T)),\quad {\mathcal F} \in {\mathcal C}^{\al, \frac12 \al
}(\overline{\Omega}\times [0, T]),
\end{equation}
we then obtain the following estimate:
\begin{align}\label{CK-Aug6-50}
\notag \| v\|_{\calC^{\al, \frac12 \al }(\overline{\Omega} \times
[0, T])} &\leq c \big(\| v_0\|_{\calC^{\al}(\overline{\Omega} )}+\|
\phi\|_{\calC^{\al, \frac12 \al}( \partial\Om \times [0, T])  }\\
& \quad  +T^{1-\frac12 \al} \| f\|_{L^\infty(\Om \times (0, T))}+
T^{\frac12} \| {\mathcal F}\|_{\calC^{\al, \frac12 \al }(\overline{\Omega}\times
[0, T])} \big).
\end{align}

%{\tt...Is $T^{1-\frac12 \al} \| f\|_{L^\infty(\R^n \times (0, T))}$
%correct in the above?....}

The estimate \eqref{CK-Aug6-50} is probably known to experts, but we
show it for clarity in section 2. In fact, we prove more than the
above estimate (see Theorem \ref{heat-domain-100} for the details).
Definition of H\"older spaces ${\mathcal
C}^{\al+1}(\overline{\Omega} )$, ${\mathcal C}^{\al, \frac12 \al
}(\overline{\Omega}\times [0, T])$ and $\calC^{\al, \frac12 \al}(
\partial\Om \times [0, T])$ are given in section 2.

\begin{remark}
In case that the Dirichlet boundary condition in \eqref{CK-Aug6-40},
$v=\phi$ on $\partial\Om \times (0, T)$, is replaced by the Neumann condition
$\frac{\partial v}{\partial {\bf n}}=\psi$ on $\partial\Om$, if $
\psi\in\calC^{1+\al, \frac12 +\frac12 \al}( \partial\Om \times [0,
T])$ is assumed, then the same estimate as \eqref{CK-Aug6-50} can be
valid.
\end{remark}

Because of non-local effect for the Stokes system, the estimate
\eqref{CK-Aug6-50} is not clear.  If we assume, however, further
additional assumptions for $u_0$ and $\phi$, then the H\"older
estimate is available. To be more precisely, if we assume, instead
of \eqref{CK-Aug28-10}, that
\begin{equation}\label{CK-Aug6-60}
u_0\in \calC_{{{\mathcal D}_{\eta}}}(\overline{\Omega}),\qquad
\phi\in \calC^{\al,\frac12 \al} (\partial\Omega\times[0,T]),\qquad
\phi\cdot {\bf n}\in \dot \calC_{{\mathcal D}_{\eta}} (\pa \Om;\dot
\calC^{\frac12 \al} [0, T]),
\end{equation}
\begin{equation}\label{CK-Aug28-20}
f \in %L^\infty(\Om \times [0, T]) \cap
L^\infty(0, T; \calC_{D_\eta}(\overline \Om)),\qquad {\mathcal F}\in  \calC^{\al,
\frac12 \al}( \overline \Om \times [0, T]),
\end{equation}
then the similar estimate as \eqref{CK-Aug6-50} can be obtained. The
details of function spaces $\calC_{{{\mathcal
D}_{\eta}}}(\overline{\Omega})$, $\calC^{\al,
\frac12 \al}( \overline \Om \times [0, T])$ and $\dot \calC_{{\mathcal
D}_{\eta}} (\pa \Om;\dot \calC^{\frac12 \al} [0, T])$ are also given
in section 2.

Our first main result reads as follows:

\begin{theorem}\label{mainthm-Stokes}
Let $0 < \al<1 $. Let $ u_0, \phi, f, {\mathcal F}$ satisfy the
conditions \eqref{CK-Aug6-60}-\eqref{CK-Aug28-20}.
%and $\int_{\pa \Om}\phi(Q,t) {\bf n}(Q) dQ =0$.
Furthermore, $u_0$ and $\phi$ satisfy the compatibility condition
\eqref{CK-Aug29-10}.
%$\phi|_{t =0} = u_0$ on $\pa \Om$.
Then, there exists a unique weak
solution $u $ of the Stokes equations of
\eqref{CK-Aug6-10}-\eqref{CK-Aug6-20} in the class $
\calC^{\al,\frac12 \al}(\overline{\Om} \times [0, T])$ such that
\begin{align*}
&\| u\|_{{\mathcal C}^{\al,\frac12 \al} (\overline{\Om} \times [0,
T])} \leq c\big( \| \phi\|_{{\mathcal C}^{\al, \frac12 \al}(\pa \Om
\times [0, T])} +  \|  \phi \cdot {\bf n}\|_{\dot {\mathcal
C}_{{\mathcal D}_\eta} (\pa \Omega; \dot {\mathcal
C}^{\frac{\al}2}[0, T])}  +\| u_0\|_{{\mathcal
C}^{\al}_{{\mathcal D}_\eta} (\overline{\Om})}\\
&\qquad +  \max ( T^{\frac12 }, T^{\frac12 -\frac12 \al}  )  \|
{\mathcal F}\|_{  \calC^{\al, \frac12 \al}(\overline{\Om} \times [0, T]) } +
\max (T, T^{\frac12 -\frac12 \al} )  \|
 f \|_{ L^\infty((0, T);  \mathcal C_{{\mathcal D}_\eta} (\overline{\Om}
 ))}\big).
\end{align*}
\end{theorem}

The notion of weak solution of the Stokes system in the class
$\calC^{\al,\frac12 \al}(\overline{\Om} \times [0, T])$ is given in
section 3 (see Definition \ref{defn-Stokes}).

We note that $\phi\in\calC^{\al, \frac12 \al}( \partial\Om \times
[0, T])$ implies the condition \eqref{CK-Aug6-60} with replacement
of $\alpha$ by $\beta$ for any $0<\beta<\alpha$, since
\begin{align*}%\label{CK-aug28-300}
\| \phi\|_{{\mathcal C}^{\beta, \frac12 \beta}(\pa \Om \times [0,
T])} +  \|  \phi \cdot {\bf n}\|_{\dot {\mathcal C}_{{\mathcal
D}_\eta} (\pa \Omega; \dot {\mathcal C}^{\frac{\beta}2}[0, T])}\le
c\| \phi\|_{{\mathcal C}^{\al, \frac12 \al}(\pa \Om \times [0, T])}.
\end{align*}
Similarly, we also note that
\begin{align*}%\label{CK-Nov1-10}
\| u_0\|_{{\mathcal C}^{\beta}_{{\mathcal D}_\eta}
(\overline{\Om})}\le c\| u_0\|_{{\mathcal
C}^{\alpha}(\overline{\Om})},\qquad \| {\mathcal F}\|_{  \calC^{\beta, \frac12
\beta}(\overline{\Om} \times [0, T]) }\le c\|{\mathcal F}\|_{  \calC^{\al,
\frac12 \al}(\overline{\Om} \times [0, T]) },
\end{align*}
\begin{align*}%\label{CK-aug28-310}
\| f \|_{ L^\infty(0, T;  \mathcal C^{\beta}_{{\mathcal D}_\eta}
(\overline{\Om} ))}\le c\norm{f}_{\calC^{\al, \frac12 \al}(
\overline{\Om} \times (0, T)) }.
\end{align*}
Therefore, a direct consequence of Theorem \ref{mainthm-Stokes} is
the following:

\begin{cor}\label{corollary-Stokes}
Let $\alpha\in(0,1)$. Suppose that $u_0\in
\calC^{\al}(\overline{\Omega} )$, $\phi\in \calC^{\al,\frac12 \al}
(\partial\Omega\times[0,T])$, and $u_0$ and $\phi$ satisfy the
compatibility condition \eqref{CK-Aug29-10}. Assume further that $f,
{\mathcal F} \in \calC^{\al,\frac12 \al}
(\overline{\Omega}\times[0,T])$. Then, there exists unique weak
solution $u$ of \eqref{CK-Aug6-10}-\eqref{CK-Aug6-20} in the class
$\calC^{\beta,\frac12 \beta} (\overline \Om \times [0,T])$ for any
$\beta<\al$. Furthermore, $u$ satisfies
\begin{align*}
\norm{u}_{\calC^{\beta,\frac12 \beta} (\overline \Om \times [0,T])} & \le c \big(
\|u_0\|_{\calC^{\al}(\overline{\Omega} )} + \| \phi\|_{\calC^{\al,
\frac12 \al}( \partial\Om \times [0, T])  } +  \max ( T^{\frac12 }, T^{\frac12 -\frac12 \be}  )
\norm{{\mathcal F}}_{\calC^{\al, \frac12 \al}( \overline \Om \times [0, T]) }\\
%\label{CK-Aug6-70}
& \qquad
+ \max (T, T^{\frac12 -\frac12 \be} ) \norm{f}_{\calC^{\al,
\frac12 \al}(\overline \Om \times [0, T]) } \big).
\end{align*}
\end{cor}

Next, we show that $\beta$ cannot be extended to $\al$ in Corollary
\ref{corollary-Stokes}. To be more precise, there exists a boundary
data $\phi\in \calC^{\al,\frac12 \al} (\partial\Omega\times[0,T])$
such that $u\notin \calC^{\al,\frac12 \al} (\overline \Om \times [0,T])$, even
if $u_0$, $f$ and $F$ are smooth. This implies that the result in
Theorem \ref{mainthm-Stokes} seems optimal. Our second result is to
construct a solution $u\notin \calC^{\al,\frac12 \al}
(\overline \Om \times [0,T])$ of \eqref{CK-Aug6-10}-\eqref{CK-Aug29-10} when
$\phi\in\calC^{\al, \frac12 \al}(\pa \Om \times [0, T])$.

\begin{theorem}\label{counter-exmaple}
Theorem \ref{mainthm-Stokes} is not true, if $\phi$ is assumed to
belong to $\calC^{\al, \frac12 \al}(\pa \Om \times [0, T])$ only.
\end{theorem}

As an application, we consider nonlinear types of drift equations
coupled fluid equations. Let $\rho:\Omega\times [0, T]\rightarrow
\R$, $\theta:\Omega\times [0, T]\rightarrow \R$ and $u:\Omega\times
[0, T]\rightarrow \R^n$ satisfy
\begin{equation}\label{CK-Aug18-9}
\partial_t \rho + u \cdot \nabla  \rho- \Delta \rho= \nabla\cdot F(\rho, \theta, \nabla \theta,
u),
\end{equation}
\begin{equation}\label{CK-Aug18-10}
\partial_t \theta + u \cdot \nabla  \theta- \Delta \theta= f(\rho, \theta, \nabla \theta,
u),
\end{equation}
\begin{equation}\label{CK-Aug18-20}
\partial_t u + u\cdot \nabla u -\Delta u +\nabla p=G(\rho, \theta, \nabla \theta,
u),
%+\nabla\cdot T(\theta, \nabla \theta, u, \nabla u),
\qquad {\rm div}\, u=0
\end{equation}
with initial data  $\rho_0$, $\theta_0$ and $u_0$.
Here %$g:\R\times\R^n\times \R^n\times \R^{n^2}\rightarrow\R$,
$f:\R\times\R\times\R^n\times \R^n\rightarrow\R$ and $F,
G:\R\times\R\times\R^n\times \R^n\rightarrow\R^n$ %and $T:\R\times\R^n\times\R^n\times \R^{n^2}\rightarrow\R^{n^2}$
are $\calC^1$ scalar and vector valued functions with polynomial
growth conditions. To be more precise, we assume that for $(x, y, z,
w)\in \R\times\R\times\R^n\times \R^n$ there exists an integer $l$
with $1\le l<\infty$ such that $f$, $F$, and $G$ satisfy
\begin{equation}\label{CK-Aug17-10}
\abs{f(x, y, z, w)}+\abs{F(x, y, z, w)}+\abs{G(x, y, z, w)}%+\abs{T(x, y, z)}
\le C\bke{1+\abs{x}+\abs{y}+\abs{z}+\abs{w}}^l,
\end{equation}
\begin{equation}\label{CK-Aug17-10-5}
\abs{\nabla f(x, y, z, w)}+\abs{\nabla F(x, y, z, w)}+\abs{\nabla
G(x, y,
z, w)}%+\abs{\nabla T(x, y, z)}
\le C\bke{1+\abs{x}+\abs{y}+\abs{z}+\abs{w}}^{l-1}.
\end{equation}

Under our consideration, no-flux boundary conditions are assigned
for $\rho$ and $\theta$ and no-slip boundary condition of $u$ is
assumed, namely
\begin{equation}\label{CK-Aug17-11}
\frac{\partial\rho}{\pa {\bf n}}=0,\qquad\frac{\partial\theta}{\pa
{\bf n}}=0,\qquad u=0\qquad \mbox{ on }\,\,\partial\Omega.
\end{equation}

For nonlinear system \eqref{CK-Aug18-9}-\eqref{CK-Aug17-11}, with
the aid of results in Theorem \ref{mainthm-Stokes}, we can also
establish local well-posedness in the H\"older spaces. Our last main
result reads as follows:

\begin{theorem}\label{Theorem1}
Let the initial data $(\rho_0, \theta_0, u_0)$ be given in
$\calC^{\alpha}(\overline{\Omega})\times\calC^{\alpha+1}(\overline{\Omega})\times
\calC^{\al}_{{{\mathcal D}_{\eta}}}(\overline{\Omega})$ for
$\alpha\in(0,1)$. Assume that $F$, $G$ and $f$ satisfy the
assumption \eqref{CK-Aug17-10}-\eqref{CK-Aug17-10-5}. There exists $T_1>0 $ such that a
pair of unique solution $(\rho, \theta, u)$ for
\eqref{CK-Aug18-9}-\eqref{CK-Aug18-20} with \eqref{CK-Aug17-11} can be constructed in the
class $\calC^{\al,\frac12 \al} ( \overline \Om \times [0,T_1])\times
\calC^{\al+1,\frac12 \al+\frac12 } ( \overline \Om \times [0,T_1])\times
\calC^{\al,\frac12 \al} (\overline \Om \times [0,T_1])$.
\end{theorem}

\begin{remark}
The result of Theorem \ref{Theorem1} could be applicable to various
types of concrete equations involving fluid motions. For an specific
example, the Keller-Segel-Navier-Stokes equations, a mathematical
model describing the dynamics of a certain bacteria living in fluid
and consuming oxygen, can be considered. For such model we can
establish local well-posedness in the H\"older space as a
consequence of Theorem \ref{Theorem1} (see section 4 for more
details).
\end{remark}

This paper is organized as follows.
%In Section 2, preliminary works are introduced and mentioned.
In Section 2, %some function spaces are introduced and
H\"older estimates of solutions for the heat equations are computed.
Section 3, 4 and 5 are devoted to providing the proofs of Theorem
\ref{mainthm-Stokes}, Theorem \ref{counter-exmaple} and Theorem
\ref{Theorem1}, respectively. Some technical lemmas are proved in
Appendix.

%%%%%%%%%%%%%%%%%%%%%%%%%%%%%%%%%%%%%%%%%%%%%%%%%%%%%%%%%%%%%%%%%%%%%%%%%%%%%%%%%
%%%%%%%%%%%%%%%%%%%%%%%%%%%%%%%%%%%%%%%%%%%%%%%%%%%%%%%%%%%%%%%%%%%%%%%%%%%%%%%%%
%%%%%%%%%%%%%%%%%%%%%%%%%%%%%%%%%%%%%%%%%%%%%%%%%%%%%%%%%%%%%%%%%%%%%%%%%%%%%%%%%

\section{Preliminaries}
\setcounter{equation}{0}

%We denote weak $L^q$ space by $L^q_{w}(\R^n)$, which is defined as
%the space of all measurable functions $f$ satisfying
%\[
%\sup_{\alpha>0} \alpha\abs{\{x:
%\abs{f(x)}>\alpha\}}^{\frac{1}{q}}<\infty,
%\]
%with the norm
%\[
%\norm{f}_{L^{q}_{w}}=\sup_A \abs{A}^{-\frac{1}{q'}}\int_A
%\abs{f(x)}dx.
%\]
% We recall the integral Young's inequality (see e.g. \cite{LR}).
%\begin{equation}\label{CKL-Jan12-10}
%\norm{f*g}_{L^r}\leq C\norm{f}_p\norm{g}_{L^{q}}, \qquad
%   \frac{1}{r}+1=\frac{1}{p}+\frac{1}{q},\qquad 1\le p,q,r\le\infty.
%\end{equation}

%\section{Estimates of heat equations}

We first introduce the notation and present preparatory results that
are useful to our analysis. We start with the notation. Let $\Omega$
be an open domain in $\R^n$. %{\color{red}{ For $1\leq q\leq \infty$, we denote by
%$W^{k,q}(\Omega)$ the usual Sobolev spaces, namely
%$W^{k,q}(\Omega)=\{f\in L^q(\Omega): D^{\alpha}f\in L^q(\Omega),
%0\leq \abs{\alpha}\leq k\}$. }}
The letter $c$ is used to represent a
generic constant, which may change from line to line, and
$c(*,\cdots,*)$ is considered a positive constant depending on
$*,\cdots,*$.  We introduce
a homogeneous H\"older space in $\Om$ with exponent $\alpha\in
(0,1)$, denoted by $\dot{\mathcal C}^{\al}(\overline{\Omega})$,
defined by
\[
\dot{\mathcal C}^{\al}(\overline{\Omega}):=\{f\in L^1(\Omega):
\norm{f}_{\dot{\mathcal
C}^{\al}(\overline{\Omega})}:=\sup_{x,y\in\overline{\Omega}}
\frac{|f (x) - f(y)|}{|x-y|^\al}<\infty\}.
\]
Usual H\"older space with exponent $\alpha\in (0,1)$, denoted by
${\mathcal C}^{\al}(\overline{\Omega})$, is specified as
\[
\calC^{\al} (\overline \Om):=\bket{f\in L^1(\Omega): \norm{f}_{{\mathcal
C}^{\al}(\overline{\Omega})}:=\norm{f}_{L^{\infty}(\Omega)}+\norm{f}_{\dot{\mathcal
C}^{\al}(\overline{\Omega})}<\infty}.
\]
Furthermore, we introduce following function classes
\[
\dot{\calC}^{\alpha}_{{{\mathcal
D}_{\eta}}}(\overline{\Omega})=\{f\in\dot{\calC}^{\alpha}(\overline{\Omega}):
\norm{f}_{\dot{\calC}^{\alpha}_{{{\mathcal
D}_{\eta}}}(\overline{\Omega})}:=\sup_{x,y\in
\overline{\Omega}}\frac{\abs{f(x)-f(y)}}{\abs{x-y}^{\alpha}\eta(\abs{x-y})}<\infty
\},
\]
where $\eta:\R^+\rightarrow\R^+$ is increasing Dini continuous,
namely $\int_0^1 \frac{\eta(r)}{r}dr<\infty $. Similarly, we define
$\calC^{\alpha}_{{{\mathcal D}_{\eta}}}(\overline{\Omega})$ by
\[
\calC^{\alpha}_{{{\mathcal
D}_{\eta}}}(\overline{\Omega})=\{f\in\calC^{\alpha}(\overline{\Omega}):\norm{f}_{\calC^{\alpha}_{{{\mathcal
D}_{\eta}}}(\overline{\Omega})}:=
\norm{f}_{\calC^{\alpha}(\overline \Omega)}+\norm{f}_{\dot{\calC}^{\alpha}_{{{\mathcal
D}_{\eta}}}(\overline{\Omega})}<\infty\}.
\]
%We also define a class of functions, denoted by
%$\dot{\calC}^{\alpha}_{{{\mathcal D}_{\eta}}}(\pa \Om;
%\overline{\Omega})$, than $\dot{\calC}^{\alpha}_{{{\mathcal
%D}_{\eta}}}(\overline{\Omega})$ as follows:
%\[
%\dot{\calC}^{\alpha}_{{{\mathcal D}_{\eta}}}(\pa \Om;
%\overline{\Omega})=\{f\in\dot{\calC}^{\alpha}(\overline{\Omega}):
%\norm{f}_{\dot{\calC}^{\alpha}_{{{\mathcal D}_{\eta}}}(\pa \Om  ;
%\overline{\Omega} )}:=\sup_{P, Q \in \pa \Om,z\in
%\overline{\Omega}}\frac{\abs{f(P-z)-f(P) -f(Q-z) +
%f(Q)}}{|z|^{\alpha}\eta(\abs{P -Q})}<\infty \}.
%\]
%One can easily see that $\dot{\calC}^{\alpha}_{{{\mathcal
%D}_{\eta}}}(\pa \Om; \overline{\Omega})$ is wider than
%$\dot{\calC}^{\alpha}_{{{\mathcal D}_{\eta}}}(\overline{\Omega})$,
%namely $\dot{\calC}^{\alpha}_{{{\mathcal
%D}_{\eta}}}(\overline{\Omega})\subset\dot{\calC}^{\alpha}_{{{\mathcal
%D}_{\eta}}}(\pa \Om; \overline{\Omega})$, which are useful for our
%purpose.

In case of non-stationary function $f\in L^1(\Om\times (0,T))$, we
recall a seminorm of $f$, which is H\"older continuous with exponent
$\alpha\in (0,1)$ in spatial and temporal variable, denoted by
$\norm{f}_{\dot {\mathcal C}^{\al,\frac12 \al}(\overline{\Om} \times [0, T])}$,
indicated as follows:
\[
%\dot C^{\al,\frac12 \al}(\overline{\Om} \times [0, T]):=
\norm{f}_{\dot {\mathcal C}^{\al,\frac12 \al}(\overline{\Om} \times [0, T])}:=
\|f\|_{L^\infty ((0, T); \dot {\mathcal C}^{\al}(\overline{\Om}))}  +
\|f\|_{L^\infty ( \Om; \dot {\mathcal C}^{\frac12 \al}([0, T])}
\]
\[
= \sup_t \sup_{x,y}  \frac{|f (x,t) - f(y,t)|}{|x-y|^\al} + \sup_x
\sup_{t,s} \frac{|f(x,t) - f(x,s)|}{|t-s|^{\frac12 \al}}.
\]
We also remind an H\"older space with exponent $\alpha\in (0,1)$ in
$\Omega\times (0,T)$, written as $\calC^{\al, \frac12 \al}
(\overline{\Om} \times [0, T])$, which is given by
\begin{align*}
\calC^{\al, \frac12 \al} (\overline{\Om} \times [0, T]):=\bket{f\in
L^1\,:\,\| f\|_{\calC^{\al,\frac12 \al}(\overline{\Om} \times [0, T])}:
= \| f\|_{L^\infty(\Om \times (0, T))} + \norm{f}_{\dot
\calC^{\al,\frac12 \al}(\overline{\Om} \times [0, T])}<\infty}.
\end{align*}
%where
%\begin{align*}
%\|g \|_{\dot C^{\al,\frac12 \al}(\Om \times (0, T))}
% &= \|g \|_{L^\infty (0, T; \dot C^{\al}(\Om))}  + \|g \|_{L^\infty (\Om; \dot C^{\frac12 \al}((0, T))}\\
%& = \sup_t \sup_{x,y}  \frac{|g (x,t) - g(y,t)|}{|x-y|^\al} + \sup_x
%\sup_{t,s} \frac{|g(x,t) - g(x,s)|}{|t-s|^{\frac12 \al}}.
%\end{align*}

Let $\eta$ be a increasing Dini-function defined above. To treat
non-zero boundary data under our considerations, we also introduce a
function class $ \dot C_{{\mathcal D}_{\eta}} (\pa \Om;\dot
\calC^{\frac12 \al} [0, T])$ defined by
\begin{align*}
\dot \calC_{{\mathcal D}_{\eta}} (\pa \Om;\dot \calC^{\frac12 \al} [0, T]):
&= \bket{f \, | \, \sup_{P, Q \in
\pa \Om}\frac{\norm{f(P,\cdot)-f(Q,\cdot)}_{\dot{\mathcal
C}^{\frac{1}{2}\alpha}[0,T]}}{\eta(\abs{P-Q})} < \infty },\quad \al
\in (0, 1)
\end{align*}
equipped with the norm $\norm{f}_{\dot \calC_{{\mathcal D}_{\eta}} (\pa
\Om;\dot \calC^{\frac12 \al} [0, T])}:=\sup_{P, Q \in \pa
\Om}\frac{\norm{f(P,\cdot)-f(Q,\cdot)}_{\dot{\mathcal
C}^{\frac{1}{2}\alpha}[0,T]}}{\eta(\abs{P -Q})}$, which is
equivalently
as %Here we mean $\norm{f}_{\dot C_{{\mathcal D}_{\eta}} (\pa
%\Om;\dot C^{\frac12 \al} [0, T])}$ by
\[
\norm{f}_{\dot \calC_{{\mathcal D}_{\eta}} (\pa \Om;\dot \calC^{\frac12 \al}
[0, T])}
%\sup_{P, Q \in\pa\Om}\frac{\norm{f(P,\cdot)-f(Q,\cdot)}_{\mathcal
%C^{\frac{1}{2}\alpha}[0,T)}}{\eta(\abs{P-Q})}
=\sup_{P,Q \in \pa \Om}\sup_{ s,t \in [0, T]} \frac{|f(P,t) -f(P,s)
- f(Q,t) + f(Q,s)|}{|t-s|^{\frac12 \al} \eta(|P-Q|)}.
\]
%Restricting $x, y$ on the boundary, we denote
%\[
%\dot C_{{\mathcal D}_{\eta}} (\pa \Om;\dot C^{\frac12 \al} [0, T]):=
%\bket{f \, | \, \sup_{P, Q \in
%\pa\Om}\frac{\norm{f(P,\cdot)-f(Q,\cdot)}_{\dot{\mathcal
%C}^{\frac{1}{2}\alpha}[0,T)}}{\eta(\abs{x-y})} < \infty },\quad \al
%\in (0, 1)
%\]
For our purpose, as a limiting case of $\alpha=0$, we introduce
\[
L^\infty(0, T; \dot \calC_{D_\eta}(\overline \Om)): = \{ f  \, | \,
\sup_{t} \sup_{x,y\in\bar\Om} \frac{|f(x,t) - f(y,t)|}{\eta(|x-y|)}
< \infty \}.
\]

We recall some estimates of heat equations in following lemmas.

\begin{lemma}\label{proheat1}
Let $\al\in (0,\infty)$, $0<T< \infty$ and $u_0:\R^n \rightarrow\R^n$ be
a vector field, which belongs to ${\mathcal C}^\alpha_{{\mathcal
D}_\eta}(\R^n )$. We set $W(x,t): = \int_{\R^n} \Ga(x-z,t) u_0(z)
dz$, where $\Gamma$ is the heat kernel. Then, $W\in \calC^{\alpha,
\frac{\al}2 }(\R^n\times [0, T])$ and $W$ satisfies
\begin{equation}\label{CK-Aug5-10}
\|W\|_{ \calC^{\al,\frac{\al}2  }(\R^n\times [0, T])}\leq
c\|u_0\|_{ \calC^{\alpha}(\R^n)}.%,\qquad \|W\|_{L^\infty(\R^n\times
%(0, T))}\leq c\|u_0\|_{L^\infty(\R^n)}.
\end{equation}
Furthermore, if $\al \in (0, 1) $ and  $\Omega\subset\R^n$ be a bounded domain with smooth
boundary, then
%\begin{equation}\label{CK-June17-310}
%\|  W  \|_{  C^{\alpha,\frac{\al}2}( \R^n \times [0, T])} \leq c
% \| u_0\|_{  C^{\al}  (\R^n)},
%\end{equation}
\begin{equation}\label{CK-June17-310-1}
 \|  W \cdot {\bf n}\|_{\dot {\mathcal C}_{{\mathcal D}_\eta} (\pa \Omega; \dot {\mathcal C}^{\frac{\al}2}[0, T])}
 \leq c   \| u_0\|_{ \calC^{\al}_{{\mathcal D}_\eta} (\R^n)}.
\end{equation}

\end{lemma}
\begin{proof}
Since the estimates in \eqref{CK-Aug5-10} are well-known, we omit
its details (see e.g. \cite{Lady}) and we just show the estimate
 \eqref{CK-June17-310-1}. Indeed, %noting that
%$W \in C^{\al, \frac12 \al}(\pa \Om \times [0, T])$  via
%\eqref{CK-June17-310-1} and
using $u_0 \in  {\mathcal
C}^\alpha_{{\mathcal D}_\eta}(\R^n )$, we compute for
$P,Q\in\partial\Om$
\[
W (P,t) \cdot{\bf n}(P) - W(P,s) \cdot{\bf n}(P) - W(Q,t) \cdot{\bf
n}(Q) + W(Q,s)\cdot{\bf n}(Q)
\]
\[
 = \int_{\R^n} \Big(\Ga(z,t) -\Ga(z,s) \Big) \Big(   u_0(P -z) \cdot{\bf n}(P) -   u_0(Q-z) \cdot {\bf n}(Q) \big) dz
\]
\[
= \int_{\R^n} \Big(\Ga(z,t) -\Ga(z,s) \Big) \Big(   u_0(P -z)
\cdot{\bf n}(P) -   u_0(P ) \cdot{\bf n}(P) -   u_0(Q-z) \cdot {\bf
n}(Q) +   u_0(Q) \cdot{\bf n}(Q) \Big) dz.
\]
For the second equality, we used $\int_{\R^n}(\Ga(z,t) -
\Ga(z,s) \big) dz =0 $ for all $0< s, \,t$.  We note that
\begin{align*}
& |u_0(P -z) \cdot{\bf n}(P) -   u_0(P ) \cdot{\bf n}(P) -   u_0(Q-z) \cdot {\bf
n}(Q) +   u_0(Q) \cdot{\bf n}(Q)|\\
& =|\big(  u_0(P -z) - u_0(P) \big) \cdot \big({\bf n}(P) - {\bf n}(Q) \big)  -  \big(  u_0(P-z) - u_0 (P) - u_0(Q-z) + u_0 (Q) \big)  \cdot {\bf
n}(Q)\\
& \leq c\Big(\| u_0\|_{\dot c^\al(\R^n)}  | P -Q| + \| u_0\|_{\dot \calC_{D_\eta}(\R^n)} \eta(|P - Q|) \Big) |z|^\al .
\end{align*}
Hence, for $s < t$, we have
\[ |W (P,t) \cdot{\bf n}(P) - W(P,s) \cdot{\bf n}(P) - W(Q,t) \cdot{\bf
n}(Q) + W(Q,s)\cdot{\bf n}(Q) |
\]
\[
\leq \|   u_0\|_{\calC^{\al}_{{\mathcal D}_{\eta}} (\R^n )}  \int_{\R^n}
|\Ga(z,t) -\Ga(z,s) | |z|^\al dz \big(  \eta(|P-Q|) + |P -Q| \big)
\]
\[
\leq   \|   u_0\|_{\calC^{\al}_{{\mathcal D}_{\eta}} (\R^n )}  \int_{\R^n}
\int_s^t |D_\tau\Ga(z,\tau)| d\tau|z|^\al   dz \big(  \eta(|P-Q|) + |P -Q| \big)
\]
\[
\leq   c\|   u_0\|_{\calC^{\al}_{{\mathcal D}_{\eta}} (\R^n )} \int_s^t
\tau^{\frac12 \al -1}   d\tau \big(  \eta(|P-Q|) + |P -Q| \big)
\]
\[  \leq   c\|
u_0\|_{\calC^{\al}_{{\mathcal D}_{\eta}}  (\R^n )} (t-s)^{ \frac12 \al}
\big(  \eta(|P-Q|) + |P -Q| \big).
\]
This completes the proof.
\end{proof}

%The next lemma is well-known (see e.g. \cite{CJ}), and thus, we just
%state it and skip its details.
%
%\begin{lemma}\label{propheat2}
%Let $\al\in (0,1)$ and $f\in C^{\alpha,\frac{\al}2  }(\R^n\times
%{\mathbb R})$. Suppose that $v(x,t): =\int^t_{-\infty} \int_{\R^n}
%\Gamma(x-y,t-s)f(y,s)dyds$, where $\Gamma$ is the heat kernel. Then
%$v\in \calC^{\al+2 ,\frac{\al}2+1 }(\R^n\times {\mathbb R} )$ and
%$v$ satisfies
%\begin{align*}
%\|v\|_{\dot C^{\al+2,\frac{\al}2+1 }(\R^n\times {\mathbb R})}&\leq
%c\|f\|_{\dot C^{\alpha,\frac{\al}2  }(\R^n\times {\mathbb R})}.
%\end{align*}
%\end{lemma}
%
%{\color{red}{

For notational convention, we denote for a measurable function $f$
in $\R^n\times {\mathbb R}$
%\begin{equation}\label{CK-Aug19-10}
%\Lambda(f)(x,t) :=\int^t_{-\infty} \int_{\R^n}
%\Gamma(x-y,t-s)f(y,s)dyds,
%\end{equation}
\begin{equation}\label{CK-Aug19-20}
\Lambda_0(f)(x,t) :=\int^t_{0} \int_{\R^n}
\Gamma(x-y,t-s)f(y,s)dyds,
\end{equation}
where $\Gamma$ is the heat kernel.

Next, we also present estimates of heat equation with external
force with zero initial. It may be probably well-known to experts,
we present its details in the Appendix for reader's convenience.

\begin{lemma}\label{lemm4}
Let $T>0$, $0 < \al <1$, $f \in   \calC^{\al, \frac12 \al}(\R^n \times
[0, T])$ and $\Lambda_0(f)$ be defined in \eqref{CK-Aug19-20}. Then,
\begin{align}\label{0814infty}
\|\Lambda_0(f)\|_{\dot \calC^{\al, \frac12 \al}(\R^n \times [0, T])}
\leq c T^{1-\frac12 \al} \| f\|_{L^\infty(\R^n \times (0, T))},
\end{align}
%{\color{red}{
%\begin{align}\label{al+2}
%\|\Lambda_0(f)\|_{\dot C^{\al+2, \frac12 \al +1 } (\R^n \times (0,
%T))} \leq c \| f\|_{  \dot C^{\al, \frac12 \al}(\R^n \times (0, T))
%},
%\end{align}
%
%}}
\begin{align}\label{al+0}
\|\Lambda_0(f)\|_{\dot {\mathcal C}^{\al +1,  \frac12 \al +\frac12}(\R^n \times
[0, T]) } \leq c \max\bke{ T^\frac12, T^{\frac12 -\frac12 \al}}\|
f\|_{L^\infty (0, T; {\mathcal C}^\al(\R^n ))}.
\end{align}
%From \eqref{al+2} and \eqref{al+0},
\begin{align}\label{al+2-2}
\| \na \Lambda_0(f)\|_{\dot {\mathcal C}^{\al+1, \frac12 \al +\frac12 } (\R^n
\times [0, T])} \leq c \| f\|_{ L^\infty(0, T;  \dot {\mathcal C}^{\al}(\R^n)) },
\end{align}
\begin{align}\label{al+0-2}
\|\na \Lambda_0(f)\|_{\dot {\mathcal C}^{\al ,  \frac12 \al }(\R^n \times [0,
T]) } \leq c T^\frac12 \| f\|_{ L^\infty (0, T; \dot {\mathcal C}^{\al}(\R^n ))},
\end{align}
\begin{align}\label{al+4}
\|\na\Lambda_0(f) \|_{\dot {\mathcal C}^{\al  +\ep,  \frac12 \al  +\frac12
\ep}(\R^n \times [0, T]) } \leq c  T^{\frac12-\frac12 \ep}\| f\|_{L^\infty(0, T; \dot {\mathcal C}^{\al }(\R^n ))}, \qquad 0 < \ep <1.
\end{align}
\end{lemma}

%\begin{remark}
%A direct consequence of the above estimates are, due to the real
%interpolation of \eqref{al+2-2} and \eqref{al+0-2}, the following estimates.
%%{\color{red}{
%%\begin{align}\label{al+3}
%%\|\Lambda_0(f) \|_{\dot C^{\al +1 +\ep,  \frac12 \al +\frac12
%%+\frac12 \ep}(\R^n \times (0, T)) } \leq c \max\bke{
%%T^{\frac12-\frac12 \ep}, T^{\frac12 -\frac12 \al -\frac12 \ep}}\|
%%f\|_{\dot C^{\al, \frac12 \al}(\R^n \times (0, T))},
%%\end{align}
%%
%%}}
%\begin{align}\label{al+4}
%\|\na\Lambda_0(f) \|_{\dot {\mathcal C}^{\al  +\ep,  \frac12 \al  +\frac12
%\ep}(\R^n \times (0, T)) } \leq c  T^{\frac12-\frac12 \ep}\| f\|_{L^\infty(0, T; \dot {\mathcal C}^{\al }(\R^n ))}.
%\end{align}
%\end{remark}

Lastly, we consider the initial-boundary value problem of heat
equation \eqref{CK-Aug6-30}-\eqref{CK-Aug6-40}.
%\begin{equation}\label{BN-1}
%v_t -\De v =   f+{\rm\,div}\, F \qquad \mbox{in} \quad \Om \times
%[0, T],
%\end{equation}
%\begin{equation}\label{BN-101}
%v(x,0) =v_0 \qquad \mbox{and}\qquad v = \psi\quad
%\mbox{on}\,\,\partial\Om\times [0,T].
%\end{equation}
Here we assume that $v_0\in \calC^{\al+k}(\overline{\Omega} )$, $
\psi\in\calC^{\al+k, \frac12 (\al+k)}( \partial\Om \times [0, T])$,
$f\in \calC^{\al, \frac12 \al}(\overline{\Omega} \times [0, T])$ and
${\mathcal F}\in \calC^{\al, \frac12 \al}(\overline{\Omega}\times [0, T])$, where
$k$ is either $0$ or $1$. We let $\tilde f, \tilde {\mathcal F} \in \calC^{\al,
\frac12 \al}(\R^n \times [0, T])$ an extension of $f, {\mathcal F}$,
respectively, such that $\| \tilde f \|_{\calC^{\al, \frac12 \al}(\R^n
\times [0, T]) } \leq c \| f \|_{\calC^{\al, \frac12 \al}(\overline{\Om}
\times [0, T])}$ and $\| \tilde {\mathcal F}\|_{\calC^{\al, \frac12 \al}(\R^n
\times [0, T]) } \leq c \| {\mathcal F} \|_{\calC^{\al, \frac12 \al}(\overline{\Om}
\times [0, T])}$. Similarly, we denote by $\tilde v_0$ the extension
of $v_0$ such that $\| \tilde v_0 \|_{\calC^{\al+k, \frac12
(\al+k)}(\R^n) } \leq c \| v_0 \|_{\calC^{\al+k, \frac12
(\al+k)}(\overline{\Om})}$.

\begin{theorem}\label{heat-domain-100}
Let $\Omega$ be an bounded domain with $\calC^2$ boundary. Suppose
that $f, {\mathcal F}\in \calC^{\al, \frac12 \al}(\overline{\Om}
\times [0, T])$, $\psi\in\calC^{\al+k, \frac12 (\al+k)}( \partial\Om
\times [0, T])$ and $v_0\in\calC^{\al+k}(\overline{\Om})$ with $\psi|_{t =0} = v_0|_{\pa \Om}$, where
$k=0$ or $k=1$. Then, there exists a unique solution
$v\in\calC^{\al+k, \frac12 (\al+k)}(\overline{\Omega}\times [0, T])$
of \eqref{CK-Aug6-30}-\eqref{CK-Aug6-40} and $v$ satisfies
\begin{equation*}
\| v\|_{\calC^{\al, \frac12 \al }(\overline{\Omega} \times [0, T])}
\leq c \big(\| v_0\|_{\calC^{\al}(\overline{\Omega} )}+\|
\psi\|_{\calC^{\al, \frac12 \al}( \partial\Om \times [0, T])  }
\end{equation*}
\begin{equation}\label{CK-Aug20-500}
%+c\max\bke{ T^{\frac12}, T^{(\frac12 -\frac12 \al)}} \| f\|_{C^{\al,
%\frac12 \al  }(\overline{\Omega} \times [0, T])}+
+T^{1-\frac12 \al} \| f\|_{L^\infty(\Om \times (0, T))}+
T^{\frac12} \| {\mathcal F}\|_{\calC^{\al, \frac12 \al }(\overline{\Omega}\times
[0, T])} \big),
\end{equation}
\begin{equation*}
\| v\|_{\calC^{\al+1, \frac12 (\al+1)}(\overline{\Omega} \times [0,
T])} \leq c  \big( \| v_0\|_{\calC^{\al+1}(\overline{\Omega} )}+\|
\psi\|_{\calC^{\al+1, \frac12 (\al+1)}( \partial\Om \times [0, T])  }
\end{equation*}
\begin{equation}\label{heat-est-400-1}
+\max\bke{ T^{\frac12}, T^{(\frac12 -\frac12 \al)}} \| f\|_{\calC^{\al,
\frac12 \al  }(\overline{\Omega} \times [0, T])}+  \| {\mathcal F}\|_{\calC^{\al,
\frac12 \al }(\overline{\Omega}\times [0, T])} \big).
\end{equation}

\end{theorem}

\begin{proof}
We only prove the estimate \eqref{heat-est-400-1}, since
\eqref{CK-Aug20-500} can be computed similarly.

For convenience, we define
\begin{equation*}%\label{BN-151}
w_0(x,t): = \int_{\R^n} \Ga(x-y,t) \tilde v_0(y) dy.
\end{equation*}
By lemma \ref{proheat1}, we have
\begin{equation}\label{BN-171}
\| w_0\|_{\calC^{\al +1 , \frac12 \al +\frac12}(\R^n \times [0, T])}
\leq c \| \tilde v_0\|_{\calC^{\al +1}(\R^n  )} \leq c \| v_0\|_{\calC^{\al +1}(\bar \Om  )}.
\end{equation}
Via  \eqref{al+0} and \eqref{al+2-2}, we also obtain
\begin{equation}\label{CK-aug20-600}
\|\Lambda_0(\tilde f)\|_{{\mathcal C}^{\al +1, \frac12 \al +\frac12 }({\mathbb R}^n
\times [0, T])} \leq c\max\bke{ T^{\frac12}, T^{(\frac12 -\frac12
\al)}} \| f\|_{{\mathcal C}^{\al, \frac12 \al }(\overline \Om \times [0, T])},
\end{equation}
\begin{equation}\label{CK-aug20-700}
\|D_x\Lambda_0(\tilde {\mathcal  F})\|_{{\mathcal C}^{\al +1,
\frac12 \al+\frac12 }(\R^n \times [0, T])} \leq c\| {\mathcal
F}\|_{{\mathcal C}^{\al, \frac12 \al }(\overline \Om \times [0,
T])}.
\end{equation}
Therefore, we note that
\[
w_0, \,\,\,\Lambda_0(\tilde f), \,\,\, D_x\Lambda_0(\tilde {\mathcal F}) \in
\calC^{\al+1 , \frac12 \al+\frac12  }(\pa \Om \times [0, T]).
\]
Next let $w_1$ be the solution of the following equation
\begin{align*}%\label{BN-201}
& \hspace{45mm}\partial_t w_1 -\De w_1 = 0  \quad \mbox{in} \quad \Om \times [0, T],\\
%\end{equation}
%\begin{equation}\label{BN-301}
&w_1(x,0) =0 \qquad \mbox{and}\qquad w_1 =\psi- w_0 -
\Lambda_0(\tilde f)+ D_x\Lambda_0(\tilde {\mathcal F} ) \quad
\mbox{on}\,\,\partial\Om\times [0,T].
\end{align*}
It is well-known  that
\begin{equation}\label{CK-aug20-800}
\|w_1\|_{\calC^{\al +1, \frac12 \al +\frac12 }(\overline{\Omega}
\times [0, T])} \leq c\| \psi- w_0 - \Lambda_0(\tilde f)+
D_x\Lambda_0(\tilde {\mathcal F})\|_{{\mathcal C}^{\al+1, \frac12 (\al+1)}(
\partial\Om \times [0, T])  }.
\end{equation}
With aid of \eqref{BN-171}, \eqref{CK-aug20-600} and
\eqref{CK-aug20-700}, the righthand side of \eqref{CK-aug20-800} can
be estimated as in \eqref{heat-est-400-1}. Noting that $v = w_0 +
w_1+\Lambda_0(\tilde f)- D_x\Lambda_0(\tilde{\mathcal F})$ is the solution of
\eqref{CK-Aug6-30}-\eqref{CK-Aug6-40} and thus, $v$ satisfies the estimate
\eqref{heat-est-400-1}.
%\[
%\| u\|_{\calC^{\al +1, \frac12 \al +\frac12 }(\overline{\Omega}
%\times [0, T])} \leq c  \| v_0\|_{C^{\al +1 }(\overline{\Omega} )}
%\]
%\[
%+c\max\bke{ T^{\frac12}, T^{(\frac12 -\frac12 \al)}} \| f\|_{C^{\al,
%\frac12 \al  }(\overline{\Omega} \times [0, T])}+ c \| F\|_{C^{\al,
%\frac12 \al }(\overline{\Omega}\times [0, T])} .
%\]
This completes the proof.
\end{proof}

\begin{remark}\label{rem-3}
In case that the boundary condition in \eqref{CK-Aug6-40}, $v=\psi$ on
$\partial\Om$, is replaced by the Neumann condition $\frac{\partial
v}{\partial {\bf n}}=\psi$ on $\partial\Om$, if $
\psi\in\calC^{1+k+\al, \frac12 (1+k+\al)}( \partial\Om \times [0,
T])$ is assumed, then the same result of Theorem
\ref{heat-domain-100} can be obtained.
\end{remark}

\section{Proof of Theorem \ref{mainthm-Stokes}}
\setcounter{equation}{0}

In this section, we consider the boundary value problem of the
following Stokes system \eqref{CK-Aug6-10}-\eqref{CK-Aug6-20}. Let
$\Om$ be a $\calC^2$ bounded domain in $\R^n$. First we introduce
the notion of weak solutions for the Stokes system.
\begin{defn}\label{defn-Stokes}
%{\bf Weak solution to the Stokes system} Let $0<\al<1$.
Suppose that $ {\mathcal F}=\{F_{ij}\}_{i,j=1}^n\in
\calC^{\al,\frac{\al}{2}}(\overline \Om \times [0,T])$, $ f \in
L^\infty(\Om \times (0,T))$,  $g\in \calC^{\al,\frac{\al}{2}}(\pa
\Om\times [0,T])$ and $u_0 \in \calC^\al(\overline \Om)$. We say
that a vector field $u$ is a weak solution in the class
$\calC^{\al,\frac{\al}{2}}(\overline{\Om} \times [0,T])$ for the
Stokes system \eqref{CK-Aug6-10} with initial boundary condition
\eqref{CK-Aug6-20} if the following conditions are satisfied:

\begin{itemize}
\item[(i)] $u\in\calC^{\al,\frac{\al}{2}}(\overline{\Om} \times
[0,T])$ and $\nabla u\in L^\infty_{\rm loc}(\Om \times (0,T))$.

\item[(ii)]
For each $\Phi\in C^\infty_0(\Om \times (0,T))$ with $\mbox{\rm
div}_x\Phi=0$
\[
\int^T_0\int_{\Om}\nabla u:\nabla \Phi dxdt=\int^T_0\int_{\Om}u\cdot
\Phi_t + f \cdot \Phi -{\mathcal F}:\nabla \Phi dxdt\]

\item[(iii)]
$u(x,0)=u_0(x) $ in $\Om$.

\item[(iv)] $u(P,t)=g(P,t)$ in $ P \in \pa \Om \times (0,T)$.
\end{itemize}

\end{defn}

%{\bf Proof of Theorem \ref{mainthm-Stokes} }
%\\
For $f$ and ${\mathcal F}$ given in Theorem
\ref{mainthm-Stokes}, we denote by $\tilde   f$ and
$\tilde {\mathcal F}$ the extension of $f$ and ${\mathcal F}$, respectively, to $\R^n
\times (0,T)$ such that $\tilde f $ and $\tilde {\mathcal F}$ have compact
supports.
%et $f$ and $F$ be givne
%
%
%For the proof of theorem \ref{stokes-boundary-nonhomogeneous}, let
%$\tilde   f$ and $\tilde F$   be    extensions  of $ f$ and $F$ to
%$\R^n \times (0,T)$, respectively,  such that $\tilde f $ and
%$\tilde F$ have compact supports.
Let ${\mathbb P}$ be  the Helmholtz projection operator  on $\R^n $ such that
\begin{align*}
[{\mathbb P}\tilde{f}]_j(x,t) & =\delta_{ij} \tilde f_i+\int_{\R^n}
D_{x_i}D_{x_j}N(x-y)\tilde{f}_i(y,t)dy=\delta_{ij}\tilde{f}_i+R_iR_j\tilde{f}_i, \\
[{\mathbb P}\,{\rm div}
\,\tilde{\mathcal F}]_j &  =D_{x_k}\Big(\delta_{ij}\tilde{F}_{ki}+R_iR_j\tilde{F}_{ki}\Big),
\end{align*}
where $R_i$ is Riesz transform in $\R$.
%\begin{align*}
%{\mathbb P}\tilde{f}]_j(x,t)=\delta_{ij} \tilde f_i+\int_{\R} D_{x_i}D_{x_j}N(x-y)\tilde{f}_i(y,t)dy=\delta_{ij}\tilde{f}_i+R_iR_j\tilde{f}_i,\\
%({\mathbb P}\,{\rm div}
%\,\tilde{F})_j=D_{x_k}\Big(\delta_{ij}\tilde{F}_{ki}+R_iR_j\tilde{F}_{ki}\Big).
%\end{align*}
%and we also define an operator ${\mathbb Q}$ by
%\[
%{\mathbb Q}\tilde{f}=-\int_{\R} D_{x_i}N(x-y)\tilde{f}_i(y,t)dy.
%\]
%Then, it is well-known that $\tilde{f}$ is decomposed as
%\[
%\tilde{f}={\mathbb P}\tilde{f} +\nabla_x{\mathbb Q}\tilde{f}\quad
%\mbox{      and      } \quad {\rm div}\,{\mathbb
%P}\tilde{f}=0\,\,\mbox{ in }\,\R^n\times (0,T).
%\]

We define  $V^1$ and $V^2$ by
\begin{align*}%\label{V}
V^1_j(x,t)  & := \La_0([{\mathbb P}\tilde f]_j) (x,t)  = \La_0[ \delta_{ij}\tilde{f}_i+R_iR_j\tilde{f}_i] (x,t),\\   % \int^t_{0}\int_{\R^n} \Gamma(x-y,t-s)[{\mathbb P}\tilde{f}]_j(y,s)dyds\\
 %  & =\int^t_{0}\int_{\R^n} \Gamma(x-y,t-s)[\delta_{ij}\tilde{f}_{i}+R_iR_j\tilde{f}_{i}](y,s)dyds,\\
V^2_j(x,t)     & := \La_0([{\mathbb P}\,{\rm div}
\,\tilde{\mathcal F}]_j ) (x,t) = - D_x\La_0 ( \delta_{ij}\tilde{F}_{ki}+R_iR_j\tilde{F}_{ki} )(x,t).  % \int^t_{0}\int_{\R^n}
%D_y\Gamma(x-y,t-s)[\delta_{ij}\tilde{F}_{ki}+R_iR_j\tilde{F}_{ki}](y,s)dyds.
 \end{align*}
We observe that $V^1$  and $V^2$  satisfy the equations
\[
 \begin{array}{l}\vspace{2mm}
V^1_t - \De V^1  ={\mathbb P}\tilde{f},\quad {\rm div} \, V^1=0
\qquad \mbox{ in }\,\,
 \R^n \times (0,T), \\
\hspace{20mm}V^1|_{t=0}= 0\qquad\mbox{ on }\,\,\R^n.
\end{array}
\]
\[
 \begin{array}{l}\vspace{2mm}
V^2_t - \De V^2  ={\mathbb P}\, {\rm div} \, \tilde{\mathcal F},\quad {\rm
div} \, \, V^2=0\qquad \mbox{ in }\,\,
 \R^n \times (0,T), \\
\hspace{20mm}V^2|_{t=0}= 0\qquad\mbox{ on }\,\,\R^n.
\end{array}
\]
Since support of $\tilde f$ is bounded, we obtain  $\|
R_iR_j \tilde f\|_{L^\infty(\R^n)} \leq c \| \tilde f\|_{
L^\infty(0, T;\dot  \calC_{D_\eta} (\R^n))}$. By \eqref{0814infty}, we have
\begin{align*}%\label{CK-1107-1}
  \| V^1\|_{\dot \calC^{\al, \frac12 \al  }(\R^n \times [0, T])} & \leq  c
T^{1 -\frac12 \al} \|
\tilde f + R_iR_j \tilde f\|_{L^\infty (\R^n \times (0, T))}\\
  & \leq  c   T^{1 -\frac12 \al} \big( \|   \tilde f \|_{L^\infty
(\R^n \times 0, T)} +   \|
\tilde f \|_{ L^\infty(0, T;\dot  \calC_{D_\eta} (\R^n))} \big) \\
& \leq  c   T^{1 -\frac12 \al}   \|
 \tilde f \|_{ L^\infty(0, T;   \calC_{D_\eta} (\R^n ))},\\
   \| V^1\|_{L^\infty(\R^n \times (0, T))} & \leq  c
T  \|
\tilde f + R_iR_j \tilde f\|_{L^\infty (\R^n \times (0, T))}\\
& \leq  c   T    \|
 \tilde f \|_{ L^\infty(0, T;   \calC_{D_\eta} (\R^n ))}
%& \leq  c   T^{(1 -\frac12 \al)}  \|    f \|_{C^{\al, \frac12 \al} (\Om \times [0, T])}.
\end{align*}
Hence, we have
\begin{align}\label{CK-1107-1}
  \| V^1\|_{  \calC^{\al, \frac12 \al  }(\R^n \times [0, T])}
& \leq  c  \max( T, T^{1 -\frac12 \al} )   \|
 \tilde f \|_{ L^\infty(0, T;   \calC_{D_\eta} (\R^n ))}.
%& \leq  c   T^{(1 -\frac12 \al)}  \|    f \|_{C^{\al, \frac12 \al} (\Om \times [0, T])}.
\end{align}
%For the second inequality we use the fact $\| R_iR_j \tilde f\|_{L^\infty(\R^n)} \leq c \| \tilde f\|_{ L^\infty(0, T;\dot  C_{D_\eta} (\R^n))}$.
Moreover, we note that
\begin{align}\label{CK1105-2}
\notag &V^1(P,t)\cdot {\bf n}(P)  -V^1(P,s)\cdot {\bf n}(P) - V^1(Q,t)\cdot
{\bf n}(Q) + V^1(Q,s)\cdot {\bf n}(Q)\\
\notag &=\Big( V^1(P,t) - V^1(P,s) \Big) \cdot \Big({\bf n}(P) - {\bf n}(Q) \Big)
 + \Big(V^1(P,t) - V^1(P,s) - V^1(Q,t) + V^1(Q,s) \Big) \cdot {\bf n}(Q)\\
&: = I_1 + I_2.
\end{align}
%\begin{align*}
%V^1(P,t)\cdot {\bf n}(P)  -V^1(P,s)\cdot {\bf n}(P) - V^1(Q,t)\cdot {\bf n}(Q) + V^1(Q,s)\cdot {\bf n}(Q)\\
%=( V^1(P,t) - V^1(P,s) ) \cdot ({\bf n}(P) - {\bf n}(Q) )
% + (V^1(P,t) - V^1(P,s) - V^1(Q,t) + V^1(Q,s) ) \cdot {\bf n}(Q).
%\end{align*}
From \eqref{0814infty}, we get
\begin{align}\label{CK1105-1}
 |I_1 | \leq  c\| V^1 \|_{L^\infty(\R^n ; \dot C^{\frac12 \al } [0,T])}|t-s|^{\frac12 \al} |P -Q| \leq  c   T^{1 -\frac12 \al}   \|
 \tilde f \|_{ L^\infty(0, T;   \calC_{D_\eta} (\R^n ))} |t-s|^{\frac12 \al} |P -Q|.
\end{align}
And
\begin{align*}
I_2 & =\int_s^t \int_{\R^n}  \big(\Ga(P-z, t-\tau) - \Ga(Q-z, t -\tau ) \big) {\mathbb P}[ \tilde f](z,\tau) dzd\tau  \\
&\quad  + \int_0^s \int_{\R^n} \big(\Ga(P-z, t-\tau) - \Ga(P-z, s -\tau ) - \Ga(Q-z, t-\tau) + \Ga(Q-z, s -\tau )  \big) {\mathbb P}[\tilde f ](z,\tau) dzd\tau.
\end{align*}
The first term is
\begin{align*}
& |\int_s^t \int_{\R^n} \int_0^1 D_z    \Ga( \te P + (1 -\te) Q -z, t-\tau) \cdot (P -Q) d\te  {\mathbb P}[ \tilde f ](z,\tau)   dzd\tau|\\
&\leq c| P -Q| \| \tilde f\|_{ L^\infty(0, T;   \calC_{D_\eta} (\R^n ))}  \int_s^t  (t -\tau)^{-\frac12 }       dzd\tau \\
&\leq c\| \tilde f\|_{ L^\infty(0, T;   \calC_{D_\eta} (\R^n ))} |P -Q| (t -s)^\frac12.
\end{align*}
The second term is
\begin{align*}
&| \int_0^s \int_{\R^n} \int_0^1 \int_0^1 D_\te D_\ga \Ga( \te P + (1 -\te) Q-z, \la t+ (1 -\la) s -\tau)  \tilde f(z,\tau) d\te d\la dzd\tau |\\
&\leq  c\|\tilde  f\|_{ L^\infty(0, T;   \calC_{D_\eta} (\R^n ))} |P -Q| (t -s)    \int_0^1 \int_0^s    (  \la t+ (1 -\la) s -\tau )^{-\frac32}  d\tau   d\la\\
&\leq  c\| \tilde f\|_{ L^\infty(0, T;   \calC_{D_\eta} (\R^n ))} |P -Q| (t -s)   \int_0^1     (  \la t+ (1 -\la) s  )^{-\frac12} -  \la^{-\frac12}( t -s)^{-\frac12}   d\tau   d\la\\
&  = c\| \tilde f\|_{ L^\infty(0, T;   \calC_{D_\eta} (\R^n ))} |P -Q| (t -s)   \big(  (t-s)^{-1}(t^\frac12 -s^\frac12) - (t -s)^{-\frac12} \big)\\
&\leq  c\| \tilde f\|_{ L^\infty(0, T;   \calC_{D_\eta} (\R^n ))} |P -Q| (t -s)^\frac12.
\end{align*}
Hence, we have
\begin{align}\label{CK1105-3}
I_2 \leq  \| \tilde f\|_{ L^\infty(0, T;   \calC_{D_\eta} (\R^n ))} |P -Q| (t -s)^\frac12.
\end{align}
By \eqref{CK1105-2}-\eqref{CK1105-3},  we have
\begin{align}\label{CK-september-001}
 \|  V^1 \cdot {\bf n}\|_{\dot {\mathcal C}_{{\mathcal D}_\eta} (\pa \Omega; \dot {\mathcal C}^{\frac{\al}2}[0,
T])} \leq  c   \max (T^{1 -\frac12 \al }, T^{\frac12 -\frac12 \al} )  \|
 f \|_{ L^\infty(0, T;  \mathcal C_{{\mathcal D}_\eta} (\overline{\Om} ))}.
\end{align}

Next, we estimate $V^2$. Since $R_i: \dot \calC^\al(\R^n) \ri \dot \calC^\al(\R^n)$ is bounded, by \eqref{al+0}, we have
\begin{align*}%\label{CK-september-001-1}
\notag\|  V^2\|_{ \dot \calC^{\al, \frac12 \al  } (\R^n \times [0, T])}
&\leq c
\max ( T^{\frac12 }, T^{\frac12 -\frac12 \al}  ) \|{\mathbb P}[ \tilde{\mathcal F}] \|_{  L^\infty(0, T;  \dot\calC^{\al }(  \R^n )) }\\
& \leq c
\max ( T^{\frac12 }, T^{\frac12 -\frac12 \al}  ) \| \tilde {\mathcal F} \|_{  L^\infty(0, T;  \dot\calC^{\al }( \overline \R^n )) }.
\end{align*}

We estimate $\| V^2\|_{L^\infty(\R^n \times (0, T))}$.  It is well known that $D_x\Gamma_t\in {\mathcal H}^1(\R^n)$ with $\|D_x\Gamma_t\|_{ {\mathcal H}^1(\R^n)}\leq ct^{-\frac{1}{2}}$, where ${\mathcal H}^1(\R^n)$ denotes Hardy space.  Since $R_i : BMO(\R^n) \ri BMO(\R^n)$ is bounded,  we have
\begin{align*}
\| V^2\|_{L^\infty(\R^n \times (0, T))} & \leq \big(\| \tilde {\mathcal F} \|_{L^\infty(\R^n \times (0, T))}  \int_0^t |D_x\Ga(x-y,t) dy|
   + \int_0^t \| R_iR_j  \tilde {\mathcal F }(t)\|_{BMO(\R^n)} \|D_x\Gamma_t\|_{ {\mathcal H}^1(\R^n)} \big) \\
& \leq c    t^\frac12  \| \tilde {\mathcal F} \|_{L^\infty(\R^n \times (0, T))}.
\end{align*}
Hence, we have
\begin{align}\label{CK-september-001-1}
 \|  V^2\|_{   \calC^{\al, \frac12 \al  } (\R^n \times [0, T])}
& \leq c
\max ( T^{\frac12 }, T^{\frac12 -\frac12 \al}  ) \| {\mathcal F} \|_{  L^\infty(0, T;   \calC^{\al }( \overline \Om )) }.
\end{align}

Moreover,
\begin{align}\label{CK1105-4}
\notag&V^2(P,t)\cdot {\bf n}(P)  -V^2(P,s)\cdot {\bf n}(P) - V^2(Q,t)\cdot {\bf n}(Q) + V^2(Q,s)\cdot {\bf n}(Q)\\
\notag&=( V^2(P,t) - V^2(P,s) ) \cdot ({\bf n}(P) - {\bf n}(Q) )
 + (V^2(P,t) - V^2(P,s) - V^2(Q,t) + V^2(Q,s) ) \cdot {\bf n}(Q)\\
&: = II_1 + II_2.
\end{align}
By \eqref{al+0-2}, we get
\begin{align*}
II_1 \leq c \| V^2 \|_{L^\infty(\R^n ; \dot \calC^{\frac12 \al } [0,T])}|t-s|^{\frac12 \al} |P -Q| \leq  c   T^{\frac12 }   \|
 \tilde {\mathcal F} \|_{ L^\infty(0, T;  \dot \calC^\al  (\R^n ))} |t-s|^{\frac12 \al} |P -Q|
\end{align*}
and
\begin{align}\label{CK1105-5}
\notag &\int_s^t \int_{\R^n}   \na_z \Ga(z, t-\tau) \big({\mathbb P}[ \tilde {\mathcal F}] (P -z,\tau) -{\mathbb P}[ \tilde {\mathcal F}] (Q-z, \tau) \big) dzd\tau  \\
& \quad +\int_0^s \int_{\R^n} \big( \na_z\Ga(z, t-\tau) -  \na_z\Ga(z, s
-\tau )\big( {\mathbb P}[ \tilde{\mathcal F}](P -z,\tau) -{\mathbb P}[ \tilde {\mathcal F}] (Q -z, \tau) \big)
dzd\tau.
\end{align}
The first term is dominated by
\begin{align*}
  | P -Q|^\al \| {\mathbb P}[ \tilde{\mathcal F}]\|_{ L^\infty(0, T; \dot \calC^\al(\R^n))}  \int_s^t  (t -\tau)^{-\frac12  }       dzd\tau
&\leq c\| \tilde {\mathcal F}\|_{ L^\infty(0, T; \dot \calC^\al(\R^n))}  |P -Q|^\al (t -s)^{\frac12 }.
\end{align*}
The second term is
\begin{align*}
&\int_0^s \int_{\R^n} \int_0^1 \int_0^1  D_\ga \Ga(  z, \la t+ (1 -\la) s -\tau)  \big( {\mathbb P}[ \tilde {\mathcal F}](P -z,\tau) - {\mathbb P}[ \tilde {\mathcal F}](Q -z, \tau) \big)   d\la dzd\tau\\
&\leq  c\| \tilde{\mathcal F}\|_{L^\infty( 0, T; \dot \calC^\al (\R^n ))} |P -Q|^\al    \int_0^1 \int_0^s    (  \la t+ (1 -\la) s -\tau )^{-\frac32}  d\tau   d\la\\
&\leq  c\| \tilde {\mathcal F}\|_{L^\infty( 0, T; \dot \calC^\al (\R^n ))} |P -Q|^\al  (t -s)^\frac12.
\end{align*}
Hence, we have
\begin{align}\label{CK1105-6}
|II_2 | \leq c   \| \tilde {\mathcal F}\|_{L^\infty( 0, T; \dot \calC^\al (\R^n ))} |P -Q|^\al  (t -s)^\frac12.
\end{align}
From \eqref{CK1105-4}-\eqref{CK1105-6}, we have
\begin{align}\label{CK-september-002}
\|  V^2 \cdot {\bf n} \|_{\dot \calC_{D_{\eta}}(\pa \Om; \dot C^{\frac12 \al} [0, T])}  \leq c \max\bke{ T^{\frac12 }, T^{\frac12 -\frac12
\al }}  \|   {\mathcal F}\|_{L^\infty( 0, T; \dot \calC^\al (\overline \Om ))}.
\end{align}

Next we treat initial data $u_0$. Let $\widetilde{u}_0$ be an
extension of  $u_0$ satisfying that $
 \mbox{div} \, \widetilde{u}_0=0\mbox{ in }\R^n$.
Letting $v$ by
   \begin{equation*}%\label{small-v}
  v(x,t):=\int_{\R^n}\Gamma(x-y,t) \widetilde{u}_0(y)dy.
  \end{equation*}
We observe that $v$ satisfies the equations
 \[
% \label{heatsystem}
\begin{array}{l}\vspace{2mm}
v_t - \De v  =0,\quad {\rm div} \,  v=0 \qquad \mbox{ in }\,\,
 \R^n \times (0,T),\\
\hspace{20mm}v|_{t=0}= \tilde{u}_0 \qquad\mbox{ on }\,\,\R^n.
\end{array}
\]
By Lemma \ref{proheat1}, we have
\begin{align*} %\label{CK-August17-310}
\|  v \|_{  \calC^{\alpha,\frac{\al}2}( \R^n \times [0, T])} \leq c
 \| \tilde u_0\|_{  \calC^{\al}  (\R^n)} \leq c \|  u_0\|_{  \calC^{\al}  (\overline \Om)},\\
%\end{equation}
%\begin{equation}\label{CK-August17-310-1}
 \|  v \cdot {\bf n}\|_{\dot {\mathcal C}_{{\mathcal D}_\eta} (\pa \Omega; \dot {\mathcal C}^{\frac{\al}2}[0,
T])} \leq c   \| \tilde u_0\|_{\dot \calC^{\al}_{{\mathcal D}_\eta}
(\R^n)}\leq c   \|  u_0\|_{ \calC^{\al}_{{\mathcal D}_\eta} (\overline \Om)}.
\end{align*}
We denote $G$ as
\begin{align*}%\label{CK-AUG15-100}
G = \phi - V^1|_{\pa \Om \times (0, T) } - V^2|_{\pa \Om \times (0, T)
}- v|_{\pa \Om \times (0, T)}
\end{align*}
We note that $G|_{t =0} = 0$ if $\phi|_{t =0} = u_0$ on $\pa \Om$ and
also observe that $G$ satisfies
\begin{align*}
\| G\|_{{\mathcal C}^{\al, \frac12 \al}(\pa \Om \times (0, T))} &
\leq  c\big( \| \phi \|_{{\mathcal C}^{\al, \frac12 \al}(\pa \Om \times
[0, T])} +  \max(T, T^{1 -\frac12 \al} )  \|
 f \|_{ L^\infty(0, T;   \calC_{D_\eta} (\overline \Om ))}\\
 &\qquad \qquad  + \max\bke{ T^{\frac12 }, T^{\frac12 -\frac12
\al }}  \|   {\mathcal F}\|_{L^\infty( 0, T; \dot \calC^\al (\overline \Om ))}  + \| u_0\|_{{\mathcal C}^\al (\overline \Om)} \big),\\
 \|  G \cdot {\bf n}\|_{\dot {\mathcal C}_{{\mathcal D}_\eta} (\pa \Omega; \dot {\mathcal C}^{\frac{\al}2}[0,
T])} & \leq  c \big(  \|  \phi \cdot {\bf n}\|_{\dot {\mathcal
C}_{{\mathcal D}_\eta} (\pa \Omega; \dot {\mathcal
C}^{\frac{\al}2}[0, T])} +   \| u_0\|_{\dot C^{\al}_{{\mathcal
D}_\eta} (\overline \Om)}  +  \max (T^{1 -\frac12 \al }, T^{\frac12 -\frac12 \al} )   \|
 f \|_{ L^\infty(0, T;   \calC_{D_\eta} (\overline\Om ))} \big).
\end{align*}

We decompose the solution $u$ in \eqref{CK-Aug6-10}-\eqref{CK-Aug6-20} as the form of $u =
V^1+ V^2 + v + w$, where $w$ solves the following equations:
\begin{align}\label{maineq-2-2}
\left\{\begin{array}{l}\vspace{2mm} w_t - \De w + \na \pi =0, \qquad
\Om \times (0, T),\\ \vspace{2mm} {\rm div} \, w =0, \qquad \Om
\times (0, T),\\ \vspace{2mm} w|_{\pa \Om \times (0,T)} = G, \quad
w|_{t=0} =0.
\end{array}
\right.
\end{align}
Hence, solving the equations  \eqref{CK-Aug6-10}-\eqref{CK-Aug6-20} is reduced to treat the
equations \eqref{maineq-2}. For the estimate in Theorem
\ref{mainthm-Stokes}, %\eqref{0815-holderinequality},
it suffices to
obtain that
\begin{align}\label{CK-Aug15-200}
 \| w\|_{{\mathcal C}^{\al,\frac12 \al} (\Om \times [0, T])} \leq c
\big(  \|  G \cdot {\bf n}\|_{\dot {\mathcal C}_{{\mathcal D}_\eta}
(\pa \Omega; \dot {\mathcal C}^{\frac{\al}2}[0, T])} +  \|
G\|_{{\mathcal C}^{\al, \frac12 \al}(\pa \Om \times [0, T])} \big).
\end{align}

\subsection{Invertibility of boundary integral operators}

In this subsection, we will provide the estimate
\eqref{CK-Aug15-200}. Denoting $w$, $\pi$ and $G$ in
\eqref{maineq-2-2} by $u$, $q$ and $g$, respectively, we consider
\begin{align}\label{maineq-2}
\left\{\begin{array}{l}\vspace{2mm} u_t - \De u + \na q =0, \qquad
\Om \times (0, T),\\ \vspace{2mm} {\rm div} \, u  =0, \qquad \Om
\times (0, T),\\ \vspace{2mm} w|_{\pa \Om \times (0,T)} = g, \quad
u|_{t=0} =0.
\end{array}
\right.
\end{align}
Due the result of  Solonnikov (\cite{So}), the solution  of \eqref{maineq-2} can be written in the form
\begin{align}\label{form}
u (x,t) = {\mathcal U}[\Phi](x,t) + \na {\mathcal V}[\Psi](x,t),
\end{align}
where ${\mathcal V}$ is electrostatic potential of a single layer, i.e.,
\begin{equation}\label{potential-V}
{\mathcal V}[\Psi](x,t) = \int_{\pa \Om} N(x-Q) \Psi(Q,t) dQ,
\end{equation}
where $N$ is fundamental solution of Laplace equation. On the other
hand, ${\mathcal U}$ is referred as the hydrodynamical potential, which is
defined by
\begin{equation}\label{potential-U}
{\mathcal U}[\Phi](x,t) = \int_0^t \int_{\pa \Om} {\mathcal G}(x, Q, t-s)
\Phi(Q,s) dQds.
\end{equation}
Here ${\mathcal G}$ is the tensor given by
\begin{equation}\label{tensor-g-10}
{\mathcal G}(x,Q,t)= -2\frac{\pa \Ga(x-Q,t)}{\pa {\bf n}_Q} \big(I -
{\bf n}(Q)\otimes {\bf n}(Q) \big) +4 \big(\na_x - {\bf n}(Q) \frac{\pa }{\pa
{\bf n}}\big) q(x,Q,t),
\end{equation}
where ${\bf n}(Q)$ is unit outer normal vector at $Q \in \pa \Om$.
The corresponding pressure tensor is given as
\[
q(x,Q,t) = \int_{\prod (x,Q)} \frac{\pa \Ga (Z-Q,t)}{\pa {\bf n}}
\na N(x-Z) dZ,
\]
%{\tt...Notation of pressure tensor is to be changed....}
%\begin{align*}
%{\mathcal G}(x,Q,t)& = -2\frac{\pa \Ga(x-Q,t)}{\pa {\bf n}_Q} (I - {\bf n}(Q)\otimes {\bf n}(Q)) +4 (\na_x - {\bf n}(Q) \frac{\pa }{\pa {\bf n}}) Q(x,Q,t),\\
%Q(x,Q,t) & = \int_{\prod (x,Q)} \frac{\pa \Ga (Z-Q,t)}{\pa {\bf n}}
%\na N(x-Z) dZ,
%\end{align*}
where $\prod (x,Q)$ is the layer between the tangent plane at $  Q
\in \pa \Om$ and the parallel plane passing through the point $x$ (see \cite[pp 115-117]{So}).
%{\sidenote{\color{blue} we need to study pressure tensor as well as
%tensor $\mathcal G$}}

%{\tt...up to here...}

We recall some estimates of $\mathcal G=(G_{ij})_{i,j=1,2\cdots, n}$
(see section 3 in \cite{KS1}). Let $P, Q, Z\in\pa\Om$. We then have
for all $0 <
\lambda <1$ %we have that for $P, \, Q, \, Z \in \pa \Om$,
\begin{align}
|G_{ij} (P,Q, t)| & \leq c_\lambda \frac{ 1 }{t^{\frac{1 +\lambda}2} (|P-Q|^2 +  t)^{\frac{n -2\la }2} },\label{boundary-boundary1}\\
%|\partial_t G_{ij} (P,Q, t)| & \leq c_\lambda \frac{1 }{t^{\frac{3 +\lambda}2} (|P-Q|^2  + t)^{\frac{n -2\la }2} },\label{boundary-boundary2}\\
|G_{ij} (P,Z,t) - G_{ij}(Q,Z,t)| & \leq  c_\lambda \frac{
|P-Q|^\lambda  }{t^{\frac{1 +\lambda}2} (|P-Z|^2 + t)^{\frac{n -
\lambda}2} }, \quad \mbox{     if     } \,\,   |P-Q| \leq
\frac12|P-Z|.\label{boundary-boundary3}
\end{align}

%Here we assume that  $g \cdot {\bf n} =0$ on $\partial \Omega$.

%Let $P \in \pa \Om$. Using the rotation and translation, we may assume that $P =0$ such that there is $r_0$ independent of $P$ and $C^2$ function  $\Theta :{\mathbb R}^{n-1} \rightarrow {\mathbb R}$ such that
%\begin{align}\label{r-0}
%\notag \pa \Om \cap B(0, r_0) & = \{ Q = (x', \Theta(x')) \, | \, x' \in B'(0,
%r_0) \} \\
%\Om \cap B(0, r_0)& = \{ x \in B(0, r_0) \, | \, x_n > \Theta(x'),
%\quad  x' \in B'(0, r_0) \},
%\end{align}
%where $ B(0, r_0)$ and $ B'(0, r_0)$ are open ball in ${\mathbb R}^n$ and ${\mathbb R}^{n-1}$, respectively, whose center is zero and radius is $r_0$.  Furthermore, $\Theta$ satisfies that
%\begin{align}\label{flatboundary}
%|\Theta(x')| \leq C |x'|^2,\quad
%|\na^{\prime}\Theta (x')|\leq C|x'|, \quad
%| \na^{\prime}\na^{\prime}\Theta (x^{\prime})|  \leq C
%\end{align}
%for $ x^\prime\in B^{'}(0, r_0)$.

%
%{\color{red}{
%Then, we obtain that for  $x \in \Om$ and  $ Q \in \pa \Om $,
%\begin{align}
%|G_{ij} (P,Q, t)| & \leq c  \frac{\de^{\la}(x) }{t^{\frac{1 +\la }2} (|P-Q|^2 +  t)^{\frac{n }2} },\label{boundary-interior1}\\
%|G_{ij t} (P,Q, t)| & \leq c  \frac{\de^\la (x)}{t^{\frac{3 +\la  }2} (|P-Q|^2  + t)^{\frac{n }2} }\label{boundary-interior2}.
%\end{align}
%
%
%}}

Let $\Phi$ and $\Psi$ satisfy the following condition;
\begin{equation}\label{compatible-condition}
\Phi(P, t)  \cdot {\bf n}(P) =0, \qquad \int_{\pa \Om} \Psi(Q, t) dQ
=0.
\end{equation}

For convenience, for any vector field $h$ defined on $\pa\Om$ we
denote by $h_{tan}$ the tangential componential of $h$, i.e.
$h_{tan} = h - {\bf n}(h \cdot {\bf n})$. By the \eqref{form}, we
solve the following equations
\begin{align}\label{boundarysystem}
\notag \Phi  + U_{tan} [\Phi]  + \na_S V [\Psi] = g_{tan},\\
\Psi + K^* [\Psi] + {\bf n}\cdot U
[ \Phi] = g \cdot {\bf n}
\end{align}
(see \cite[pp 120, (2.26)]{So}).
Here,  $V[\Psi]$ and $U[\Phi]$  are the direct values of \eqref{potential-V} and \eqref{potential-U}   on
$\pa \Om$, respectively, i.e.
\begin{align}\label{CK1107-2}
U[\Phi](P,t) = \int_0^t \int_{\pa \Om} {\mathcal G}(P, Q,
t-s) \Phi(Q,s) dQds,\\
\label{CK1107-3}
V[\Psi](P,t) = \int_{\pa \Om} N(P-Q) \Psi(Q,t) dQ,\qquad P \in \pa \Om
\end{align}
and
\begin{align}\label{double-layer}
K^* \Psi (P,t) = p.v \,\, c_n\int_{\pa \Om} \frac{(P -Q)\cdot {\bf n}(P)}{|P-Q|^n} \Psi(Q,t) dQ.
\end{align}
In addition, %$U_{tan}(\Phi)$ indicates its tangent
%component of $\mathcal G$, that is $U_{tan}(\Phi) =
%{\mathcal G}(\Phi) - {\bf n} ({\bf n} \cdot {\mathcal G}(\Phi))$.
$\na_S V$ indicates the tangential gradient of $V$ on $\pa \Om$,
namely $\na_S V = \na V - {\bf n} \frac{\pa V}{\pa {\bf n}}$.
% is
%tangent component of the gradient of $V$ on $\pa \Om$, and $g_{tan}
%= g - {\bf n}(g \cdot {\bf n})$.

\begin{lemma}\label{lemma508}
Let $P, \, Q  \in \pa \Om$ and $0 <\al < 1$.  There is
$\de=\delta(\al)$ with $0<\de $ such that the tensor
$\mathcal G$ given in \eqref{tensor-g-10} satisfies
\begin{align}\label{CK-August}
\int_0^t \int_{\pa \Om} |\mathcal G(P,Z,\tau) - \mathcal
G(Q,Z,\tau)| dZd\tau\leq ct^{\de} |P - Q|^\al,
\end{align}
\begin{equation}\label{CK-June20-200}
\int_s^t \int_{\pa \Om} |\mathcal G(P,Z,\tau) | dZd\tau  \leq
c(t-s)^{\de}, \qquad 0\leq s < t,\qquad P \in \pa \Om.
\end{equation}
\end{lemma}
\begin{proof}
First, we prove \eqref{CK-August}.
Let $r = |P -Q|$. Then, we have
\[
\int_0^t \int_{\pa \Om} |\mathcal G(P,Z,\tau) - \mathcal
G(Q,Z,\tau)| dZd\tau
%\]
%\[
=\int_0^t \int_{|P -Z| < 2r} \cdots  dZd\tau + \int_0^t \int_{|P -Z|
> 2r} \cdots  dZd\tau := I_1 + I_2.
\]
%\begin{align*}
%&   \int_0^t \int_{\pa \Om} |G(P,Z,t-s) - G(Q,Z,t-s)|  dZds \\
%&=\int_0^t \int_{|P -Z| < 2r} |G(P,Z,t-s) - G(Q,Z,t-s)|  dZds  + \int_0^t \int_{|P -Z| > 2r} |G(P,Z,t-s) - G(Q,Z,t-s)|  dZds\\
% & \quad = I_1 + I_2.
%\end{align*}
Via the inequality \eqref{boundary-boundary1}, for $0 < \lambda < 1$
we have
\begin{align*}
I_1 & \leq  c\int_0^t \int_{|Z| < 2r} \frac{  1  }{\tau^{\frac{1
+\la
 }2} (|Z|^2 + \tau)^{\frac{n -2\la}2} } dZd\tau  \leq
c\int_0^t  \tau^{ -1 + \frac{\lambda}2}   \int_{|Z| < 2
\frac{r}{\sqrt{\tau}}} \frac{  1  }{  (|Z|^2 + 1)^{\frac{n -2
\lambda}2} } dZd\tau.
\end{align*}
%We note that
%\begin{equation}\label{CK-June19-10}
%\int_{|Z| < 2 \frac{r}{\sqrt{s}}} \frac{  1  }{  (|Z|^2 +
%1)^{\frac{n -2 \lambda}2} } dZ\le c\min\bket{1, (
%\frac{r}{\sqrt{s}})^{n-1}, ( \frac{r}{\sqrt{s}})^{-1 + 2\lambda}}
%\end{equation}
In case that $\al < \frac12$, we take $\lambda$ with $\al < \lambda
< \frac12$. If $r^2 \le t$, then
\begin{align*}
I_1 %& =  c\int_0^{r^2}  s^{ -1 + \frac{\lambda}2}   \int_{|Z| < 2
%\frac{r}{\sqrt{s}}} \frac{  1  }{  (|Z|^2 + 1)^{\frac{n -2
%\lambda}2} } dZds
%      + c\int_{r^2}^t  s^{ -1 + \frac{\lambda}2}   \int_{|Z| < 2 \frac{r}{\sqrt{s}}} \frac{  1  }{  (|Z|^2 + 1)^{\frac{n -2 \lambda}2} } dZds\\
    & \leq  c\int_0^{r^2}  \tau^{ -1 + \frac{\lambda}2} d\tau
      + c\int_{r^2}^t  \tau^{ -1 + \frac{\lambda}2}   ( \frac{r}{\sqrt{\tau}})^{n-1}
      d\tau
   \leq cr^\lambda \leq ct^{\frac12 \lambda -\frac12 \al} r^\al.
\end{align*}
On the other hand, if $r^2 > t$, then we have
\begin{align*}
I_1     \leq  c\int_0^t  \tau^{ -1 + \frac{\lambda}2} d\tau \leq
ct^{\frac{\lambda}2} \leq ct^{\frac12 \lambda -\frac12 \al}
r^\al.
\end{align*}
In case that $\al \geq \frac12$, we take $\lambda$ with $\al <
\lambda < \frac12 + \frac12 \al$. If $r^2 \le t$, then
\begin{align*}
I_1 %& =  c\int_0^{r^2}  s^{ -1 + \frac{\lambda}2}   \int_{|Z| < 2
%\frac{r}{\sqrt{s}}} \frac{  1  }{  (|Z|^2 + 1)^{\frac{n -2
%\lambda}2} } dZds
%      + c\int_{r^2}^t  s^{ -1 + \frac{\lambda}2}   \int_{|Z| < 2 \frac{r}{\sqrt{s}}} \frac{  1  }{  (|Z|^2 + 1)^{\frac{n -2 \lambda}2} } dZds\\
    & \leq  c\int_0^{r^2} \tau^{ -1 + \frac{\lambda}2} ( \frac{r}{\sqrt{\tau}})^{-1 + 2\lambda} d\tau
      + c\int_{r^2}^t  \tau^{ -1 + \frac{\lambda}2}   ( \frac{r}{\sqrt{\tau}})^{n-1} d\tau
   \leq cr^\lambda \leq ct^{\frac12 \lambda -\frac12 \al} r^\al.
\end{align*}
In case that $r^2 > t$, we get
\begin{align*}
I_1     &\leq  c\int_0^t \tau^{ -1 + \frac{\lambda}2} (
\frac{r}{\sqrt{\tau}})^{-1 + 2\lambda} d\tau \leq r^{-1 +2\lambda}
t^{\frac12 -\frac{\lambda}2} \leq ct^{\frac12 \lambda -\frac12 \al}
r^\al.
\end{align*}

For $I_2$, we take $\la > \al$. Using the estimate \eqref{boundary-boundary3}, we have
\begin{align*}
I_2 \leq c  r^\lambda   \int_0^t \int_{|Z| > 2r}\frac{ 1
}{\tau^{\frac{1 +\lambda}2} (|Z|^2 + \tau)^{\frac{n - \lambda}2}
}dZd\tau\leq c r^\lambda \int_0^t  \tau^{-1}\int_{|Z| >
2\frac{r}{\sqrt{\tau}}}
            \frac{ 1  }{  (|Z|^2 + 1)^{\frac{n - \lambda}2}
            }dZd\tau.
%& \leq c r^\lambda t^{\frac12 -\frac{\lambda}2}.
\end{align*}
If $r^2 \le t$, then
\begin{equation*}%\label{CK-June19-20}
I_2 \leq c  r^\lambda \int_0^{r^2}  \tau^{-1}( \frac{r}{\sqrt{\tau}}
)^{-1 +\lambda}d\tau + c r^\lambda \int_{r^2}^t  \tau^{-1} d\tau
\leq c r^{\lambda}    + c r^\lambda |\ln \,\frac{t}{r^2}|  \leq
ct^{\frac12 \lambda -\frac12 \al} r^\al.
\end{equation*}
For the last inequality, we used the fact of $|\ln \,\frac{t}{r^2}|
\leq c (\frac{t}{r^2})^{\frac{\la}2 -\frac{\al}2} $ for
$\frac{t}{r^2} \ge 1$. On the other hand, if $r^2 > t$, we have
\begin{align*}
I_2 \leq c  r^\lambda \int_0^t  \tau^{-1}( \frac{r}{\sqrt{\tau}}
)^{-1 +\lambda}d\tau \leq c  r^{-1+ 2\lambda} t^{\frac12
-\frac{\lambda}2} &\leq ct^{\frac12 \lambda -\frac12 \al} r^\al.
\end{align*}
%Hence, we complete the proof.
%\end{proof}
%
%\begin{lemma}\label{lemma508-2}
%Let $P\in \pa \Om$, $0 <\al < 1$ and $0 < t$. Then, there exists
%$\de > 0$ such that
%\begin{equation}\label{CK-June20-200}
%\int_0^t \int_{\pa \Om} |\mathcal G(P,Z,\tau) | dZd\tau  \leq
%ct^{\de},
%\end{equation}
%%\begin{equation}\label{CK-June20-210}
%%\int_0^s \int_{\pa \Om}|\mathcal  G(P,Z,\tau) | dZd\tau \leq ct^{\de} (t -s)^{\frac12 \al}.
%%\end{equation}
%%
%%\begin{align*}
%%\int_s^t \int_{\pa \Om} |G(P,Z,t-\tau) | dZd\tau & \leq ct^{\de} (t -s)^{\frac12 \al},\\
%%   \int_0^s \int_{\pa \Om}| G(P,Z,t-\tau) - G(P,Z,s -\tau)|   dZd\tau & \leq ct^{\de} (t -s)^{\frac12 \al}.
%%\end{align*}
%%
%\end{lemma}
%
%\begin{proof}

It remains to prove \eqref{CK-June20-200}. Due the inequality
\eqref{boundary-boundary1}, we have
\[
\int_s^t \int_{\pa \Om} |\mathcal G(P,Z, \tau)| dZd\tau \leq  c
\int_s^t \int_{\pa \Om} \frac{1 }{\tau^{\frac{1+\lambda }2}
(|P-Z|^2 +  \tau )^{\frac{n -2\lambda }2} }    dZd\tau.
\]
%\[
%\leq  c \int_0^t \int_{|P -Z| < r_0}\cdots dZd\tau +  c  \int_0^t
%\int_{|P -Z|> r_0}\cdots dZd\tau := I_{11} + I_{12},
%\]
%
%\begin{align*}
% \int_s^t \int_{\pa \Om} |G(P,Z,t-\tau)| dZd\tau
%  & \leq  c  \int_s^t \int_{\pa \Om}
%           \frac{1 }{(t-\tau)^{\frac{1+\lambda }2} (|P-Z|^2 +  (t-\tau) )^{\frac{n -2\lambda }2} }    dZd\tau \\
%  & \leq  c \int_s^t \int_{|P -Z| < r_0}
%           \frac{1 }{(t-\tau)^{\frac{1+\lambda }2} (|P-Z|^2 +  (t-\tau) )^{\frac{n -2\lambda }2} }    dZd\tau \\
%   &  +  c  \int_s^t \int_{|P -Z|> r_0}
%           \frac{1 }{(t-\tau)^{\frac{1+\lambda }2} (|P-Z|^2 +  (t-\tau) )^{\frac{n -2\lambda }2} }    dZd\tau \\
%  & = I_{11} + I_{12},
%  \end{align*}
Since $\Om$ is a bounded domain, taking $ \frac12 < \la  <1$, we can
see that $\int_{\pa \Om}\frac{1 }{ (|P-Z|^2 + \tau )^{\frac{n
-2\lambda }2} } dZ<c_\la$, and thus, the estimate
\eqref{CK-June20-200} is immediate.
\end{proof}

\begin{lemma}\label{lemma1}
Let $U$ be the hydrodynamical potential in \eqref{CK1107-2}.
Suppose that $\Phi\in L^\infty(\pa \Om; C^{\frac12 \al}[0, T])$ and $\Phi|_{t =0} =0$.
Then, there is $c > 0$ such that
\begin{align}\label{inequality508-1}
\| U[\Phi] \|_{\calC^{\al, \frac12 \al} (\pa \Om \times [0, T])} &\leq c
T^{\de}\| \Phi\|_{L^\infty(\pa \Om; \calC^{\frac12 \al}[0, T])},
\end{align}
\begin{align}\label{inequality508-2}
\| U[\Phi] \|_{ \dot \calC_{{\mathcal D}_{\eta}} (\pa \Om ; \dot
C^{\frac12 \al}[0, T])} & \leq c T^{\de}\| \Phi\|_{L^\infty(\pa
\Om; \calC^{\frac{1}2 \al} [0, T])},
\end{align}
where $\de > 0$ is the number in Lemma \ref{lemma508}.
\end{lemma}

\begin{proof}
We first note, due to \eqref{CK-June20-200} in Lemma \ref{lemma508},
that
\begin{align}\label{1209linfty}
\| U[\Phi]\|_{L^\infty(\pa \Om \times (0, T))} 
\leq \| \Phi\|_{L^\infty(\pa \Om \times (0, T))} \int_0^T |{\mathcal G}(\cdot, Q, t)|dQdt
\leq c T^\de \| \Phi\|_{L^\infty(\pa \Om \times (0, T))}.
\end{align}

We first note, due to \eqref{CK-August} in Lemma \ref{lemma508},
that
\[
\abs{U[\Phi] (P,t) - U[\Phi] (Q,t) }\leq c \int_0^t \int_{\pa \Om}
|\mathcal G(P,Z,t-\tau) - \mathcal G(Q,Z,t-\tau)| |\Phi(Z,\tau)|
dZd\tau
\]
\[
\leq c  \| \Phi\|_{L^\infty(\pa \Om \times [0, T])}\int_0^t
\int_{\pa \Om} |\mathcal G(P,Z,\tau) - \mathcal G(Q,Z,\tau)| dZd\tau
\]
\[
\leq c \| \Phi\|_{L^\infty(\pa \Om \times [0, T])} |P - Q|^\al
t^{\de}.
\]
Therefore, we obtain
\begin{align}\label{Gcal}
\|U[\Phi] \|_{L^\infty(0, T; \dot \calC^\al(\pa \Om))} \leq c T^{\de}
\| \Phi\|_{L^\infty(\pa \Om \times [0, T])}.
\end{align}
We also obtain for $s < t$
\begin{align*}
U[\Phi] (P,t) - U[\Phi] (P,s)= \int_s^t
\int_{\pa \Om} \mathcal G(P,Z,\tau)  \Phi(Z,t-\tau) dZd\tau \\
+   \int_0^s \int_{\pa \Om}  \mathcal G(P,Z,\tau) \big( \Phi(Z,
t-\tau) -\Phi(Z,s- \tau) \big) dZd\tau:=I_1+I_2.
\end{align*}
Since $\Phi(Z, 0) =0$, we have
\begin{align*}
|I_1| & \leq  \int_s^t
\int_{\pa \Om} |\mathcal G(P,Z,\tau)| |  \Phi(Z,t-\tau) -\Phi(Z, 0)|dZd\tau \\
& \leq  \| \Phi\|_{L^\infty(\pa \Om; \dot \calC^{\frac{\al}2} [0, T])} | t
-s|^{\frac{\al}2} \int_s^t
\int_{\pa \Om} |\mathcal G(P,Z,\tau)|  dZd\tau,\\
|I_2| & \leq  \int_0^s
\int_{\pa \Om} |\mathcal G(P,Z,\tau)| |  \Phi(Z,t-\tau) -\Phi(Z,s-\tau)|dZd\tau \\
& \leq  \| \Phi\|_{L^\infty(\pa \Om; \dot \calC^{\frac{\al}2} [0, T])} | t
-s|^{\frac{\al}2} \int_0^s \int_{\pa \Om} |\mathcal G(P,Z,\tau)|
dZd\tau.
\end{align*}
Hence,  via \eqref{CK-June20-200} in Lemma \ref{lemma508}, we
obtain
\begin{align}\label{1124}
\|U[\Phi] \|_{L^\infty(\pa \Om; \dot \calC^{\frac12 \al} [0, T])} \leq c
T^{\de} \| \Phi\|_{L^\infty(\pa \Om; \dot \calC^{\frac{\al}2} [0, T])}.
\end{align}
By \eqref{1209linfty}, \eqref{Gcal}  and \eqref{1124}, we completes the proof of \eqref{inequality508-1}.

For $s < t$, we
note that
\[
U[\Phi] (P,t) - U[\Phi] (P,s) - U[\Phi] (Q,t) + U[\Phi]
(Q,s)
\]
\[
=   \int_s^t \int_{\pa \Om}  \big( \mathcal G(P,Z,\tau)  -\mathcal
G(Q,Z, \tau) \big) \Phi(Z,t-\tau)  dZd\tau
\]
\[
+   \int_0^s \int_{\pa \Om}  \big( \mathcal G(P,Z,\tau) - \mathcal G(Q,Z, \tau) \big) \big( \Phi(Z, t-\tau) - \Phi(Z, s -\tau) \big) dZd\tau\\
= I_1 + I_2.
\]
Again using $\Phi(Z, 0) =0$, we obtain
\begin{align*}
|I_1| & \leq  \int_s^t
\int_{\pa \Om} |\mathcal G(P,Z,\tau) - \mathcal G(Q,Z, \tau)| |  \Phi(Z,t-\tau) -\Phi(Z, 0)|dZd\tau \\
& \leq  \| \Phi\|_{L^\infty(\pa \Om; \dot \calC^{\frac{\al}2} [0, T])} | t
-s|^{\frac{\al}2} \int_s^t
\int_{\pa \Om} |\mathcal G(P,Z,\tau) - \mathcal G(Q,Z, \tau)|  dZd\tau,\\
|I_2| & \leq  \int_0^s
\int_{\pa \Om} |\mathcal G(P,Z,\tau)  - \mathcal G(Q,Z, \tau)| |  \Phi(Z,t-\tau) -\Phi(Z,s-\tau)|dZd\tau \\
& \leq  \| \Phi\|_{L^\infty(\pa \Om; \dot \calC^{\frac{\al}2} [0, T])} | t
-s|^{\frac{\al}2} \int_0^s \int_{\pa \Om} |\mathcal G(P,Z,\tau)  -
\mathcal G(Q,Z, \tau)|  dZd\tau.
\end{align*}
By \eqref{CK-August} in Lemma \ref{lemma508}, we obtain
\eqref{inequality508-2}. This completes the proof.
\end{proof}

\begin{lemma}\label{lemma2}
Let $V$ be the electrostatic potential of a single layer given in
\eqref{CK1107-3}. Suppose that $\Psi \in C^{\al, \frac12 \al}(\pa
\Om \times [0, T]) \cap \dot \calC_{{\mathcal D}_{\eta}} (\pa \Om
; \dot \calC^{\frac12 \al}[0, T])$. Then, there is $c>0$ such that
\begin{align}\label{CK-1208-1}
\| \na_S V[\Psi ]  \|_{  \calC^{\al, \frac12 \al}(\pa \Om \times [0,
T])}   \leq c \big( \| \Psi\|_{   \calC^{\al, \frac{\al}2}(\pa
\Om\times[0, T])} + \|\Psi\|_{\dot \calC_{{\mathcal D}_{\eta}} (\pa \Om
; \dot \calC^{\frac12 \al}[0, T])} \big).
\end{align}
Furthermore, if  $\int_{\pa \Om} \Psi(Q,t) dQ =0$ for all $t\in [0,
T]$, then
\begin{equation}\label{CK-Aug16-300}
\|\Psi\|_{\calC^{\al, \frac{\al}2}(\pa\Om \times [0, T])} \leq c \| (I +
K^*)\Psi \|_{\calC^{\al, \frac{\al}2}(\pa\Om \times [0, T])},
\end{equation}
\begin{equation}\label{CK-Aug16-400}
\|\Psi\|_{ \dot  \calC_{D_{\eta}} (\pa
\Om; \dot \calC^{\frac12}([0, T]))}  \leq c \|
(I + K^*)[\Psi ]\|_{    \calC_{D_{\eta}} (\pa
\Om;  \calC^{\frac12}([0, T])}.
\end{equation}
\end{lemma}
\begin{proof}
Since $ \na_S V[\Psi]: \dot \calC^\al(\pa \Om) \ri \dot \calC^\al(\pa \Om)$
is bounded, it follows that
\begin{equation}\label{CK-1208-2}
 \| \na_S V[\Psi] \|_{L^\infty(0, T; \dot  \calC^{\al }(\pa \Om  ))} \leq c   \| \Psi \|_{L^\infty(0, T; \dot \calC^{\al }(\pa \Om ) )}.
\end{equation}
%We  estimate $ \| \na_S V[\Psi] \|_{C^{\al,\frac12 \al}(\pa \Om \times (0, T))}$.
%First, by the boundedness of $ \na_S V[\Psi]: \dot C^\al(\pa \Om)
%\ri \dot C^\al(\pa \Om)$, we have
%\begin{align*}
% \| \na_S V[\Psi] \|_{L^\infty(0, T; \dot C^{\al }(\pa \Om  ))} \leq c   \| \na_S V[\Psi] \|_{L^\infty(0, T; \dot C^{\al }(\pa \Om ) )}.
%\end{align*}
Next, we will show that
\begin{equation}\label{CK-June21-20}
\| \na_S V[\Psi ]  \|_{L^\infty(\pa \Om, \dot \calC^{\frac{\al}2} [0,
T])}   \leq c \big( \| \Psi\|_{L^\infty(\pa \Om; \dot \calC^{\frac12
\al}[0, T])} + \|\Psi\|_{\dot \calC_{{\mathcal D}_{\eta}} (\pa \Om ;
\dot \calC^{\frac12 \al}[0, T])} \big).
\end{equation}
Indeed, we compute
\begin{align*}
\na_S V[\Psi](P,t)  - &\na_S V[\Psi](P,s) = \sum_1^{n-1}T_l(P)\int_{\pa \Om} \frac{(P-Z)\cdot (T_l(P) -T_l(Z) ) }{|P-Z|^n} \big(\Psi(Z,t) - \Psi(Z,s) \big) dZ\\
 & \quad + \sum_1^{n-1}T_l(P) \int_{\pa \Om} \frac{(P-Z)\cdot  T_l(Z) }{|P-Z|^n} \big(\Psi(Z,t) - \Psi(Z,s) \big) dZ:=K_1+K_2,
\end{align*}
where $T_l(P), \,1 \leq l \leq n-1$ are tangential vector on $P \in \pa \Om$.
Since $|T_l(P) - T_l(Q) | \leq c | P-Q|$, one can easily see the first term $K_1$ is estimated as
\begin{equation}\label{CK-June21-30}
|K_1 |\le c \| \Psi\|_{L^\infty(\pa \Om:\dot  \calC^{\frac12 \al}[0, T])}
|t-s|^{\frac{\al}2}.
\end{equation}
Since $\int_{\pa \Om} \frac{(P-Z)\cdot  T_l(Z) }{|P-Z|^n} dZ =0$, the
second term $K_2$ can be estimated as follows:
\begin{equation*}
|K_2|=|\int_{\pa \Om} \frac{(P-Z)\cdot  T_l(Z) }{|P-Z|^n} \big(\Psi(Z,t)
- \Psi(Z,s) - \Psi(P,t) + \Psi(P,s) \big) dZ|
\end{equation*}
\begin{equation}\label{CK-June21-40}
\leq c|t-s|^{\frac12 \al} \int_{\pa \Om} \frac{\eta(|P -Z|)
}{|P-Z|^{n-1}} dZ \| \Psi\|_{\dot \calC_{{\mathcal D}_{\eta}} (\pa \Om ;
\dot \calC^{\frac12 \al}[0, T])} \leq  c  |t-s|^{\frac12 \al} \|
\Psi\|_{\dot \calC_{{\mathcal D}_{\eta}} (\pa \Om ; \dot \calC^{\frac12
\al}[0, T]) }.
\end{equation}

Adding up \eqref{CK-June21-30} and \eqref{CK-June21-40}, we deduce
\eqref{CK-June21-20}.

Using the same argument, we get
\begin{align}\label{CK-June21-100}
\| \na_S V[\Psi]\|_{L^\infty(\Om \times (0, T))} \leq c \big(\| \Psi\|_{L^\infty(\Om \times (0, T))} + \| \Psi\|_{\dot \calC_{{\mathcal D}_{\eta}} (\pa \Om ;
\dot \calC^{\frac12 \al}[0, T]) } \big).
\end{align}
Hence, from \eqref{CK-1208-2}, \eqref{CK-June21-20} and  \eqref{CK-June21-100}, we complete the proof of \eqref{CK-1208-1}.

It remains to show the estimates
\eqref{CK-Aug16-300}-\eqref{CK-Aug16-400}. Since $\Om$ is  a ${\mathcal C}^2$ domain,
%{\tt... We need to decide if domain is $\calC^2$ or smooth...},  %$K^* : C_{\si}^\al (\pa \Om) \ri C_{\si}^\al(\pa \Om)$ and
$K^* : \calC^\al_{\si,{\mathcal D}_{\eta}} (\pa \Om) \ri \calC^\al_{\si,{\mathcal
D}_{\eta}}(\pa \Om)$ is   compact operator, where
\begin{align*}
%C_{\si}^\al(\pa \Om) : =\{ \Psi \in C^\al(\pa \Om) \, | \, \int_{\pa \Om} \Psi =0 \}, \quad
\calC_{\si, D_\eta}^\al(\pa \Om) : =\{ \Psi \in \calC_{D_\eta}^\al(\pa \Om) \, | \, \int_{\pa \Om} \Psi =0 \}.
\end{align*}
Since %$I + K^* : C_{\si}^\al (\pa \Om) \ri C_{\si}^\al(\pa \Om)$ and
$ I + K^* : \calC^\al_{\si, D_{\eta}} (\pa \Om) \ri \calC^\al_{\si, D_{\eta}}(\pa \Om)$ is injective, by Fredhlom operator theory,
%$I + K^* : C_{\si}^\al (\pa \Om) \ri C_{\si}^\al(\pa \Om)$ and
$ I + K^* : \calC^\al_{\si, D_{\eta}} (\pa \Om) \ri \calC^\al_{\si, D_{\eta}}(\pa
\Om)$ is bijective operator. Using the same argument, we note that
$I + K^* : \calC^\al_{\si} (\pa \Om) \ri \calC_{\si}^\al(\pa \Om)$ and  $I + K^* :
L^\infty_{\si} (\pa \Om) \ri L^\infty_\si(\pa \Om)$ are bijective
operators. Hence, for $\Psi$ satisfying $\int_{\pa \Om} \Psi =0$, we have
\begin{align*}
\|\Psi\|_{\calC^{\al }(\pa\Om) } \leq c \| (I +
K^*)\Psi \|_{\calC^{\al }(\pa\Om  )}, \\
\|\Psi \|_{\calC^\al_{{\mathcal D}_{\eta}} (\pa \Om)} \leq c \|
(I + K^*)[\Psi ]\|_{  \calC^\al_{{\mathcal D}_{\eta}} (\pa
\Om)},\\
\|\Psi\|_{L^\infty(\pa\Om )} \leq c \| (I +
K^*)[\Psi] \|_{L^\infty (\pa\Om  )}.
\end{align*}
In particular, for $s,t \in [0, T]$, we obtain
\begin{align*}
\|\Psi(t) - \Psi(s)\|_{L^\infty(\pa\Om )} \leq c \| (I +
K^*)[\Psi(t) - \Psi(s)] \|_{L^\infty (\pa\Om  )},\\
\|\Psi(t) - \Psi(s) \|_{\dot \calC^\al_{{\mathcal D}_{\eta}} (\pa
\Om)} \leq c \| (I + K^*)[\Psi(t) - \Psi(s) ]\|_{
\calC^\al_{{\mathcal D}_{\eta}} (\pa \Om)}.
\end{align*}
The above estimates immediately imply
\eqref{CK-Aug16-300}-\eqref{CK-Aug16-400}. This completes the proof.
\end{proof}

By Lemma \ref{lemma1}, we have
\begin{align}\label{1124-1}
\|U \|_{\calC_0^{\al, \frac12 \al}(\pa \Om \times [0, T]) \ri \calC_0^{\al,
\frac12 \al}(\pa \Om \times [0, T])}, \,\, \|U \|_{\calC_{{\mathcal D}_\eta 0}^{\al, \frac12 \al}(\pa \Om \times [0, T]) \ri \calC_{{\mathcal D}_\eta 0}^{\al,
\frac12 \al}(\pa \Om \times [0, T])} \leq c T^\de,
\end{align}
where
\begin{align*}
 \calC_0^{\al,
\frac12 \al}(\pa \Om \times [0, T]) = \{ f \in  \calC^{\al, \frac12
\al}(\pa \Om \times [0, T])\, \big| \, f|_{t = 0} \},\\
 \calC_{{\mathcal D}_\eta 0}^{\al,
\frac12 \al}(\pa \Om \times [0, T]) = \{ f \in  \calC_{{\mathcal D}_\eta }^{\al, \frac12
\al}(\pa \Om \times [0, T])\, \big| \, f|_{t = 0} \}.
\end{align*}
Hence, for $cT^\de < 1$, the operators $I + U_{tan}: \calC_0^{\al, \frac12
\al}(\pa \Om \times [0, T]) \ri \calC_0^{\al, \frac12 \al}(\pa \Om \times
[0, T])$ and $I + U_{tan}: \calC_{{\mathcal D}_\eta 0}^{\al,
\frac12 \al}(\pa \Om \times [0, T]) \ri  \calC_{{\mathcal D}_\eta 0}^{\al,
\frac12 \al}(\pa \Om \times [0, T])$ are bijective. Therefore, we have that
%\begin{lemma}\label{lemma3}
%\begin{itemize}
%\item
%%\begin{align*}
%%\| K^* \Psi\|_{C^{\al, \frac{\al}2}(\pa\Om \times (0, T))} \leq c \|
%%\Psi\|_{C^{\al, \frac{\al}2}(\pa\Om \times (0, T))}.
%%\end{align*}
%For $\Psi \in C^{\al, \frac12 \al}(\pa \Om \times (0, T)) $ with $\int_{\pa \Om} \Psi =0$,
%\begin{align*}
%\| \Psi\|_{C^{\al, \frac{\al}2}(\pa\Om \times (0, T))} &\leq c \| (I
%+ K^*)\Psi \|_{C^{\al, \frac{\al}2}(\pa\Om \times (0, T))},\\
%\|\Psi(t) - \Psi(s)\|_{C_{{\mathcal D}_{\eta}} (\pa \Om)} & \leq c
%\|  (I + K^*)[\Phi(t) - \Phi(s)]\|_{  C_{{\mathcal D}_{\eta}} (\pa
%\Om)}.
%\end{align*}
%\item
there is $T_0 > 0$ such that for $T \leq T_0$ and $\Psi$ satisfying $\Psi|_{t =0} =0$,
\begin{align}\label{lemma3}
\notag\|\Phi \|_{ \calC^{\al,\frac12 \al} (\pa \Om \times [0, T])} & \leq c \|
\Phi +  U_{tan} [\Phi ]\|_{ \calC^{\al,\frac12 \al} (\pa \Om \times [0,
T])},\\
\|\Phi \|_{\calC_{{\mathcal D}_\eta }^{\al,
\frac12 \al}(\pa \Om \times [0, T])} & \leq c \|
\Phi +  U_{tan} [\Phi ]\|_{ \calC_{{\mathcal D}_\eta }^{\al,
\frac12 \al}(\pa \Om \times [0, T])}.
\end{align}

%\end{itemize}
%\end{lemma}
%
%\begin{proof}
%
%\end{proof}

\begin{proposition}\label{prop1}
Let $T<\infty$. Suppose that $g \in \calC_0^{\al,\frac{\al}2} (\pa \Om
\times [0, T])$ with
$g \cdot {\bf n}  \in \dot \calC_{{\mathcal D}_{\eta}}(\pa \Om; \dot \calC^{\frac{\al}2} ([0, T]) $,  and  satisfying the condition $
\int_{\pa \Om} g(Q,t) \cdot {\bf n}(Q) dQ =0 ,\quad \forall t \in (0, T)$,
 then, the system \eqref{boundarysystem} has
a unique solution $\Phi,  \, \Psi  \in \calC_0^{\al, \frac{\al}2}(\pa \Om
\times [0, T])$, $ \Psi\in  \dot \calC^{\al}_{D_\eta}(\pa \Om ; \dot \calC^{\frac12 \al}[0, T])$ with the conditions \eqref{compatible-condition}.
%\begin{align*}
%\Phi \cdot {\bf n} =0, \qquad \int_{\pa \Om} \Psi (Q,t) dQ =0
%,\qquad \forall t \in [0, T].
%\end{align*}
Furthermore, $(\Phi, \Psi)$ satisfies the following inequality:
\begin{align*}
\| \Phi\|_{\calC^{\al, \frac{\al}2}(\pa \Om \times [0, T])} + \| \Psi\|_{\calC^{\al, \frac{\al}2}(\pa \Om \times [0, T])} + \| \Psi
\|_{\dot \calC^{\al}_{D_\eta}(\pa \Om ; \dot \calC^{\frac12 \al}[0, T])} \leq
c\big( \| g\|_{ \calC^{\al,\frac{\al}2} (\pa \Om \times [0, T])} + \|
g\cdot {\bf n}\|_{\dot \calC_{D_\eta} (\pa \Om; \dot \calC^{\frac12 \al}[0, T] ) } \big),
\end{align*}
where $c=c(T)$.
\end{proposition}

\begin{proof}
Let $T \leq T_0$, where $T_0$ is  a constant defined   \eqref{lemma3}.
By  \eqref{lemma3}, we  solve the following
equation:
\[
\Phi_1 + U_{tan}[\Phi_1] =g_{tan}
\]
and $\Phi_1$ satisfies
\begin{align*}
 \Phi_1 \in \calC_0^{\al, \frac12 \al}(\pa \Om \times [0, T]), \qquad \| \Phi_1\|_{\calC^{\al,\frac{\al}2}(\pa \Om \times [0, T])} \leq c\|
g_{tan}\|_{\calC^{\al,\frac{\al}2}(\pa \Om \times [0, T])}
\end{align*}
and  $\Phi_1 \cdot {\bf n} =0$.  Note that since ${\mathcal U}[\Phi_1]$ is divergence free,  $\int_{\pa \Om} {\bf n}
\cdot U [\Phi_1] =0$.  In the proof of Lemma \ref{lemma2}, there is $\Psi_1$ we
solve
\begin{align*}%\label{1124-2}
\Psi_1 + K^* [\Psi_1] = g\cdot {\bf n} - {\bf n} \cdot U
[\Phi_1]
\end{align*}
and by   \eqref{lemma3} and Lemma \ref{lemma2}, $\Psi_1 \in \calC_0^{\al, \frac12 \al}(\pa \Om \times [0, T])$ satisfies
\[
\| \Psi_1\|_{\calC^{\al,\frac{\al}2}(\pa \Om \times [0, T])}\leq
c\big(\| g\cdot {\bf n}\|_{\calC^{\al,\frac{\al}2}(\pa \Om \times [0,
T])} + T^\de \| \Phi_1\|_{\calC^{\al,\frac{\al}2}(\pa \Om \times [0, T])}
\big),
\]
\[
\| \Psi_1\|_{\dot \calC^{\frac{\al}2}( [0, T]; \dot \calC_{{\mathcal
D}_{\eta}}(\pa \Om ))}\leq  c \big(\| g\cdot {\bf n}\|_{\dot
\calC^{\frac{\al}2}( [0, T]; \dot \calC_{{\mathcal D}_{\eta}}(\pa \Om ))} +
T^\de \| \Phi_1\|_{\calC^{\al,\frac{\al}2}(\pa \Om \times [0, T])} \big).
\]
%where we used Lemma \ref{lemma1} and Lemma \ref{lemma2}.
%\begin{align*}
%\Psi_1 + K^* [\Psi_1]
%&= g\cdot {\bf n} - {\bf n} \cdot U [\Phi_1],\\
%\| \Psi_1\|_{C^{\al,\frac{\al}2}(\pa \Om \times (0, T))}
%&\leq \| g\cdot {\bf n} - {\bf n} \cdot U [\Phi_1]\|_{C^{\al,\frac{\al}2}(\pa \Om \times (0, T))}\\
%&\leq \| g\cdot {\bf n}\|_{C^{\al,\frac{\al}2}(\pa \Om \times (0, T))}
%+  \| {\bf n} \cdot U [\Phi_1]\|_{C^{\al,\frac{\al}2}(\pa \Om \times (0, T))},\\
%\| \Psi_1\|_{C^{\frac{\al}2}( 0, T; C^\ep(\pa \Om ))}
%& \leq \| g\cdot {\bf n}\|_{C^{\frac{\al}2}( 0, T; C^\ep(\pa \Om ))}
%+ \| {\bf n} \cdot U [\Phi_1]\|_{C^{\frac{\al}2}( 0, T; C^\ep(\pa \Om ))} \\
%& \leq  \| g\cdot {\bf n}\|_{C^{\frac{\al}2}( 0, T; C^\ep(\pa \Om
%))} + T^a \|  \Phi_1\|_{L^\infty(\pa \Om)}.
%\end{align*}
Iteratively, we define $(\Phi_{m+1}, \Psi_{m+1})$ for any
$m=1,2,\cdots$ as follows:
\begin{align}\label{iterate}
\notag \Phi_{m+1} + U_{tan}[\Phi_{m+1}] =g_{tan} - \na_S V[\Psi_m],\\
\Psi_{m+1} + K^* [\Psi_{m+1}]= g\cdot {\bf n} - {\bf n} \cdot
U[\Phi_{m+1}].
\end{align}
We then note that $ \Phi_{m+1}$ satisfies
\begin{align}\label{estimatephim+1}
\notag \| \Phi_{m+1}\|_{\calC^{\al,\frac{\al}2}(\pa \Om \times [0, T])} \leq c
\big( \| g_{tan}\|_{\calC^{\al,\frac{\al}2}(\pa \Om \times [0, T])} + \|
\na_S V[\Psi_m] \|_{\calC^{\al,\frac{\al}2}(\pa \Om \times [0, T])}
\big)\\
%\end{align*}
%\begin{equation}\label{CK-June21-700}
\leq  c \big( \| g_{tan}\|_{\calC^{\al,\frac{\al}2}(\pa \Om \times [0,
T])} +  \| \Psi_m \|_{\calC^{\al,\frac{\al}2}(\pa \Om \times [0, T])} +
\|\Psi_m \|_{\dot \calC^{\frac{\al}2}( 0, T;  \dot \calC_{{\mathcal D}_{\eta}} (\pa
\Om ))} \big),
\end{align}
where we used Lemma \ref{lemma2}.  On the other hand, for $\Psi_{m+1}$ we observe that
\begin{align}\label{estimatepsim+1}
\notag\| \Psi_{m+1}\|_{\calC^{\al,\frac{\al}2}(\pa \Om \times [0, T])}
\leq c\big( \| g\cdot {\bf n}\|_{\calC^{\al,\frac{\al}2}(\pa \Om \times [0, T])}
+  T^\de\|  \Phi_{m+1}\|_{\calC^{\al,\frac{\al}2}(\pa \Om \times [0, T])} \big),\\
\| \Psi_{m+1}\|_{ \dot \calC^{\frac{\al}2}( 0, T;\dot  \calC_{{\mathcal
D}_{\eta}}(\pa \Om ))}
 \leq  c\big(\| g\cdot {\bf n}\|_{ \dot \calC^{\frac{\al}2}( 0, T; \dot \calC_{{\mathcal D}_{\eta}}(\pa \Om ))}
 +T^\de \|  \Phi_{m+1}\|_{\calC^{\al,\frac{\al}2}(\pa \Om \times [0, T])} \big),
\end{align}
where we used  \eqref{lemma3}  and Lemma \ref{lemma1}.

For uniformly
convergence, we denote $\phi_m = \Phi_{m+1} - \Phi_m$ and $\psi_m =
\Psi_{m+1} - \Psi_m$. Then, $(\phi_{m}, \psi_m)$ solves
\[
\phi_{m+1} + U_{tan}[\phi_{m+1}] =- \na_S V[\psi_m],
\]
\[
\psi_{m+1} + K^* [\psi_{m+1}]=- {\bf n} \cdot U[\phi_{m+1}]
\]
and it satisfies
\begin{align*}
\| \phi_{m+1}\|_{\calC^{\al,\frac{\al}2}(\pa \Om \times [0, T])}
&\leq  c\big( \|\psi_m \|_{\calC^{\al, \frac{\al}2}( \pa \Om  \times [0, T])}
+ \|\psi_m \|_{\dot \calC^{\frac{\al}2}( [0, T]; \dot \calC_{{\mathcal D}_{\eta}} (\pa \Om ))} \big),\\
\| \psi_{m+1}\|_{\calC^{\al,\frac{\al}2}(\pa \Om \times [0, T])}
&\leq  c T^\de \|  \phi_{m+1}\|_{\calC^{\al,\frac{\al}2}(\pa \Om \times [0, T])},\\
\| \psi_{m+1}\|_{ \dot \calC^{\frac{\al}2}( [0, T]; \dot \calC_{{\mathcal
D}_{\eta}}(\pa \Om ))} & \leq   cT^\de \|
\phi_{m+1}\|_{\calC^{\al,\frac{\al}2}(\pa \Om \times [0, T])}.
\end{align*}
Hence, we obtain
\begin{align*}
\| \phi_{m+1}\|_{\calC^{\al,\frac{\al}2}(\pa \Om \times [0, T])}
 &\leq   cT^\de \|  \phi_m\|_{\calC^{\al,\frac{\al}2}(\pa \Om \times [0, T])},\\
 \| \psi_{m+1}\|_{\calC^{\al,\frac{\al}2}(\pa \Om \times [0, T])} + \| \psi_{m+1}\|_{\dot \calC^{\frac{\al}2}( [0, T]; \dot \calC_{{\mathcal D}_{\eta}}(\pa \Om ))}
&\leq  c  T^\de \big(\|  \psi_{m}\|_{\calC^{\al,\frac{\al}2}(\pa \Om
\times [0, T])} +  \|  \psi_{m}\|_{\calC^{\al,\frac{\al}2}(\pa \Om
\times [0, T])} \big).
\end{align*}
This implies that there is $T^* > 0$ with $T^*\le T_0$ such that
$\{\Phi_m, \Psi_m\}$ converges for some $(\Phi^1, \Psi^1) \in
\calC_0^{\al, \frac{\al}2}(\pa \Om \times [0, T^*]) \times \calC^\al_{
D_\eta, 0}(\pa \Om; \calC^{\frac12 \al}([0, T^*]))$. From \eqref{iterate}, \eqref{estimatephim+1} and \eqref{estimatepsim+1}, $(\Psi^1, \Psi^1)$ satisfy
\begin{align*}
&\Phi^1 + U_{tan}[\Phi^1]  +\na_S V[\Psi^1] =g_{tan},\\
&\Psi^1 + K^* [\Psi^1] + {\bf n} \cdot
U[\Phi^1] = g\cdot {\bf n}\quad \mbox{in} \quad \Om \times (0, T^*) 
\end{align*}
and 
\begin{align*} 
 %\| \Psi_{m+1}\|_{\calC^{\al,\frac{\al}2}(\pa \Om \times [0, T])}
%\leq c\big( \| g\cdot {\bf n}\|_{\calC^{\al,\frac{\al}2}(\pa \Om \times [0, T])}
%+  T^\de\|  \Phi_{m+1}\|_{\calC^{\al,\frac{\al}2}(\pa \Om \times [0, T])} \big),\\
\| \Phi^1\|_{\calC^{\al,\frac{\al}2}(\pa \Om \times [0, T^*])}  + \| \Psi^1\|_{\calC^{\al,\frac{\al}2}(\pa \Om \times [0, T^*])}+ \| \Psi^1\|_{ \dot \calC^{\frac{\al}2}( 0, T^*;\dot  \calC_{{\mathcal
D}_{\eta}}(\pa \Om ))}\\
 \leq  c(T^*)\big(  \| g \|_{\calC^{\al,\frac{\al}2}(\pa \Om \times [0, T^*])}  + \| g\cdot {\bf n}\|_{ \dot \calC^{\frac{\al}2}( 0, T^*; \dot \calC_{{\mathcal D}_{\eta}}(\pa \Om ))}
% +T^\de \|  \Phi_{m+1}\|_{\calC^{\al,\frac{\al}2}(\pa \Om \times [0, T])} 
\big).
\end{align*}

To construct $(\Phi, \Psi)$ up to any time $T$, we introduce $h$, which is given as
\begin{align*}
h(P,t) = \Phi^1(P,T^*) + \Psi^1(P,T^*) {\bf n}(P) + \int_0^{T^*}
\int_{\pa \Om} {\mathcal G}(t + T^* -\tau) \Phi^1(Q, \tau) dQd\tau + \na
V[\Psi^1 (T^*)](P)
\end{align*}
for $t \in [0, T^*]$ such that $h(P, 0) = g(P, T^*)$. For $t \in [0,
T^*]$, let us $g^1(P,t) = g(P,T^* + t) - h(P,t)$ such that $g^1 \in
\calC_0^{\al, \frac{\al}2}(\pa \Om \times [0, T^*])$ and $ g^1 \cdot {\bf
n} \in  \dot \calC_{ D_\eta}(\pa \Om; \dot \calC^{\frac12 \al}([0, T^*]))$.
By above argument, there is  $(\Phi^2, \Psi^2) \in \calC_0^{\al,
\frac{\al}2}(\pa \Om \times [0, T^*]) \times \calC^\al_{0, D_\eta}(\pa
\Om; \calC^{\frac12 \al}([0, T^*]))$ such that
\begin{align*}
\Phi^2 + U_{tan}[\Phi^2] + \na_S V[\Psi^2] = h^1_{tan},\\
\Psi^2 + K^* [\Psi^2] + U[\Phi^2] \cdot {\bf n} = h^1 \cdot {\bf n}.
\end{align*}
We define $(\Phi, \Psi)$ by
\begin{align*}
\Phi(P,t) = \left\{\begin{array}{ll}
\Phi^1(P,t), & \qquad 0 \leq t  \leq T^*,\\
\Phi^2 (P,t-T^*) + \Phi^1(P,T^*), &\qquad T^* \leq t \leq 2T^*,
\end{array}
\right. \\
\Psi(P,t) = \left\{\begin{array}{ll}
\Psi^1(P,t), &\qquad 0 \leq t  \leq T^*,\\
\Psi^2 (P,t-T^*) + \Psi^1(P,T^*), & \qquad T^* \leq t \leq 2T^*.
\end{array}
\right.
\end{align*}
Then, we obtain $\Phi \in \calC_0^{\al, \frac{\al}2}(\pa \Om
\times [0, 2T^*]), \,\, \Psi  \in \calC_0^{\al, \frac{\al}2}(\pa \Om
\times [0, 2T^*])  \cap \calC^\al_{D_\eta}(\pa \Om; \calC^{\frac12
\al}([0, 2T^*]))$ and
\begin{align*}
\Phi + U_{tan}[\Phi] + \na_S V[\Psi] &= g_{tan}, \\
\Psi + K^* [\Psi] + U[\Phi] \cdot {\bf n} &= g \cdot {\bf n}, \quad \mbox{in} \quad  \pa \Om \times [0, 2T^*].
\end{align*}
We repeat the above procedure until we reach any given time $T$.
This completes the proof.
\end{proof}

\subsection{Global estimates}
Next, we estimate the global estimates of solution of Stokes equations.

\begin{proposition}\label{CK-June21-1000}
Let $\Phi \in \calC_0^{\al, \frac{\al}2}(\pa \Om \times [0, T])$   and $\Psi \in
\calC_0^{\al,\frac12 \al}(\pa \Om \times [0, T])\cap \dot \calC^{\frac{\al}2}
(0, T; \dot \calC_{{\mathcal D}_{\eta}}(\pa \Om))$ satisfying \eqref{compatible-condition}. Suppose that
${\mathcal U}[\Phi]$ and ${\mathcal V}[\Psi]$ are defined in \eqref{potential-U} and
\eqref{potential-V}.
%\begin{align*}
%U[\Phi](x,t) = \int_0^t \int_{\pa \Om} \mathcal G(x, Z, t-s) \Phi(Z,
%s) dZds.
%\end{align*}
Then,
\begin{equation}\label{CK-June21-8000}
\|{\mathcal U}[\Phi]\|_{  \calC^{ \al,\frac12 \al}(  \bar \Om \times [0, T]))} \leq c \| \Phi
\|_{\calC^{\al, \frac{\al}2}(\pa \Om \times [0, T])},
\end{equation}
\begin{equation}\label{CK-June21-820}
\| \nabla {\mathcal V}[\Psi]\|_{  \calC^{\al,\frac{\al}2}( \Om \times [0, T])} \leq
c \bke{ \| \Psi\|_{ \calC^{\al,\frac12 \al}(\pa \Om \times [0, T])} + \|
\Psi\|_{\dot \calC^{\frac{\al}2} (0, T; \dot \calC_{{\mathcal D}_{\eta}}(\pa
\Om))} }.
\end{equation}
\end{proposition}

\begin{proof}
Let $x \in \Om$. Choose $P_x \in \pa \Om$ satisfying $\de(x) = |x
- P_x|$.  Using the rotation and translation, we may assume that $x = (0, x_n)$,  $P_x =0$ and $\de(x) = x_n$.
We recall that
\begin{align}\label{509inequality-2}
\notag |\mathcal G(x,Q,t)|
& \leq   c_\la \frac{\de(x)^\la}{  t^{\frac{1+\la}2} (|x-Q|^2 +t )^{\frac{n}2  }}, \quad 0 < \la <1\\
|D_x \mathcal G(x,Q,t)| &\leq \frac{c}{t^\frac12 (x_n^2
+t)^{\frac12} (|x -Q|^2 +t)^{\frac{n}2}}
\end{align}
(see
\cite{K} and \cite{So}).  From the first inequality of \eqref{509inequality-2}, we have
\begin{align*}
| U[\Phi](x,t)| 
&\leq c \| \Phi\|_{L^\infty(\Om \times (0, T))} \int_0^t \int_{\pa \Om}  \frac{\de(x)^\la}{  \tau^{\frac{1+\la}2} (|x-Q|^2 +\tau )^{\frac{n}2  }} dQd\tau\\
&\leq c \| \Phi\|_{L^\infty(\Om \times (0, T))}.
\end{align*}

To complete the proof of \eqref{CK-June21-8000}, we first show that
\begin{align}\label{weight}
\sup_{(x,t) \in \overline \Om \times [0, T]} \de^{1 -\al} (x) \abs{\na_x
{\mathcal U}(x,t)} \leq c  \| \Phi \|_{\calC^{\al, \frac{\al}2}(\pa \Om \times [0,T
]))}.
\end{align}

It is known (see e.g. \cite[Theorem 4.1]{JK} and  \cite[Theorem
1.4]{FMM}) that the estimate \eqref{weight} implies
\begin{equation}\label{CK-June21-800}
\|{\mathcal U}[\Phi]\|_{L^\infty(0, T; \dot \calC^{ \al}( \overline \Om))} \leq c \| \Phi
\|_{\calC^{\al, \frac{\al}2}(\pa \Om \times [0, T])}.
\end{equation}
%(See Theorem 4.1 in \cite{JK} and Theorem 1.4 in \cite{FMM}.)

 Since $\Phi(P, 0) =0$,  we compute
\begin{align*}
D_{x} {\mathcal U}[\Phi](x,t)&=\int_0^t \int_{\pa \Om}   \na_x \mathcal G(x, Q,
t-s) \big( \Phi(Q, \tau) -\Phi(P_x, \tau) \big) dQd\tau\\
&\qquad + \int_{-\infty}^t \int_{\pa \Om} \na_x \mathcal G(x, Q, t-s) \big(
\Phi(P_x, \tau) - \Phi(P_x, t) \big) dQd\tau\\
&\qquad + \Phi(P_x, t) \int_{-\infty}^t \int_{\pa \Om} \na_x \mathcal G(x,
Q, t-s) dQd\tau: = I_1 + I_2 + I_3.
\end{align*}
Referring to \cite[(4.5)]{KS1}, we note that there exists a small
$\ep> 0$ such that
\begin{align*}
|I_3| \leq  c_\ep x_n^{-\ep}\| \Phi \|_{L^\infty(0, T; \pa \Om)}.
\end{align*}
Noting that for $Q \in \pa \Om$, $|x-Q|^2 \approx x_n^2 + |P_x -Q|^2$ and using
\eqref{509inequality-2}, we estimate $I_1$ and $I_2$ as
\begin{align*}
|I_1| & \leq c \| \Phi \|_{L^\infty(0, T; \dot \calC^\al (\pa \Om))}
\int_0^t   \frac{1}{(t -\tau)^\frac12 (x_n^2  + t
-\tau)^{\frac12}}
      \int_{\pa \Om} \frac{|P_x - Q|^\al}{(|P_x -Q|^2+ x_n^2 + t -\tau)^{\frac{n}2}} dQd\tau\\
& \leq c \| \Phi \|_{L^\infty(0, T; \dot \calC^\al (\pa \Om))} \int_0^t
\frac{1}{\tau^\frac12 (x_n^2  + \tau)^{1 - \frac{\al}2}}d\tau\leq
c \| \Phi \|_{L^\infty(0, T; \dot \calC^\al (\pa \Om))} x_n^{-1
+\al},
\end{align*}
\begin{align*}
|I_2|  & \leq c \| \Phi \|_{L^\infty(\pa \Om; \dot \calC^{\frac{\al}2}
[0, T])} \int_0^t   \frac{1}{(t -\tau)^\frac12 (x_n^2  + t
-\tau)^{\frac12}}\int_{\pa \Om} \frac{(t-\tau)^{\frac{\al}2}}{(|P_x -Q|^2+ x_n^2 + t -\tau)^{\frac{n}2}} dQd\tau\\
& \leq c \| \Phi \|_{L^\infty(\pa \Om; \dot \calC^{\frac{\al}2} [0, T])}
\int_0^t \frac{1}{(t -\tau)^{\frac12 -\frac{\al}2} (x_n^2  + t
-\tau)} d\tau \leq c \| \Phi \|_{L^\infty(\pa \Om; \dot
\calC^{\frac{\al}2} [0, T])}x_n^{-1 +\al}.
\end{align*}
Summing up above estimates, we obtain \eqref{weight}.

%{\color{red}{ In the case of ${\rm dist}\, (x, \pa \Om) \geq \de_0$
%for some  $\de_0 > 0$ using  the estimate \eqref{509inequality-2} we
%also obtain   \eqref{weight}. }}

Let $h > 0$ and we compute
\[
{\mathcal U}[\Phi](x,t+h) -{\mathcal U}[\Phi](x,t)
\]
\[
= \int_0^{t+h} \int_{\pa \Om}  {\mathcal G}(x,Q,t+h-\tau) \Phi(Q,\tau) dQd\tau
- \int_0^t \int_{\pa \Om}    {\mathcal G}(x,Q, t-\tau)    \Phi(Q,\tau) dQd\tau
\]
\[
 = \int_0^{t} \int_{\pa \Om}    {\mathcal G}(x,Q, \tau)  \big(
\Phi(Q, t+h -\tau ) -\Phi(Q,t-\tau) \big) dQd\tau
      +\int_t^{t+h}\int_{\pa \Om}  {\mathcal G}(x,Q, \tau)  \Phi(Q,t+h-\tau) dQd\tau
\]
\[
: = I_1 + I_2.
\]
By \eqref{CK-June20-200}, we have
\begin{align*}
|I_1| & \leq   h^{\frac12 \al} \| \Phi\|_{L^\infty(\pa \Om; \dot \calC^{\frac12 \al}[0,T])} \int_0^{t} \int_{\pa \Om}   | {\mathcal G}(x,Q, \tau) |  dQd\tau\\
& \leq  c h^{\frac12 \al} t^\de \| \Phi\|_{L^\infty(\pa \Om; \dot \calC^{\frac12 \al}[0,T])},\\
|I_2| & \leq  |\int_t^{t+h}\int_{\pa \Om}  {\mathcal G}(x,Q, \tau)  \big( \Phi(Q,t+h-\tau) - \Phi(Q, 0 )\big) dQd\tau|\\
& \leq   h^{\frac12 \al} \| \Phi\|_{L^\infty(\pa \Om; \dot \calC^{\frac12 \al}[0,T])} \int_t^{t+h} \int_{\pa \Om}   | {\mathcal G}(x,Q, \tau) |  dQd\tau\\
& \leq  c h^{\frac12 \al} (t+ h)^\de \| \Phi\|_{L^\infty(\pa \Om; \dot \calC^{\frac12 \al}[0,T])}.
\end{align*}
Hence, we have
\[
| {\mathcal U}[\Phi](x,t+h) -{\mathcal U}[\Phi](x,t) |\leq c \| \Phi \|_{L^\infty(\pa \Om; \dot \calC^{\frac{\al}2} [0, T])}
h^\frac{\al}2.
\]
Therefore, we obtain
\begin{equation}\label{CK-June21-810}
\|{\mathcal U}[\Phi]\|_{L^\infty( \Om; \dot \calC^{ \frac{\al}2}  [0, T])} \leq c \|
\Phi \|_{ L^\infty (\pa \Om; \dot \calC^{\frac{\al}2} [0, T])}.
\end{equation}
By \eqref{CK-June21-800} and \eqref{CK-June21-810}, we obtain \eqref{CK-June21-8000}.

It remains to show \eqref{CK-June21-820}. By the well known result of the harmonic function, boundedness of $K^*:L^\infty(0, T; \calC^\al(\pa \Om)) \ri L^\infty(0, T; \calC^\al(\pa \Om))      $ and \eqref{CK-1208-1}, we have
\begin{align*}
\| \nabla {\mathcal V}[\Psi]\|_{L^\infty(0, T; \calC^\al( \overline\Om))}  
&\le c\|  \nabla {\mathcal V}[\Psi]\|_{L^\infty(0, T; \calC^\al(\pa \Om))} \\
& \leq c\big( \|  \nabla_S V [\Psi]\|_{L^\infty(0, T; \calC^\al(\pa \Om))}  + \| (I + K^*)[\Psi]\|_{L^\infty(0, T; \calC^\al(\pa \Om))}   \big)\\
 &\le c \bke{ \| \Psi\|_{ \calC^{\al,\frac12 \al}(\pa \Om \times [0, T])} + \|
\Psi\|_{\dot \calC^{\frac{\al}2} (0, T; \dot \calC_{{\mathcal D}_{\eta}}(\pa
\Om))} }.
\end{align*}
For $t \neq s$,  we note that
\begin{align}\label{interior1031}
\na_x {\mathcal V}[\Psi](x,t) -\na_x {\mathcal V}[\Psi](x,s) = \nabla_x\int_{\pa \Om}
N(x -Q) \big( \Psi(Q,t) - \Psi(Q,s) \big) dQ.
\end{align}
By maximum principle of the harmonic function, we have
\begin{align*}
\abs{\na_x {\mathcal V}[\Psi](x,t) - \na_x{\mathcal  V}[\Psi](x,s)} &\leq  \sup_{Q \in \pa
\Om} \abs{\na_x {\mathcal V}[\Psi](Q,t) - \na_x {\mathcal V}[\Psi](Q,s)}.
%&= \sup_{Q \in \pa \Om} \sup_{t, s \in (0, T)} \frac{|D_x\phi(Q,t) - D_x \phi(Q,s)|}{|t-s|^{\frac12 \al}}
\end{align*}
We note that
\begin{align*}
\na_x {\mathcal V}[\Psi](Q,t) =\sum_1^{n-1}  \big( \na_x {\mathcal V}[\Psi] (Q,t) \cdot T_l (Q) \big)  T_l(Q) + \big( \na_x {\mathcal V}[\Psi] (Q,t) \cdot {\bf n}(Q) \big)  {\bf n}(Q).
\end{align*}
Here,
\[
\na_x {\mathcal V}[\Psi] (Q,t) \cdot T_l (Q) =  p.v \int_{\pa \Om} \frac{(Q -Z)
\cdot T_l(Q)}{|Q -Z|^n} \Psi(Z,t) dZ,
\]
\[
\na_x {\mathcal V}[\Psi] (Q,t) \cdot {\bf n} (Q) = \bke{ I + K^*
}[\Psi](Q,t).
\]
%
%\begin{align*}
% \na_x V[\Psi] (Q,t) \cdot T_l (Q) &=  p.v \int_{\pa \Om} \frac{(Q -z) \cdot T_l(Q)}{|Q -Z|^n} \Psi(Z,t) dZ,\\
% \na_x \psi (Q,t) \cdot {\bf n} (Q) &  = (-\frac12 I + K^* ) \Psi(Q,t),
%\end{align*}
Here $K^*[\Psi](Q,t)$ is defined  in \eqref{double-layer}.
%\begin{align*}
%K^*[\Psi](Q,t) = p.v \int_{\pa \Om} \frac{(Q -Z) \cdot {\bf
%n}(Q)}{|Q -Z|^n} \Psi(Q,t) dZ.
%\end{align*}
Since $\Om$ is a smooth domain, it is known that
\begin{align*}
|K^* [ \Psi(\cdot,t) - \Psi (\cdot,s)](Q )| & \leq c \| \Psi(\cdot, t) - \Psi
(\cdot, s)\|_{L^\infty (\pa \Om)}.
\end{align*}
Since $ \int_{\pa \Om} \frac{(Q -Z) \cdot T_l(Z)}{|Q -Z|^n} dZ =0$,
we have
\[
\bke{\na_x {\mathcal V} [\Psi](Q,t) - \na_x {\mathcal V}[\Psi] (Q,s)} \cdot T_l (Q)
\]
\[
=\int_{\pa \Om} \frac{(Q -Z) \cdot  (T_l(Q) - T_l(Z))}{|Q -Z|^n}
\big( \Psi(Z,t) -\Psi(Z,s) \big)   dZ
\]
\[
+  \int_{\pa \Om} \frac{(Q -Z) \cdot T_l(Z)}{|Q -Z|^n}
\big(\Psi(Z,t) -\Psi(Q,t) - \Psi(Z,s) +\Psi(Q,s)\big) dZ
\]
\[
\leq c  |t -s|^{\frac12 \al} \| \Psi \|_{ L^\infty(\pa \Om; \dot
\calC^{\frac12 \al}(\pa \Om))} + |t -s|^{\frac12 \al} \| \Psi \|_{\dot
\calC^{\frac12 \al}( 0, T; \dot \calC_{{\mathcal D}_{\eta}}(\pa \Om))} \int_{\pa
\Om} \frac{\eta(|Q -Z|)}{|Q -Z|^{n-1}}dZ
\]
\[
\leq  c  |t -s|^{\frac12 \al} \big ( \| \Psi \|_{ L^\infty(\pa \Om;
\dot \calC^{\frac12 \al}(\pa \Om))} +\| \Psi \|_{\dot \calC^{\frac12 \al}(
0, T; \dot \calC_{{\mathcal D}_{\eta}}(\pa \Om))} \big).
\]
Therefore, we obtain
\begin{align*}
\|\na_x V[\Psi]\|_{L^\infty(\Om; \dot \calC^{\frac12 \al}[0, T])} =
\sup_{t, s \in [0, T]} \frac{\|\na_x\psi(\cdot,t) - \na_x
\psi(\cdot,s)\|_{L^\infty(\pa \Om)}}{|t-s|^{\frac12 \al}}
%&\leq c  \sup_{t, s \in (0, T)} \frac{\|g (\cdot,t) - g(\cdot,s)\|_{C^\ep (\pa \Om)}}{|t-s|^{\frac12 \al}}\\
\leq  \| \Psi\|_{ \calC^{\frac12 \al}(0, T;  \calC_{{\mathcal D}_{\eta}}
(\pa \Om))}.
\end{align*}
This completes the proof.
\end{proof}

\begin{remark}\label{rem-4}
Let $\Om_\de = \{ x \in \Om \, | \, {\rm dist} \,\, (x, {\rm supp} \, \Psi)) \geq \de \}$ for $\de > 0$.
Then, from \eqref{interior1031}, we can obtain
\begin{align*}
\|\na_x V[\Psi]\|_{L^\infty(\Om_\de; \dot \calC^{\frac12 \al}[0, T])} \leq c_\de \| \Psi\|_{L^\infty(\pa \Om; \dot \calC^{\frac12 \al}[0, T])}.
\end{align*}
This implies
\begin{align*}
\| u\|_{\calC^{\al, \frac12 \al}(\bar \Om_\de \times [0, T])} \leq c_\de \big(  \| \Phi\|_{\calC^{\al, \frac12 \al}(\pa \Om \times [0, T])} + \| \Psi\|_{  \calC^{\al, \frac12 \al}(\pa \Om \times [0, T])}\big).
\end{align*}

\end{remark}

Summarizing the above results, we obtain the following.
\begin{theorem}\label{stokes-boundary}
Let $0 < \al<1 $. Let $g \in \calC^{\al, \frac12 \al}(\pa \Om \times (0,
T))$ such that   $g   \cdot {\bf n} \in \dot \calC^{\frac{\al}2} (0, T;
\dot \calC_{{\mathcal D}_{\eta}} (\pa \Om))$ and  $\int_{\pa \Om}g(Q,t)
{\bf n}(Q) dQ =0$.  Then, there exists a unique solution $u \in \calC^{\al,
\frac12 \al}(\Om \times [0, T])$ of the Stokes system
\[
u_t -\De u + \na P = 0,\quad {\rm div}\,u=0,\qquad \mbox{in }\,\,\Om
\times (0, T)
\]
\[
u|_{\pa \Om \times (0, T)} = g,\qquad u(x,0)=0.
\]
Furthermore, $u$ satisfies
\begin{equation}\label{CK-June21-1100}
\| u\|_{\calC^{\al, \frac12 \al}(\overline \Om \times [0, T])} \leq c \big( \|
u_0\|_{\calC^\al(\overline \Om)} + \|u_0\|_{\dot C_{{\mathcal D}_{\eta}}(\overline \Om)} +
\| g\|_{C^{\al, \frac12 \al}(\overline \Om \times [0, T])  } + \| g\cdot {\bf
n} \|_{ \dot \calC^{\frac{\al}2} (0, T; \dot C_{{\mathcal D}_{\eta}}
(\pa \Om))} \big).
\end{equation}
%
%\begin{align*}
%\left\{\begin{array}{l}
%u_t -\De u + \na P = 0, \quad \Om \times (0, T),\\
%div \, u =0,\\
%u|_{\pa \Om \times (0, T)} = g,\\
%u_{t=0} =0,\\
%\| u\|_{C^{\al, \frac12 \al}(\Om \times (0, T))} \leq c \big(  \|
%g\|_{C^{\al, \frac12 \al}( \Om \times (0, T))  } + \| g\cdot {\bf n}
%\|_{ C^{\frac{\al}2} (0, T; \cdot C^\ep (\pa \Om))} \big).
%\end{array}
%\right.
%\end{align*}
\end{theorem}
\begin{proof}
%Let $\tilde u_0$ be an extension of $u_0$ and  $W_1(x,t) =
%\int_{\R^n} \Ga(x-y, t) \tilde u_0(y) dy$.
From Proposition
\ref{prop1},  there is $(\Phi, \Psi) \in \calC^{\al, \frac{\al}2}(\pa
\Om) \times \calC^{\frac{\al}2}(0, T; \calC_{{\mathcal D}_{\eta}} (\pa
\Om))$ such that
\begin{align*}
\notag \Phi  + U_{tan} [\Phi]  + \na_S V [\Psi] = g_{tan},\\
\Psi + K^* [\Psi] + {\bf n}\cdot U
[ \Phi] = g \cdot {\bf n}.
\end{align*}
Let $ u(x,t) = {\mathcal U}[\Phi](x,t) + {\mathcal V}[\Psi](x,t)$.
Then, the estimate \eqref{CK-June21-1100} of $u$ is consequences of
Proposition \ref{CK-June21-1000}. %It remains to show that $\nabla u$
This completes the proof.
\end{proof}

As mentioned earlier, results of Theorem \ref{mainthm-Stokes} is
direct due to Theorem \ref{stokes-boundary}. Since its verification
is direct, we omit its details.

%%%%%%%%%%%%%%%%%%%%%%%%%%%%%%%%%%%%%%%%%%%%%%%%%%%%%%%%%%%%%%%%%%%%%%%%%%%%%%%
%%%%%%%%%%%%%%%%%%%%%%%%%%%%%%%%%%%%%%%%%%%%%%%%%%%%%%%%%%%%%%%%%%%%%%%%%%%%%%%
%%%%%%%%%%%%%%%%%%%%%%%%%%%%%%%%%%%%%%%%%%%%%%%%%%%%%%%%%%%%%%%%%%%%%%%%%%%%%%%

\section{Construction of an example in Theorem \ref{counter-exmaple}}\label{example}
\setcounter{equation}{0}

In this section, we construct an example, which shows that the condition of boundary data in  Theorem \ref{mainthm-Stokes} is crucial.
\\
\begin{pfthm2}

We consider the Stokes system \eqref{CK-Aug6-10} in two dimensions. Suppose that $\Om \subset {\mathbb R}^2_+$ and  a part of
boundary is flat and it contains, via translation, the open unit
interval, e.g. $\{x_1\in\R: \abs{x_1}<2\}$. We let $g=(g_1,
g_2):\R\times \R\rightarrow\R^2$ such that $g_1$ is identically
zero, that is $g_1 =0$, and $g_2$ is defined by
\begin{align*}
g_2(x_1, t) = (|x_1|^2 +  |t|)^{\frac12 \al}  \bke{\arctan
\frac{x_1^\al}{ t^{\frac12 \al}}}   \bke{\arctan
\frac{t^{\frac12 \al}}{x_1^\al}} \chi_{\{x_1> 0\}}(x_1) \chi_{\{t >
0\}}(t) \phi(x_1),
\end{align*}
where $\phi \in C^\infty_c (-1, 1)$ with $\phi \equiv 1$ in
$(-\frac12, \frac12)$. Clearly, $g$ is supported in $B_1\times \R_+$
and $g \in C^{\al, \frac12 \al}({\mathbb R} \times {\mathbb R} )$
(See Theorem 1.4 in \cite{FMM}). However, we can see that $g \notin
\dot {\mathcal C}_{D_\eta}( {\mathbb R} \times {\mathbb R})$.
Indeed, suppose that $ g \in
\dot {\mathcal C}_{D_{\eta_0}}( {\mathbb R} \times {\mathbb R}) $
for some Dini-continuous function $\eta_0$. Note that $\liminf_{r\ri 0} \eta_0(r) =0$.  Taking $x_1 = t^2$, we have
\begin{align*}
\| g\|_{\dot {\mathcal C}_{D_{\eta_0}}( {\mathbb R} \times {\mathbb R}) }
& \geq \frac{|g(0,0) - g(0,s) -g(x_1, 0) + g(x_1,t)|}{t^{\frac{\al}2} \eta_0(x_1) }\\
&= \frac{ g(t^2,t) }{t^{\frac12 \al} \eta_0(t^2) }=\frac{2^\frac12
t^{\frac12 \al} } {t^{\frac12 \al} \eta_0(t^2) }=\frac{2^\frac12  }
{  \eta_0(t^2) } \ri \infty \qquad \mbox{ as }\, t \ri 0.
\end{align*}
We  consider the Stokes system in a half-space with boundary
data $g$ and the solution $u = (u^1, u^2) $ is represented  by (see
\cite{K} and \cite{So})
\begin{align*}
u^i(x,t) &= \sum_{j=1}^{2}\int_0^t \int_{{\mathbb R} } K_{ij}(
x_1-y_1,x_2,t-s)g_j(y_1,s) dy_1ds,\qquad i=1,2,
\end{align*}
where
\begin{align*}
K_{ij}(x_1-y_1,x_2,t) &  =  -2 \delta_{ij} D_{x_n}\Ga(x_1-y_1, x_2,t)  -L_{ij} (x_1-y_1,x_2,t) \\
& \quad +  \de_{j2} \de(t)  D_{x_i} N(x_1-y_1,x_2),\qquad i,j=1,2
\end{align*}
with
\begin{align*}
L_{ij}(x, t)  =  D_{x_j}\int_0^{x_2} \int_{{\mathbb R}}   D_{z_2}
\Ga(z,t) D_{x_i}   N(x-z)  dz,\qquad i,j=1,2.
\end{align*}
From Remark \ref{rem-4}, we obtain that $u \in \calC^{\al, \frac12 \al}(\pa \Om \times [0, 1])$.
This completes the proof. 

Here the tangential component of $u$, i.e. $u^1$, is given by
\begin{align*}
&u^1(x,t)  = - \int_0^t \int_{{\mathbb R} }
L_{12}(x_1-y_1,x_2,t-s)g_2(y_1,s) dy_1 ds+ \int_{{\mathbb R}}
D_{y_1} N(x_1-y_1,x_2 )g_2(y_1, t) dy_1\\
&\quad = - \int_0^t \int_{{\mathbb R}} L_{21}(
x_1-y_1,x_2,t-s)g_2(y_1,s) dy_1ds -  \int_0^t
\int_{{\mathbb R}} D_{x_2}\Ga(x_1 -y_1, x_2, t-s) H g_2(y_1, s) dy_1 ds\\
&\qquad + \int_{{\mathbb R}} D_{y_1} N( x_1-y_1,x_2 )g_2 (y_1, t)
dy_1: = I_{1}(x,t) + I_2(x,t) + I_3(x,t),
\end{align*}
where $H$ is a Hilbert transform defined as
\begin{align*}
H g(y_1) = p.v.\frac1{\pi} \int_{{\mathbb R}} \frac1{y_1 -z_1} g(z_1 ) dz_1  =
\lim_{\ep \ri 0} \frac1{\pi}\int_{|y_1 -z_1| > \ep} \frac1{y_1 -z_1} g(z_1 ) dz_1.
\end{align*}
%{\color{red}{Note  that by the property of kernel, we have $I_1(0, 0) =0$. Since
%$u^1(x_1, 0, t) = g_2(x_1, t)$ and $g_2(x_2, 0) =0$, we have
%$u_1^2(0,0) =0$????.}}

It can be
checked that $I_1 \in C^{\al, \frac12 \al} ({\mathbb R} \times
{\mathbb R})$ (see e.g. \cite{CJ}) and so we obtain
%\begin{align*}
$|I_1(0,x_2,t)- I_1 (0,0,0)| \leq c (x_2^2 + t)^{\frac12 \al}$.
%\end{align*}
  Hence, we have
\begin{align*}
|u^1(0, x_2,t) - u^1(0,0)| =  |u^1(0,x_2,t)| \geq  | I_2(0, x_2,t)+
I_3(0,x_2,t)| - c (x_2^2 + t)^{\frac12 \al}.
\end{align*}

%Hence, for $0 < t \leq 1$, we get
%\begin{align}\label{u-1-estimate}
%|u_1(0, x_2,t) - u_1(0,0)|  \geq  | I_2(0, x_2,t)| - c  (x_2^2 + t)^{\frac12 \al}.
%\end{align}

Now, we estimate $| I_2(0, x_2,t)|$. Since $g_2$ is a function  in
Holder continuous and has compact support in $ x_1 \in (0,1) $, $Hg_2 $ is
bounded in ${\mathbb R}$ and $|Hg_2(y_1)| \leq c |y_1|^{-1}$ for $|y_1| \geq
1/2$. For $|y_1| \leq \frac12 $, we have
\begin{equation}\label{CK-Oct29-10}
Hg_2(y_1,s) =    \int_{{\mathbb R}}   \frac{1}{y_1 -z_1}  g_2(z_1,\tau) dz_1 =
\int_0^{|y_1|}  \cdots dz_1 + \int_{|y_1|}^{2|y_1|} \cdots dz_1 +
\int_{2|y_1|}^1 \cdots dz_1.
\end{equation}
%\begin{align*}
%Hg(y,s)
%& =    \int_{{\mathbb R}}   \frac{1}{y -z}  g(z,\tau) dz\\
%&  = \int_0^{|y|}  \frac{1}{y -z}g(z,\tau) dz + \int_{|y|}^{2|y|}
%\frac{1}{y -z} g(z,\tau) dz +   \int_{2|y|}^1 \frac{1}{y -z}
%g(z,\tau) dz.
%\end{align*}
Here, using change of variable, the first and second terms are
estimated as follows:
\[
\int_0^{|y_1|}  \cdots dz_1 +   \int_{|y_1|}^{2|y_1|} \cdots dz_1
=\int_0^{|y_1|} \frac{1}{z_1} (g_2(y_1-z_1,\tau)  - g_2(y_1
+z_1,\tau) ) dz_1
\]
\[
\leq c\int_0^{|y_1|} \frac{1}{z_1} z_1^\al dz_1 \thickapprox
|y_1|^\al.
\]
%\begin{align*}
%\int_0^{|y_1|}  \cdots dz_1 +   \int_{|y_1|}^{2|y_1|} \cdots dz_1
%=\int_0^{|y_1|} \frac{1}{z_1} (g_2(y_1-z_1,\tau)  - g_2(y_1 +z_1,\tau) ) dz_1  \leq
%\int_0^{|y_1|} \frac{1}{z_1} z_1^\al dz_1 \thickapprox |y_1|^\al.
%\end{align*}
It remains to estimate the third term in \eqref{CK-Oct29-10}.
Firstly, in case that $\tau  < (2|y_1|)^2 $, we obtain
\begin{align}\label{CK1107-5}
&\notag \int_{2|y_1|}^1 \frac{1}{y_1-z_1 } g_2(z_1,\tau)dz_1
\thickapprox \int_{2|y_1|}^1 \frac{1}{y_1-z_1 }(|z_1|^2 +
\tau)^{\frac12 \al}
 \bke{\arctan \frac{z_1^\al}{ \tau^{\frac12 \al}}}    \bke{\arctan  \frac{\tau^{\frac12 \al}}{z_1^\al}} dz_1\\
& \qquad\quad\thickapprox -\int_{2|y_1|}^1 \frac{1}{z_1 } (|z_1|^\al
+ \tau^{\frac12 \al} ) \frac{\tau^{\frac12 \al}{z_1^\al}} dz_1
\thickapprox  - \tau^{\frac12 \al} \big( -\ln \, |y_1| + \tau ^{
\al} |y_1|^{-\al} \big).
\end{align}
On the other hand, if $\tau  >(2 |y_1|)^2 $, then  we have
\begin{align}\label{CK1107-8}
\notag \int_{2|y_1|}^1 \frac{1}{y_1-z_1 } g_2(z_1,\tau)dz_1
&  =  \int_{2|y_1|}^{\tau^\frac12} \frac{1}{y_1-z_1 }(|z_1|^2 +  \tau)^{\frac12 \al}
 \bke{\arctan \frac{z_1^\al}{ \tau^{\frac12 \al}}}     \bke{\arctan \frac{\tau^{\frac12 \al}}{z_1^\al}} dz_1\\
\notag & \quad +  \int_{\tau^\frac12}^1 \frac{1}{y_1-z_1 }(|z_1|^2 +
\tau)^{\frac12 \al}  \bke{\arctan \frac{z_1^\al}{
\tau^{\frac{\al}2}}}
  \bke{\arctan (\frac{\tau^{\frac12 \al}}{z_1^\al}} dz_1  \\
\notag&  \thickapprox  -\int_{2|y_1|}^{\tau^\frac12} \frac{1}{z_1 }(|z_1|^\al +
\tau^{\frac12 \al} ) \frac{z_1^\al}{ \tau^{\frac{\al}2}} dz_1
               -  \int_{\tau^\frac12}^1 \frac{1}{z_1 }(|z_1|^\al +  \tau^{\frac12 \al} )  \frac{\tau^{\frac12 \al}}{z^\al_1}   dz_1\\
& = -\big(\frac1{2\al} (\tau^{\frac12 \al} - \tau^{-\frac12 \al} |y_1|^{2\al}) -\frac1\al |y_1|^\al- \tau^{\frac12 \al} \ln \,
\tau^\frac12 \big).
\end{align}
For $x_1 = 0 $ and  $x_2 > 0$, we get
\begin{align*}%\label{I's}
I_2(0,x_2, t) &= \int_0^t \int_{{\mathbb R}} D_2\Ga( -y_1, x_2, t-s)
H g_2(y_1, s) dy_1 ds= J_1+ J_2,
\end{align*}
where
\begin{align*}
J_1 : & = \int_0^t
\int_{ |y_1| \geq \frac12} D_2\Ga( -y_1, x_2, t-s) H g_2(y_1, s) dy_1 ds,\\
J_2 : & = \int_0^t \int_{|y_1| \leq \frac12 } D_2\Ga( -y_1, x_2, t-s) H
g_2(y_1, s) dy_1 ds.
%I_3 : & = \int_0^t
%\int_2^\infty D_2\Ga(-y_1, x_2, x^2_2-s) H g(y_1, s) dy_1 ds.
\end{align*}
Noting that $\int^{\infty}_{\frac{2}{\sqrt{s}}}
e^{-y_1^2} \frac{1}{y_1} dy_1 ds \leq c e^{-\frac4s}$ for $s \leq 1$
and $e^{-a} \leq c_k a^{-k}$ for $a, k > 0$, we have
\begin{align*}
J_1 & \leq  2 \int_0^t \int_{\frac12}^\infty D_2\Ga( y_1, x_2, t-s)
\frac{1}{y_1} dy_1 ds
%& \leq  \int_0^t \frac{x_2}{s^2} e^{-\frac{x_2^2}{s}} \int_{-\infty}^{-2} e^{-\frac{y_1^2}{s}} \frac{1}{y_1} dy_1 ds\\
\leq  c\int_0^t \frac{x_2}{s^2} e^{-\frac{x_2^2}{s}} \int^{\infty}_{\frac{2}{\sqrt{s}}} e^{-y_1^2} \frac{1}{y_1} dy_1 ds\\
& \leq c \int_0^t  \frac{x_2}{s^2} e^{-\frac{x_2^2}{s}} e^{-\frac4s}
ds \leq  c\int_0^t  \frac{x_2}{s^2} e^{-\frac{x_2^2}{s}} s^{k} ds.
\end{align*}
Using the change of variables ($\eta: =\frac{x_2^2}{s} $) and taking $k = \frac{\al +1}2 < 1$, we have
\begin{align*}
J_1 & \leq c \int_{\frac{x_2^2}{t}}^\infty x_2^{2k -1} \eta^{-k}
e^{-\eta} d\eta\leq c x_2^\al.
\end{align*}

Next, from \eqref{CK1107-5} and \eqref{CK1107-8}, we have
\begin{align*}
J_2 & =  \int_0^t
\int_{|y_1| \leq \frac12 } D_2\Ga( y_1, x_2, s) H g_2(y_1, t-s) dy_1 ds\\
& \thickapprox  -\int_0^t \frac{x_2}{s^2}  e^{-\frac{x_2^2}{s}} \big( \int_0^{s^\frac12} e^{-\frac{y_1^2}{s}} \big( s^{-\frac12 \al}(s^\al - y_1^{2\al}) + (s^{\frac12 \al} - y_1^\al) - s^{\frac12 \al} \ln \, s^\frac12   + s^{\frac12 \al} \big) dy_1 ds \\
%& =  -\int_0^t \frac{x_2}{s^2}  e^{-\frac{x_2^2}{s}} \big( \int_0^{s^\frac12} e^{-\frac{y_1^2}{s}} \big( s^{-\frac12 \al}(s^\al - y_1^{2\al}) + (s^{\frac12 \al} - y_1^\al) - s^{\frac12 \al} \ln \, s^\frac12   + s^{\frac12 \al} \big) dy_1 ds \\
&  \qquad - \int_{s^\frac12}^1 e^{-\frac{y_1^2}{s}} s^{\frac12 \al}
\big( \ln \, y + s^\al y_1^{-\al} \big) dy_1 \big) ds:= J_2^1 + J_2^2.
\end{align*}
$J_2^1$ is computed as follows:
\begin{align*}
 J_2^1 & \leq   c \int_0^t \frac{x_2}{s^2}  e^{-\frac{x_2^2}{s} }     s^{\frac12 \al} |\ln \, s |
                         \int_0^{s^\frac12} e^{-\frac{y_1^2}{s}}   dy_1 ds\\
 & = c  \int_0^t \frac{x_2}{s^2}  e^{-\frac{x_2^2}{s} }    s^{\frac12 \al} |\ln \, s |
                        s^\frac12 \int_0^1 e^{- y_1^2 }   dy_1 ds\\
 &  =   c \int_0^t  x_2  e^{-\frac{x_2^2}{s} }
                          s^{\frac12 \al -\frac32}| \ln \, s|     ds\\
& = c x_2^\al \int_{\frac{x_2^2}{t}}^\infty   s^{-\frac12 -\frac12 \al} |\ln \frac{x_2^2}{s} |  e^{-s} ds\\
& \leq   cx_2^\al   \big(   |\ln x_2|  \int_{\frac{x_2^2}{t}}^\infty
s^{-\frac12 -\frac12 \al}   e^{-s}
ds +    \int_{\frac{x_2^2}{t}}^\infty
s^{-\frac12 -\frac12 \al} |\ln s|  e^{-s}
ds  \big).
\end{align*}
Similarly, $J_2^2$ is estimated by
\begin{align*}
J_2^2
& \leq c \int_0^t \frac{x_2}{s^2}  e^{-\frac{x_2^2}{s} } s^{\frac12 \al} \int_{s^\frac12}^1 e^{-\frac{y_1^2}{s}}  | \ln \, y_1|  dy_1  ds\\
& =  c \int_0^t \frac{x_2}{s^2}  e^{-\frac{x_2^2}{s} } s^{\frac12 \al}  s^\frac12 \int^{s^{-\frac12}}_1 e^{-y_1^2} | \ln \,s^{\frac12} y_1 |  dy_1  ds\\
& \leq c \int_0^t  x_2   e^{-\frac{x_2^2}{s} } s^{\frac12 \al -\frac32}   |\ln \, s |  ds\\
& = c x_2^\al \int_{\frac{x_2^2}{t}}^\infty   s^{-\frac12 -\frac12 \al} |\ln \frac{x_2^2}{s} |   e^{-s} ds\\
& \leq c x_2^\al  \big(|\ln x_2|
\int_{\frac{x_2^2}{t}}^\infty
s^{-\frac12 -\frac12 \al}  e^{-s} ds  + \int_{\frac{x_2^2}{t}}^\infty   s^{-\frac12 -\frac12 \al} \ln s  e^{-s} ds  \big) .
\end{align*}

Hence, for $x^2 \leq t$, we have
\begin{align*}
|I_2| \leq c x_2^\al   |\ln x_2|.
\end{align*}

Now, we estimate $I_3$.

\begin{equation}\label{CK-Oct29-100}
I_3 (0,x_2,t) =    \int_{{\mathbb R}}   \frac{z_1}{z_1^2 + x_2^2} g_2(z_1,t)
dz_1 = \int_0^{x_2}   \cdots dz_1 + \int_{x_2}^{2x_2}   \cdots dz_1 +
\int_{2x_2}^1  \cdots dz_1.
\end{equation}
Again, due to the change of variable, the first and second terms are
estimated by
\begin{align*}
\int_0^{x_2}   \cdots  dz_1 +   \int_{x_2}^{2x_2}  \cdots  dz_1
& =\int_0^{x_2} \frac{z_1}{z_1^2 + x_2^2} (g(y_1-z_1,\tau)  - g(y_1 +z_1,\tau) )
dz_1\\
& \leq c \int_0^{x_2} \frac{1}{z_1} z_1^\al dz_1\thickapprox x_2^\al.
\end{align*}

In case that $t  >(2x_2)^2  $, the last term in \eqref{CK-Oct29-100}
is computed as follows:
\begin{align*}
\int_{2x_2}^1  \frac{z_1}{z_1^2 + x_2^2} g(z_1,t)dz_1
&  =  \int_{2x_2}^{t^\frac12}  \frac{z_1}{z_1^2 + x_2^2}(|z_1|^2 +  t)^{\frac12 \al}   \arctan (\frac{z_1}{ t^{\frac12}})^\al    \arctan (\frac{t^\frac12}{z_1})^\al dz_1\\
& \quad +  \int_{t^\frac12}^1  \frac{z_1}{z_1^2 + x_2^2}(|z_1|^2 +  t)^{\frac12 \al}   \arctan (\frac{z_1}{ t^{\frac12}})^\al    \arctan (\frac{t^\frac12}{z_1})^\al dz_1  \\
&  \thickapprox  -\int_{2x_2}^{t^\frac12}  \frac{1}{z_1} (|z_1|^\al +
t^{\frac12 \al} ) (\frac{z_1}{ t^{\frac12}})^\al dz_1
               -  \int_{t^\frac12}^1 \frac{1}{z_1 }(|z_1|^\al +  t^{\frac12 \al} )  (\frac{t^\frac12}{z_1})^\al   dz_1\\
& \geq c  \big(t^{\frac12 \al} |\ln \,t| - x_2^\al - t^{\frac12 \al}).
%& \thickapprox -\big(t^{-\frac12 \al}(t^\al - x_2^{2\al}) +
%(t^{\frac12 \al} - x_2^\al) - t^{\frac12 \al} \ln \, t^\frac12   +
%t^{\frac12 \al} \big).
\end{align*}

Hence, for $x_2^2 \leq t$, we have

\begin{align*}
|I_3(0, x_2, t) | \geq  c  \big(t^{\frac12 \al} |\ln \,t| - x_2^\al - t^{\frac12 \al}).
\end{align*}

Summing up above estimates, for $ x_2 \leq t^\frac12\ll 1$, we
obtain
\begin{align*}
|I_2(0, x_2,t)  + I_3(0,x_2,t)| \geq    c \big(t^{\frac12 \al}
|\ln \, t| - x_2^\al |\ln \, x_2|    - x_2^\al -t^{\frac12 \al}
\big).
\end{align*}
Therefore, we conclude that $u_1 \notin C^{\al, \frac12
\al}(Q^+_1)$, since
\begin{align*}
\| u_1\|_{\dot {\mathcal C}^{\al, \frac12 \al}(Q^+_1(0,0))}& =
\sup_{x,t} \frac{|u(x,t) - u(y,s)|}{(|x-y|^2 + |t-s|)^{\frac12 \al}}
\geq c\sup_{x_1=0, x_2 =t < \frac12 } \frac{|u(0,x_2,t) - u(0,0)|}{(x_2^2 + t)^{\frac12 \al}}\\
& \geq c\sup_{ t < \frac12  } \frac{t^{\frac12 \al} |\ln \, t| -
t^\al |\ln \, t| -t^{\frac{\al}2}  }{t^{\frac12 \al}}= \infty.
\end{align*}

\end{pfthm2}

%%%%%%%%%%%%%%%%%%%%%%%%%%%%%%%%%%%%%%%%%%%%%%%%%%%%%%%%%%%%%%%%%%%%%%%%%%%%%%%
%%%%%%%%%%%%%%%%%%%%%%%%%%%%%%%%%%%%%%%%%%%%%%%%%%%%%%%%%%%%%%%%%%%%%%%%%%%%%%%
%%%%%%%%%%%%%%%%%%%%%%%%%%%%%%%%%%%%%%%%%%%%%%%%%%%%%%%%%%%%%%%%%%%%%%%%%%%%%%%

\section{Proof of Theorem \ref{Theorem1}}\label{iteration}
\setcounter{equation}{0}

In this section, We present the proof of Theorem \ref{Theorem1}.\\
\\
\begin{pfthm1}

We introduce function classes $X(\Om)$ and $X(Q_T)$ defined as
follows:
\begin{align*}
X^\al (\Om) : = {\mathcal C}^{\al } (\overline{\Om})\times {\mathcal
C}^{\al +1}(\overline{\Om}) \times {\mathcal C}_{{\mathcal
D}_{\eta}}^\al (\overline{\Om}),\qquad X^\al (Q_T): = {\mathcal
C}^{\al,\frac12 \al} (\overline{Q_T})\times{\mathcal
C}^{\al+1,\frac12 \al +\frac12}(\overline{Q_T}) \times {\mathcal
C}^{\al,\frac12 \al} (\overline{Q_T})
\end{align*}
with norms 
\begin{align*}
\|(\rho, \theta, u)\|_{ X^\al(\Om)}: &= \| \rho\|_{{\mathcal C}^{\al } (\overline{\Om})} 
+ \| \te\|_{{\mathcal C}^{\al +1}(\overline{\Om})} + \| u\|_{{\mathcal C}_{{\mathcal
D}_{\eta}}^\al (\overline{\Om})}, \\
\|(\rho, \theta, u)\|_{ X^\al(Q_T)}: &= \| \rho\|_{{\mathcal C}^{\al,\frac{\al}2 } (\overline{Q_T})}
+ \| \te\|_{{\mathcal C}^{\al +1}(\overline{Q_T})} + \| u\|_{{\mathcal C}_{{\mathcal
D}_{\eta}}^{\al,\frac{\al}2} (\overline{Q_T})}.
\end{align*}
Let $(\rho_0, \theta_0, u_0)  \in X^\al(\Om)$. We consider
% and $(\theta^1, u^1) \in X(\Om \times (0, T))$ be a solution of
\[
\rho^1_t - \Delta \rho^1=0, \quad \theta^1_t -\De \theta^1  =
0,\quad u^1_t -\De u^1 +\na p^1 =0,\quad{\rm div}\,u^1=0\qquad\mbox{
in }\,\,Q_T
\]
with initial and boundary conditions
\begin{equation}\label{boudnary-aug22-10}
\rho^1(x,0)=\rho_0(x), \qquad\theta^1(x,0)=\theta_0(x), \qquad
u^1(x,0)=u_0(x),
\end{equation}
\begin{equation}\label{boudnary-aug22-20}
\frac{\pa \rho^1}{\pa {\bf n}} =0, \qquad \frac{\pa \theta^1}{\pa
{\bf n}} =0, \qquad u^1=0\qquad \mbox{on }\,\,\pa \Om\times (0, T),
\end{equation}
By Theorem \ref{mainthm-Stokes}, Theorem \ref{heat-domain-100} and Remark \ref{rem-3}, we get
\begin{align*}
\| (\rho^1, \theta^1, u^1)\|_{X^\al(Q_T)}   \leq c  \|( \rho_0,
\theta_0, u_0)\|_{X^\al (\Om)}.
\end{align*}
For $m=1,2,\cdots$ we define iteratively by $(\rho^{m+1},
\theta^{m+1}, u^{m+1})$ a solution of the following equations in
$Q_T$:
\[
\rho^{m+1}_t -\De \rho^{m+1}  =-{\rm div}(u^m\rho^m)+\nabla\cdot
{\mathcal F}^m,
\]
\[
\theta^{m+1}_t -\De \theta^{m+1}  = -u^m\cdot\nabla \theta^m+f^m,
\]
\[
u^{m+1}_t -\De u^{m+1} +\na p^{m+1} =-\nabla\cdot (u^m\otimes u^m)+
G^m,\qquad{\rm div}\, u^{m+1}=0
\]
with the same boundary and initial conditions as
\eqref{boudnary-aug22-10} and \eqref{boudnary-aug22-20}. Here ${\mathcal F}^m,
f^m$ and $G^m$ denote
\[
{\mathcal F}^m:=F(\rho^m, \theta^m, \nabla \theta^m, u^m), \quad
f^m:=f(\rho^m, \theta^m, \nabla \theta^m, u^m), \quad G^m:=G(\rho^m,
\theta^m, \nabla \theta^m, u^m).
\]

We fix $T>0$, which will be specified later. Then, we have
\begin{align}\label{1209number}
\notag &\|(\rho^{m+1}, \te^{m+1}, u^{m+1})\|_{X^\al(Q_T)}\\
\notag & \leq c \Big(\|(\rho_0, \te_0, u_0)\|_{X^\al(\Om)}  + T^\frac12 \|
u^{m} \rho^{m} \|_{ {\mathcal C}^{\al, \frac12 \al}( \overline Q_T) }+
T^\frac12 \| {\mathcal F}^m\|_{ {\mathcal C}^{\al, \frac12 \al}(\overline Q_T) }\\
\notag &  \qquad + \max(T^\frac12, T^{\frac12-\frac{\alpha}{2}}) \| u^{m} \cdot
\na \theta^{m} \|_{ {\mathcal C}^{\al, \frac12 \al}( \overline Q_T) } + \max(T^\frac12, T^{\frac12-\frac{\alpha}{2}})  \| f^m\|_{
{\mathcal C}^{\al, \frac12 \al}(\overline Q_T) }\\
 &\qquad  +
\max(T^\frac12, T^{\frac12-\frac12 \al} ) \|  u^{m} \otimes u^{m} \|_{{\mathcal C}^{\al, \frac12
\al}(\overline Q_T) }
  + \max(T,T^{\frac12-\frac{1}{2} \al}) \| G^m\|_{ {\mathcal C}^{\al,
\frac12 \al}(\overline Q_T) }\Big).
\end{align}
%
%\begin{equation}\label{theta-aug22-50}
%\|\rho^{m+1} \|_{{\mathcal C}^{\al, \frac12 \al  } ( \overline Q_T)} \leq c \big(\|
%\rho_0\|_{ {\mathcal C}^{\al  }(\overline{\Om})}+ T^\frac12 \|
%u^{m} \rho^{m} \|_{ {\mathcal C}^{\al, \frac12 \al}( \overline Q_T) }+
%T^\frac12 \| {\mathcal F}^m\|_{ {\mathcal C}^{\al, \frac12 \al}(\overline Q_T) } \big),
%\end{equation}
%\[
%\|\theta^{m+1} \|_{{\mathcal C}^{\al +1, \frac12 \al  +\frac12}
%( \overline Q_T)} \leq c \big(\| \theta_0\|_{  {\mathcal C}^{\al  +1 }(\overline{\Om}
%)}+ \max(T^\frac12, T^{\frac12-\frac{\alpha}{2}}) \| u^{m} \cdot
%\na \theta^{m} \|_{ {\mathcal C}^{\al, \frac12 \al}( \overline Q_T) }
%\]
%\begin{equation}\label{theta-aug22-100}
%+ \max(T^\frac12, T^{\frac12-\frac{\alpha}{2}})  \| f^m\|_{
%{\mathcal C}^{\al, \frac12 \al}(\overline Q_T) } \big),
%\end{equation}
%\begin{align}\label{u-aug22-110}
%\notag \| u^{m+1} \|_{{\mathcal C}^{\al, \frac12 \al} (\overline Q_T)} & \leq c \big(  \|
%u_0\|_{ {\mathcal C}^{\al }_{{\mathcal D}_{\eta}}(\overline{\Om} )}+
%\max(T^\frac12, T^{\frac12-\frac12 \al} ) \|  u^{m} \otimes u^{m} \|_{{\mathcal C}^{\al, \frac12
%\al}(\overline Q_T) }\\
%&\qquad + \max(T,T^{\frac12-\frac{1}{2} \al}) \| G^m\|_{ {\mathcal C}^{\al,
%\frac12 \al}(\overline Q_T) } \big).
%\end{align}
Let $M_0 = \frac{c}2 \|( \rho_0, \theta_0,u_0)\|_{ X^\al(\Om)}$, where $c$ is
the constant in \eqref{1209number}. Suppose
that $M$ is a number with $M>M_0$ such that
\[
\| (\rho^m, \theta^m, u^m)\|_{ X^\al(Q_T)} < M.
\]
We then see that
\[
\|  u^m   \rho^m   \|_{  C^{\al, \frac12 \al}(\overline Q_T) } +\| u^{m} \cdot
\na \theta^{m} \|_{ {\mathcal C}^{\al, \frac12 \al}(\overline Q_T) }+\|  u^{m}
\otimes u^{m} \|_{{\mathcal C}^{\al, \frac12 \al}(\overline Q_T) }
\]
\begin{equation}\label{CK-aug-100}
\leq \| u^m \|_{ C^{\al, \frac12 \al}(\overline Q_T) } \| (\rho^m, \theta^m, u^m)\|_{X^\al(Q_T)}\leq cM^2,
\end{equation}
\begin{align}\label{CK-aug-110}
\notag \|F^m\|_{ {\mathcal C}^{\al, \frac12 \al}(Q_T) }+\|f^m\|_{ {\mathcal
C}^{\al, \frac12 \al}(Q_T) }+\|G^m\|_{ {\mathcal C}^{\al, \frac12
\al}(Q_T) }\\
\notag \leq c\bke{\|\nabla {\mathcal F}^m\|_{L^{\infty}(Q_T)}+\|\nabla
f^m\|_{L^{\infty}(Q_T)}+\|\nabla G^m\|_{L^{\infty}(Q_T)}}\|(\rho^m,
\theta^m, u^m)\|_{X(Q_T)}\\
\leq  c\|(\rho^m, \theta^m, u^m)\|^l_{X^\al(Q_T)}\leq  cM^l,
\end{align}

Taking $T $ sufficiently small, due to
\eqref{CK-aug-100}-\eqref{CK-aug-110}, we obtain
\begin{equation}\label{CK-aug23-500}
  \|(\rho^{m+1}, \te^{m+1}, u^{m+1} )\|_{X^\al (Q_T)} < M
\end{equation}
Iteratively, we conclude that \eqref{CK-aug23-500} holds for all
$m$.

Next we will show that the sequence $(\rho^m, \theta^m, u^m)$ are
Cauchy in $X^\al(\overline Q_T)$ for a $T>0$. For convenience, we set
\[
\varrho^m = \rho^m - \rho^{m-1},\quad \Theta^m = \theta^m -
\theta^{m-1},\quad U^m=u^m-u^{m-1},\quad P^{m-1}=p^{m}-p^{m-1}.
\]
We then see that $(\varrho^m, \Theta^m, U^m)$ solve the following
system:
\begin{align*}%\label{CK-aug23-600}
\varrho^{m+1}_t -\De \varrho^{m+1} &= {\rm div} \, \big( \varrho^m
u^{m-1} + \rho^m U^{m}\big)+ \nabla\cdot (F^m-F^{m-1}),\\
%\end{equation}
%\begin{equation}\label{CK-aug23-610}
\Theta^{m+1}_t -\De \Theta^{m+1} & =   U^m \na \theta^m  + u^{m-1}
\na \Theta^{m}+f^m -f^{m-1},\\
%\end{equation}
%\begin{equation}\label{CK-aug23-620}
U^{m+1}_t -\De U^{m+1}_t - \De U^m + \na P^m &  =-{\rm
div}\,\big(u^m\otimes U^{m}+U^{m}\otimes u^{m-1} \big) +
G^m-G^{m-1},
\end{align*}
with following initial and boundary conditions
\begin{equation*}%\label{CK-aug23-630}
\frac{\pa\varrho^{m+1}}{\pa {\bf n}} =0, \quad \frac{\pa
\Theta^{m+1}}{\pa {\bf n}} =0, \quad U^{m+1}=0\qquad \mbox{on
}\,\,\pa \Om\times (0, T),
\end{equation*}
and initial zero conditions, namely
$\varrho^{m+1}(x,0)=\Theta^{m+1}(x,0)=U^{m+1}(x,0)=0$. Then, we have
\begin{align*}
&\| (\varrho^{m+1}, \Theta^{m+1},  U^{m+1} ) \|_{X^\al ( Q_T)}
\le c   \max(T,  T^{\frac12-\frac{\alpha}{2}} ) 
\Big(\| (  \varrho^{m}, \Theta^{m}, U^{m} ) \|_{ X^\al ( Q_T)}\Big)\times\\
& \qquad \Big(\|(\rho^m, \theta^m, u^{m-1}) \|_{X^\al (Q_T)}+\norm{(\rho^m,
\theta^m, u^m)}^{l-1}_{X^\al(Q_T)}+\norm{(\rho^{m-1}, \theta^{m-1},
u^{m-1})}^{l-1}_{X^\al(Q_T)}\Big).
\end{align*}
Choosing  sufficiently small $T$, we obtain
\[
\| ( \varrho^{m+1}, \Theta^{m+1}, U^{m+1} ) \|_{ X^\al ( Q_T)}
\leq \frac{1}{2}  \| (\varrho^{m}, \Theta^{m}, U^{m}) \|_{X^\al ( Q_T)}.
\]
Via the argument of contraction mapping, the constructed sequence is
indeed convergent in $X^\al(Q_T)$, provided that $T$ is sufficiently
small. Let $(\rho, \theta, u)$ be the limit of $(\rho^m, \theta^m,
u^m)$. Then, it is direct that $(\rho, \theta, u)$ is the unique
solution of \eqref{CK-Aug18-9}-\eqref{CK-Aug17-11}. Since its
verification is rather standard, we skip its details.
\end{pfthm1}

As an application of Theorem \ref{Theorem1}, we establish local
well-posedness in H\"older spaces for a mathematical model
describing the dynamics of oxygen, swimming bacteria, and viscous
incompressible fluids in $\bbr^2$. Such a model was proposed by
Tuval et al.\cite{TCDWKG}, formulating the dynamics of swimming
bacteria, {\it Bacillus subtilis}, which is given as
\begin{equation}\label{CK-Aug3-30} \left\{
\begin{array}{ll}
\partial_t n + u \cdot \nabla  n - \Delta n= -\nabla\cdot (\chi (c) n \nabla c),\\
\vspace{-3mm}\\
\partial_t c + u \cdot \nabla c-\Delta c =-k(c) n,\\
\vspace{-3mm}\\
\partial_t u + u\cdot \nabla u -\Delta u +\nabla p=-n \nabla
\phi,\quad \nabla \cdot u=0
\end{array}
\right. \quad\mbox{ in }\,\, Q_{T},
\end{equation}
where $c(t,\,x) : Q_{T} \rightarrow \R^{+}$, $n(t,\,x) : Q_{T}
\rightarrow \R^{+}$, $u(t,\, x) : Q_{T} \rightarrow \R^{d}$ and
$p(t,x) :  Q_{T} \rightarrow \R$ denote the oxygen concentration,
cell concentration, fluid velocity, and scalar pressure,
respectively. Here $\R^+$ indicates the set of non-negative real
numbers.

The nonnegative functions $k(c)$ and $\chi (c)$ denote the oxygen
consumption rate and the aerobic sensitivity, respectively, i.e. $k,
\chi:\R^+\rightarrow\R^+$ such that $k(c)=k(c(x,t))$ and
$\chi(c)=\chi(c(x,t))$. Initial data are given by $(n_0(x), c_0(x),
u_0(x))$ with $n_0(x),\, c_0(x) \geq 0$ and $\nabla \cdot u_0=0$.

There are many known results for the system \eqref{CK-Aug3-30}
regarding existence, regularity and asymptotics. We do not recall
previous results but give some list of reference (see e.g.
\cite{ckl}, \cite{ckl-cpde}, \cite{CFKLM}, \cite{DLM}, \cite{FLM},
\cite{Lorz}, \cite{TCDWKG}, \cite{Win2}, \cite{Win3}). As far as the
authors' concerned, local-wellposedness in H\"older spaces is not
known for the system \eqref{CK-Aug3-30} and it is, however, a direct
consequence of Theorem \ref{Theorem1}. Since $c$ is uniformly
bounded by maximum principle, non-linear terms satisfy
\eqref{CK-Aug17-10} and \eqref{CK-Aug17-10-5} with $l=2$. Therefore,
as an easy consequence of Theorem \ref{Theorem1}, for the system
\eqref{CK-Aug3-30} we have the following:
\begin{theorem}\label{corollary1}
Let the initial data $(n_0, c_0, u_0)$ be given in
$\calC^{\alpha}((\overline{\Omega}))\times
\calC^{\alpha+1}((\overline{\Omega}))\times \calC^{\al}_{{{\mathcal
D}_{\eta}}}(\overline{\Omega})$ for $\alpha\in(0,1)$ with $n_0\geq
0$ and $c_0\geq 0$. Assume that $\chi, k, \chi', k'$ are all
non-negative and
 $\chi$, $k\in C^m(\R^+)$ and $k(0)=0$, $\|
\nabla^l \phi \|_{L^1\cap L^{\infty}}<\infty$ for $1\le |l|\le m$.
There exists $T>0 $ such that unique solutions $(n, c, u)$ of
\eqref{CK-Aug3-30} exist in the class $\calC^{\al,\frac12 \al }
(\overline Q_T)\times \calC^{\al +1,\frac12 \al +\frac12}(\overline Q_T) \times
\calC^{\al,\frac12 \al} (\overline Q_T)$ for any $t<T$.
%Smoothness can be obtained as long as initial data are smooth.
\end{theorem}
%Since its verification of Theorem \ref{corollary1} is
%straightforward, we skip its details.

The result of Theorem \ref{corollary1} is a direct consequence of
Theorem \ref{Theorem1}, and we skip its details.

\section*{Appendix}

In this Appendix, we present the proof of Lemma \ref{lemm4}.\\
\\
{\bf Proof of Lemma \ref{lemm4}}

First,  we prove the estimate \eqref{al+0}. Direct computations show
that
\[
\abs{D_x\Lambda_0(f)(x,t) - D_y \Lambda_0(f)(y,t)} =\abs{ \int_0^t
\int_{\R^n} D_z\Ga(z, t-\tau) \big(f(x-z, \tau) - f(y-z, t-\tau)
\big) dzd\tau}
\]
\[
\leq c\| f\|_{L^\infty(0, T; \dot C^\al(\R^n))}|x-y|^\al   \int_0^t
( t-\tau)^{-\frac12 } d\tau
\]
\begin{align}\label{CK-Aug20-100}
\leq c T^\frac12   |x-y|^\al  \| f\|_{L^\infty(0, T; \dot
\calC^\al(\R^n))}.
\end{align}
Hence, we have
\begin{equation}\label{CK-Aug19-30}
\|\Lambda_0(f)\|_{L^\infty(0, T; \dot C^{\al +1} (\R^n))} \leq
cT^{\frac12  } \| f\|_{L^\infty(0, T; \dot \calC^\al(\R^n))}.
\end{equation}
On the other hand, for $s < t$, we obtain
\[
\abs{\Lambda_0(f)(x,t) - \Lambda_0(f)(x,s)}\le\abs{\int_s^t
\int_{\R^n} \Ga(x-z, t-\tau) f(z,\tau)  dzd\tau}
\]
\[
+\abs{\int_0^s \int_{\R^n}    \big( \Ga(x-z, t-\tau) - \Ga(x-z,
s-\tau) \big)      f(z,  \tau)    dzd\tau}= I_1 + I_2.
\]
%\begin{align*}
%V(x,t) - V(x,s) & =  \int_s^t \int_{\R^n}   \Ga(x-z, t-\tau)    f(z,\tau)  dzd\tau\\
%  & \quad +\int_0^s \int_{\R^n}    \big( \Ga(x-z, t-\tau) - \Ga(x-z, s-\tau) \big)      f(z,  \tau)    dzd\tau\\
%  & = I_1 + I_2.
%\end{align*}
We first estimate $I_1$.
\[
I_1 \leq \|   f\|_{L^\infty(\R^n \times (0, T))}  \int_s^t
\int_{\R^n} \Ga(z, t-\tau)    dzd\tau
\]
\begin{align*}%\label{CK-june17-100}
=\|   f\|_{L^\infty(\R^n \times (0, T))}  (t-s)\leq
T^{\frac12 -\frac12 \al}\|  f\|_{L^\infty(\R^n \times (0, T))}
(t-s)^{\frac{\al}2 + \frac12}.
\end{align*}
%
%\begin{align*}
%I_1  & \leq \|   f\|_{L^\infty(\R^n \times (0, T))}  \int_s^t \int_{\R^n} \Ga(z, t-\tau)    dzd\tau \\
%  & = \| \tilde f\|_{L^\infty(\R^n \times (0, T))}  (t-s) \\
%  & \leq  T^{\frac12 -\frac12 \al}\| \tilde f\|_{L^\infty(\R^n \times (0, T))}  (t-s)^{\frac{\al}2 + \frac12}.
%\end{align*}
For $I_2$, since $ \int_{\R^n}    \big( \Ga(x-z, t-\tau) - \Ga(x-z,
s-\tau) \big) dz =0$,
we have
\[
I_2= \abs{\int_0^s \int_{\R^n}   \big( \Ga(x-z, t-\tau) - \Ga(x-z,
s-\tau) \big) \big(   f(z,  \tau) - f(x, \tau) \big)    dzd\tau}
\]
\[ \leq  c\|   f\|_{L^\infty(0, T; \dot \calC^\al (\R^n))}  \int_0^s
\int_{\R^n} |x -z|^\al   | \Ga(x-z, t-\tau) - \Ga(x-z, s-\tau)|
dzd\tau
\]
\[
\leq c \|   f\|_{L^\infty(0, T; \dot \calC^\al (\R^n))}  \int_0^s
\int_{\R^n} |x-z|^\al   \int_s^t | D_\te \Ga(x-z, \te-\tau)| d\te
dzd\tau
\]
\[
 \leq c\|  f\|_{L^\infty(0, T; \dot \calC^\al (\R^n))}  \int_0^s \int_s^t (\te -\tau)^{\frac{\al}2 -1} d\te   d\tau
 \leq c \|   f\|_{L^\infty(0, T;\dot \calC^\al (\R^n))} \big(
t^{\frac{\al}2 + 1}-     s^{\frac{\al}2 +1} \big).
\]
%\begin{align*}
%I_2 & = \int_0^s \int_{\R^n}    \big( \Ga(x-z, t-\tau) - \Ga(x-z, s-\tau) \big)   \big(   \tilde f(z,  \tau) -\tilde f(x,\tau) \big)    dzd\tau\\
%& \leq  \| \tilde f\|_{L^\infty(0, T; C^\al (\R^n))}  \int_0^s \int_{\R^n} |x -z|^\al   | \Ga(x-z, t-\tau) - \Ga(x-z, s-\tau)| dzd\tau\\
%& \leq  \| \tilde f\|_{L^\infty(0, T; C^\al (\R^n))}  \int_0^s \int_{\R^n} |x-z|^\al   \int_s^t | D_\te \Ga(x-z, \te-\tau)| d\te   dzd\tau\\
%& = \| \tilde f\|_{L^\infty(0, T; C^\al (\R^n))}  \int_0^s \int_s^t \int_{\R^n}  |x-z|^\al   | D_\te \Ga(x-z, \te-\tau)| dz d\te   d\tau \\
%&  \leq \| \tilde f\|_{L^\infty(0, T; C^\al (\R^n))}  \int_0^s \int_s^t (\te -\tau)^{\frac{\al}2 -1} d\te   d\tau\\
%&  = \| \tilde f\|_{L^\infty(0, T; C^\al (\R^n))}  \int_0^s  (t -\tau)^{\frac{\al}2} - (s -\tau)^{\frac{\al}2}   d\tau\\
%%&  = \| \tilde f\|_{L^\infty(0, T; C^\al (\R^n))}  \int_0^s  (t -\tau)^{\frac{\al}2} - (s -\tau)^{\frac{\al}2}   d\tau\\
%&  = \| \tilde f\|_{L^\infty(0, T; C^\al (\R^n))} \big( t^{\frac{\al}2 + 1}-    (t -s)^{\frac{\al}2 + 1} - s^{\frac{\al}2 +1} \big)\\
%&  \leq  \| \tilde f\|_{L^\infty(0, T; C^\al (\R^n))} \big(
%t^{\frac{\al}2 + 1}-     s^{\frac{\al}2 +1} \big).
%\end{align*}
By mean-value theorem, there is $\xi \in (s,t)$ such that
\begin{align*}
t^{\frac{\al}2 + 1}-     s^{\frac{\al}2 +1} = (\frac{\al}2 +1)
\xi^{\frac{\al}2}(t-s) \leq c T^\frac12 (t-s)^{\frac{\al}2 +
\frac12}.
\end{align*}
Hence, we have
\begin{align*}
I_2 \leq c  T^\frac12  \|
f\|_{L^\infty(0, T;\dot \calC^\al (\R^n))}.
\end{align*}

Combining above estimates, we obtain
\begin{equation}\label{CK-june17-110}
\|\Lambda_0(f)\|_{L^\infty(\R^n; \dot \calC^{\frac12 +\frac{\al}2} [0,
T])} \leq \max (T^\frac12, T^{\frac12 -\frac{\al}2}) \|
f\|_{L^\infty(\R^n ;  \calC^{\frac{\al}2}[0, T])}.
\end{equation}
Via \eqref{CK-Aug19-30} and \eqref{CK-june17-110}, we obtain
\eqref{al+0}.

We prove the estimate \eqref{al+0-2}. Due to
\eqref{CK-Aug20-100}, it is direct that
\begin{equation}\label{CK-Aug20-110}
\|\nabla\Lambda_0(f)\|_{L^\infty(0, T; \dot \calC^{\al} (\R^n))} \leq
T^\frac12 \| f\|_{L^\infty(0, T; \dot \calC^\al(\R^n))}.
\end{equation}

Therefore, it suffices to show
\begin{equation}\label{CK-Aug20-120}
\|\nabla\Lambda_0(f)\|_{L^\infty(\R^n; \dot \calC^{ \frac{\al}2} [0,
T])} \leq
 T^\frac12  \| f\|_{L^\infty(\R^n
\times \dot {\mathcal C}^{\frac{\al}2}[0, T])}.
\end{equation}
Indeed,  we note that for $s < t$,
\[
\abs{\nabla\Lambda_0(f)(x,t) -
\nabla\Lambda_0(f)(x,s)}\le\abs{\int_s^t \int_{\R^n} \nabla_x
\Ga(x-z, t-\tau) f(z,\tau)  dzd\tau}
\]
\[
+\abs{\int_0^s \int_{\R^n}    \big( \nabla_x \Ga(x-z, t-\tau) - D_x
\Ga(x-z, s-\tau) \big)      f(z,  \tau)    dzd\tau}= J_1 + J_2.
\]
We first estimate $J_1$. Sine $\int_{{\mathbb R}^n} D_z \Ga(z, \tau)
dz =0$ for $\tau > 0$, using change of variables, we have
\begin{align*} %\label{CK-june17-120}
  J_1 &\leq | \int_s^t
\int_{\R^n} D_z \Ga(z, t-\tau) \big( f(x-z,\tau) - f(x,\tau)   dzd\tau|\\
  & \leq \|  f\|_{L^\infty(0, T; \dot \calC^\al (\R^n) )}  |
\int_s^t
\int_{\R^n} |D_z \Ga(z, t-\tau)| |z|^\al   dzd\tau\\
  &  \leq  \|  f\|_{L^\infty(0, T; \dot \calC^\al (\R^n) )}  |
\int_s^t
  (t -\tau)^{-\frac12 +\frac{\al}2}  d\tau\\
  &  \leq  \|  f\|_{L^\infty(0, T; \dot \calC^\al (\R^n) )} (t-s)^{\frac12 +\frac{\al}2}\\
&  \leq  T^\frac12  \|  f\|_{L^\infty(0, T; \dot \calC^\al (\R^n) )}
(t-s)^{ \frac{\al}2}.
\end{align*}
Using $ \int_{\R^n}    \big( D_x \Ga(x-z, t-\tau) - D_x \Ga(x-z,
s-\tau) \big) dz =0$, we estimate $J_2$ as follows:
\[
J_2\le \|  f\|_{L^\infty(0, T; \dot \calC^\al (\R^n))}  \int_0^s
\int_{\R^n} |x -z|^\al   | D_x \Ga(x-z, t-\tau) - D_x \Ga(x-z,
s-\tau)| dzd\tau
\]
\[
\leq  \|  f\|_{L^\infty(0, T; \dot \calC^\al (\R^n))}  \int_0^s
\int_{\R^n} |x-z|^\al   \int_s^t | D_\te  D_x \Ga(x-z, \te-\tau)|
d\te dzd\tau
\]
\[
 \leq \|  f\|_{L^\infty(0, T; \dot \calC^\al (\R^n))}  \int_0^s \int_s^t (\te -\tau)^{\frac{\al}2 -\frac32} d\te   d\tau
\]
\[
=  \|  f\|_{L^\infty(0, T; \dot \calC^\al (\R^n))}\int_0^s \big(
(s-\tau)^{\frac{\al}2 -\frac12}-     (t-\tau)^{\frac{\al}2 -\frac12}
\big) d\tau
\]
\begin{align*}
\leq  c T^\frac12   \| \tilde f\|_{L^\infty(0, T; \dot C^\al (\R^n))}
(t-s)^{\frac{\al}2}.
\end{align*}
This implies the estimate \eqref{CK-Aug20-120}, and therefore,
together with \eqref{CK-Aug20-110}, we obtain \eqref{al+0-2}.

Next, we prove the estimate %\eqref{al+2} and
\eqref{al+2-2}.
Let $\hat f(\xi),\xi\in\mathbb{R}^{n}$ denote the Fourier
transform of $f(x),x\in\mathbb{R}^n$. Let
${\mathcal S} ({\mathbb R}^{n})$ denote the \emph{Schwartz space}
on $\mathbb R^{n}$ and let ${\mathcal S}'({\mathbb R}^{n})$ be
the dual space of the Schwartz space ${\mathcal S}({\mathbb
R}^{n})$. Fix a Schwartz function $\psi\in {\mathcal S} ({\mathbb
R}^{n})$ satisfying $\hat{\psi}(\xi) > 0$ on $\frac12 < |\xi|
  < 2$, $\hat{\psi}(\xi)=0$ elsewhere, and
$\sum_{j=-\infty}^{\infty} \hat{\psi}(2^{-j}\xi) =1$
for $\xi \neq 0$. Let
\begin{align*}
\widehat{\psi_j}(\xi) &:= \widehat{\psi}(2^{-j} \xi ), \qquad (j =
0, \pm 1, \pm 2 , \cdots).
%\widehat{\psi}(\xi,\tau) &:= 1- \sum_{j=1}^\infty \widehat{\psi} (2^{-j} \xi, 2^{-2j}\tau).
\end{align*}
%Let $\psi_j(x,t) = 2^{(n+2)j} \psi(2^j x, 2^{2j} t)$ such that $\hat \psi_j(\xi,\tau) = \hat \psi(2^{-j}\xi, 2^{-2j}\tau)$.
Note that
\begin{align*}
\dot {\mathcal C}^{\al } (\R^{n}) &= \{ f \in  {\mathcal
S}'({\mathbb R}^{n}) \, | \, \sup_{-\infty < j < \infty} 2^{\al j}
\| f * \psi_j\|_{L^\infty (\R^{n})}<\infty \},
\end{align*}
where $*$ is convolution in $\R^{n}$.  Let $\Psi = \psi_{-1}  +
\psi_0 +\psi_1$  and $\Psi_j(\xi) = \Psi(2^{-j} \xi)$ such that
$supp \, \Psi_j \subset \{ 2^{-j -2} < |\xi| < 2^{-j +2} \}$ and
$\Psi \equiv 1$ in  $2^{j -1} < |\xi|   < 2^{j +1}$. We observe that
\begin{align*}
\widehat{  \na \La_0 f * \psi_j}(\xi, t)
& = \int_0^t\xi e^{-(t -\tau) |\xi|^2}  \hat f(\xi,\tau) \hat \psi_j (\xi) d\tau\\
& =\int_0^t\xi \Phi_j(\xi) e^{-(t -\tau) |\xi|^2}  \hat f(\xi,\tau)
\hat \psi_j (\xi) d\tau.
\end{align*}
Note that the $L^\infty$-multiplier norms of  $\xi \Phi_j(\xi) e^{-(t -\tau) |\xi|^2}$ and $2^j\xi \Phi(\xi) e^{-(t -\tau)  2^{2j}|\xi|^2}$ are same (see Theorem 6.1.3 in \cite{BL}). By Lemma \ref{multiplier2},
 the $L^\infty$-multiplier norm of   $\xi \Phi(\xi) e^{-(t -\tau)  2^{2j}|\xi|^2}$ is dominated  by $e^{-\frac18(t -\tau) 2^{2j}}$.
Hence, we have
\begin{align}\label{CK-Sep11-100}
\notag  \|    \na \La_0 f * \psi_j(t)\|_{L^\infty(\R^{n} )}
&\leq  \int_0^t  \| {\mathcal F}^{-1} ( \widehat{ \xi \Phi_j(\xi) e^{-(t -\tau) |\xi|^2}   \hat f  \hat \psi_j})(\tau)\|_{L^\infty(\R^{n})}d\tau\\
\notag &\leq  \int_0^t 2^j  e^{-\frac18(t -\tau) 2^{2j}} \|   (f *   \psi_j)(\tau)\|_{L^\infty(\R^{n})} d\tau\\
\notag &\leq  2^{( 1 -\al )j}\int_0^t   e^{-\frac18(t -\tau) 2^{2j}}  \|    f(\tau)\|_{\dot{\mathcal C}^\al(\R^n)} d\tau \\
&\leq  2^{-( 1 +\al )j}  \|    f\|_{L^\infty (0, T; \dot{\mathcal C}^\al(\R^n) )} \int_0^{2^{2j}t}   e^{-\frac18\tau }   d\tau.
\end{align}
Hence, we have
\begin{align}\label{CK-Aug28-100}
 \|   \na \La_0 f\|_{L^\infty( 0, T; \dot {\mathcal C}^{\al +1} (\R^{n} ))}
&\leq  c   \|    f\|_{L^\infty (0, T; \dot{\mathcal C}^\al(\R^n) )}.
\end{align}

By the same argument of \eqref{CK-Aug20-120}, we have
\begin{equation}\label{CK-Aug20-120-3}
\|\nabla\Lambda_0(f)\|_{L^\infty(\R^n; \dot \calC^{ \frac{\al}2 +\frac12} [0,
T])} \leq
 c \| f\|_{L^\infty(\R^n
\times \dot \calC^{\frac{\al}2}[0, ]))}.
\end{equation}
By \eqref{CK-Aug28-100} and \eqref{CK-Aug20-120-3}, we obtain  \eqref{al+2-2}.

For \eqref{al+4}, note that  for $\ep < 1$, we have
\begin{align*}
\int_0^{2^{2j}t}   e^{-\frac18\tau }   d\tau \leq c \int_0^{2^{2j}t} \tau^{-\frac12 -\frac12 \ep}   d\tau  = c (2^{2j}t)^{\frac12-\frac12\ep}.
\end{align*}
Hence, from \eqref{CK-Sep11-100}, we have
\begin{align}\label{CK-Aug28-100-1}
 \|   \na \La_0 f\|_{L^\infty( 0, T; \dot {\mathcal C}^{\al +\ep} (\R^{n} ))}
&\leq  c  T^{\frac12 -\frac12\ep} \|    f\|_{L^\infty (0, T; \dot{\mathcal C}^\al(\R^n) )}.
\end{align}
By the same argument of \eqref{CK-Aug20-120}, we have
\begin{equation}\label{CK-Aug20-120-3-1}
\|\nabla\Lambda_0(f)\|_{L^\infty(\R^n; \dot C^{ \frac{\al}2 +\frac12 \ep} [0,
T])} \leq
 c  T^{\frac12  -\frac12 \ep} \|    f\|_{L^\infty (0, T; \dot{\mathcal C}^\al(\R^n) )}.
\end{equation}
By \eqref{CK-Aug28-100-1} and \eqref{CK-Aug20-120-3-1}, we obtain
\eqref{al+4}. This completes the proof.
\qed

\begin{lemma}\label{multiplier2}
Let $\rho_{tj}(\xi) = \xi \Psi_j ( \xi) e^{-t|\xi|^2}$ for each integer
$j$. Then $\rho_{tj}( \xi)$ is a $L^\infty(\R)$-multiplier with the
finite norm $M(t,j)$. Moreover for $t
> 0$
\begin{align}\label{multiplier2_2}
M(t,j) & \leq c e^{-\frac14 t2^{2 j}}\sum_{0 \leq i \leq n} t^i
2^{2 ij} \leq c e^{-\frac18 t2^{2 j}}.
\end{align}
\end{lemma}
\begin{proof}
The $L^\infty(\R^n)$-multiplier norms $M(t,j)$ of $\rho_{tj}(\xi) $ is
equal to the $L^\infty(\R^n)$-multiplier norm of $\rho_{tj}^{'}(\xi)  := 2^j \xi
\Psi^{'}(\xi) e^{-t2^{2 j} |\xi|^2}, \,\, \Psi^{'} (\xi) = \xi \Psi(\xi)$(See Theorem 6.1.3 in
\cite{BL}). Now, we  make use of the Lemma 6.1.5 of
\cite{BL}. Let   $\be=(\be_1, \be_2, \cdots, \be_n) \in   ({\mathbb N} \cup \{0\})^n$.
Since $supp\;(\Psi') \subset \{\xi \in \R^n \, | \,  \frac14 <  |\xi| < 4
\}$, we have
\begin{align*}
|D^\be_{\xi} \rho_{tj}^{'}(\xi)| &\leq c \sum_{0 \leq i \leq
|\be|} t^i 2^{2 ij} e^{-\frac14t 2^{2 j}} \chi_{\frac14 < |\xi|
< 4} (\xi) \\
&\leq c  e^{-\frac18 t 2^{2 j}} \chi_{\frac14 < |\xi|
< 4} (\xi),
\end{align*}
where $\chi_A$ is the characteristic function on a set $A$.  Hence, applying Lemma 6.1.5 of
\cite{BL},  we completes the proof.
\end{proof}

%\section*{Appendix}

\section*{Acknowledgements}
T. Chang is partially supported by NRF-2014R1A1A3A04049515 and K. Kang is partially
supported by NRF-2014R1A2A1A11051161 and
NRF20151009350. %and NRF-2012R1A1A2001373.

\begin{equation*}
\left.
\begin{array}{cc}
{\mbox{Tongkeun Chang}}\qquad&\qquad {\mbox{Kyungkeun Kang}}\\
{\mbox{Department of Mathematics }}\qquad&\qquad
 {\mbox{Department of Mathematics}} \\
{\mbox{Chosun University
}}\qquad&\qquad{\mbox{Yonsei University}}\\
{\mbox{ Kwangju 501-759, Republic of Korea}}\qquad&\qquad{\mbox{Seoul, Republic of Korea}}\\
{\mbox{chang7357@yonsei.ac.kr }}\qquad&\qquad
{\mbox{kkang@yonsei.ac.kr }}
\end{array}\right.
\end{equation*}

%\begin{equation*}
%\left.
%\begin{array}{cc}
%{\mbox{Myeongju Chae}}\qquad&\qquad {\mbox{Kyungkeun Kang}}\\
%{\mbox{Department of Applied Mathematics }}\qquad&\qquad
% {\mbox{Department of Mathematics}} \\
%{\mbox{Hankyong National University
%}}\qquad&\qquad{\mbox{Yonsei University}}\\
%{\mbox{Ansung, Republic of Korea}}\qquad&\qquad{\mbox{Seoul, Republic of Korea}}\\
%{\mbox{mchae@hknu.ac.kr }}\qquad&\qquad {\mbox{kkang@yonsei.ac.kr }}
%\end{array}\right.
%\end{equation*}
%\begin{equation*}
%\left.
%\begin{array}{c}
%{\mbox{Jihoon Lee}}\\
%{\mbox{Department of Mathematics }}\\
%{\mbox{Chung-Ang University}}\\
%{\mbox{Seoul, Republic of Korea}}\\
%{\mbox{jhleepde@cau.ac.kr }}
%\end{array}\right.
%\end{equation*}

\end{document}